\documentclass[a4paper]{amsart}

\usepackage[utf8]{inputenc}
\usepackage[T1]{fontenc}

\usepackage{lmodern}

\usepackage{amssymb}

\usepackage{subfig}

\usepackage{graphicx}

\usepackage{tikz}
\DeclareGraphicsExtensions{.pdf}

\graphicspath{{images/}}

\usepackage{xfrac}

\usepackage{hyperref}

\usepackage{cleveref} 

\usepackage{mathtools}

\usepackage{centernot}

\theoremstyle{plain}

\newtheorem{theorem}{Theorem}

\newtheorem{proposition}[theorem]{Proposition}
\newtheorem{corollary}[theorem]{Corollary}
\newtheorem{lemma}[theorem]{Lemma}
\newtheorem*{conjecture}{Conjecture}
\newtheorem*{assertion}{Assertion}

\theoremstyle{definition}
\newtheorem{definition}[theorem]{Definition}

\newtheorem*{remark*}{Remark}

\newcommand{\co}{\mathit{co}}
\newcommand{\cal}[1]{\mathcal{#1}}
\newcommand{\ov}[1]{\overline{#1}}

\newcommand{\C}{\mathbb{C}}
\newcommand{\D}{\mathbb{D}}
\newcommand{\N}{\mathbb{N}}
\newcommand{\Q}{\mathbb{Q}}
\newcommand{\R}{\mathbb{R}}
\newcommand{\U}{\mathbb{U}}
\newcommand{\Z}{\mathbb{Z}}

\newcommand{\Hyp}{\mathcal{H}}
\newcommand{\Con}{\mathcal{C}}

\newcommand{\cri}{\mathit{cr}}

\newcommand{\on}{\operatorname}
\newcommand{\ds}{\displaystyle}
\newcommand{\Luloc}{L^1_{\mathrm{loc}}}

\newcommand{\stend}{\xrightharpoonup{\ast}}

\renewcommand{\o}[1]{e^{i 2 \pi #1}}
\newcommand{\opq}{\o{\pq}}
\newcommand{\pq}{{\sfrac{p}{q}}}
\newcommand{\pqn}{\sfrac{p_n}{q_n}}

\newcommand{\maps}{\longrightarrow}
\newcommand{\tend}{\longrightarrow}

\renewcommand{\phi}{\varphi}
\renewcommand{\epsilon}{\varepsilon}
\newcommand{\eps}{\epsilon}
\renewcommand{\Im}{\operatorname{Im}}
\renewcommand{\Re}{\operatorname{Re}}

\renewcommand{\emptyset}{\varnothing}

\DeclareMathOperator{\id}{\operatorname*{id}}       

\DeclareMathOperator{\Supp}{\operatorname*{{Supp}}}       

\let\oldtocsection=\tocsection
\let\oldtocsubsection=\tocsubsection
\let\oldtocsubsubsection=\tocsubsubsection
\renewcommand{\tocsection}[2]{\hspace{0em}\oldtocsection{#1}{#2}}
\renewcommand{\tocsubsection}[2]{\hspace{1em}\oldtocsubsection{#1}{#2}}
\renewcommand{\tocsubsubsection}[2]{\hspace{2em}\oldtocsubsubsection{#1}{#2}}

\title[Size of Siegel disks for cubics]{On the size of Siegel disks with fixed multiplier for cubic polynomials}
\author{Arnaud Chéritat}

\begin{document}

\begin{abstract}
We study the slices of the parameter space of cubic polynomials where we fix the multiplier of a fixed point to some value $\lambda$.
The main object of interest here is the radius of convergence of the linearizing parametrization.
The opposite of its logarithm turns out to be a sub-harmonic function of the parameter whose Laplacian $\mu_\lambda$ is of particular interest.
We relate its support to the Zakeri curve in the case the multiplier is neutral with a bounded type irrational rotation number.
In the attracting case, we define and study an analogue of the Zakeri curve, using work of Petersen and Tan.
In the parabolic case, we define an analogue using the notion of asymptotic size.
We prove a convergence theorem of $\mu_{\lambda_n}$ to $\mu_\lambda$ for
$\lambda _n= \exp(2\pi i \pqn)$ and $\lambda = \exp(2\pi i\theta)$ where $\theta$ is a bounded type irrational and $\pqn$ are its convergents.
\end{abstract}

\maketitle

\tableofcontents

\section*{Acknowledgements}

Part of the proof of the convergence theorem in \Cref{sec:7} was developed in collaboration with Ilies Zidane when he was a PhD student. The current notes are based on part of a draft written by Ilies and expanded by the author.

\section*{Structure of the document}

The first section defines and studies the parameter spaces of cubic polynomials under several normalization related to different markings, and the relations between these different spaces.

\Cref{sec:phi} recalls generalities on the linearizing power series and maps, in the case of an attracting fixed point. If the dynamics is given by a polynomial, we define a special subset of the basin that we call $U(P)$ and that is the image of the disk of convergence of the linearizing parametrization. It is strictly contained in the basin of the fixed point and will play a role for attracting multipliers similar to the role played by Siegel disks for neutral multipliers.

\Cref{sec:3} proves generalities about the radius of convergence $r$ of the linearizing power series, with a focus on its dependence on the polynomial. We prove that, if $\lambda$ is fixed, $-\log r$ is a subharmonic function of the remaining parameter. The measure $\mu_\lambda = \Delta -\log r$ is introduced. We prove that its total mass is $2\pi$.

\Cref{sec:attr} studies attracting slices when both critical points are marked. We deduce from the work of Petersen and Tan that the set of parameters for which both critical points are attracted to the fixed point is an annulus. We define the set $Z_\lambda$ for which both critical points are on $\partial U(P)$. We prove that the support of $\Delta -\log r$ is equal to $Z_\lambda$.

\Cref{sec:5} studies similar slices but in the case when $\theta$ is a bounded type number. Zakeri proved that the set of parameters for which both critical points are on the boundary of the Siegel disk is a Jordan curve $Z_\lambda$. We prove that the support of $\mu_\lambda$ is $Z_\lambda$.

\Cref{sec:parabo} is about parabolic slices, i.e.\ $\lambda$ is a root of unity. We use the asymptotic size $L$ of parabolic points as an analogue of the conformal radius of Siegel disk. 
We prove that $-\log L$ is a subharmonic function of the parameter and that its Laplacian is a sum of Dirac masses situated at parameters for which the fixed point is degenerate, i.e.\ has too many petals.

\Cref{sec:7} proves the convergence of $\mu_{\lambda_n}$ to $\mu_\lambda$ in the weak-$\ast$ topology for
$\lambda _n= \exp(2\pi i \pqn)$ and $\lambda = \exp(2\pi i\theta)$ where $\theta$ is a bounded type irrational and $\pqn$ are its convergents.

\section{Normalisations}\label{sec:normz}

Seminal works in the study of cubic polynomials include \cite{art:BH1,art:BH2,art:Zakeri}. We assume here that the reader is familiar with holomorphic dynamics.

We will consider conjugacy classes of cubic polynomials with or without marked points. The conjugacies will be by \emph{affine maps}, i.e.\ maps of the form $z\in\C \mapsto az+b$ with $a\in\C^*$ and $b\in \C$ and will have to respect the markings.

We will consider three different markings and their relations.

A priori the quotient spaces are just sets. However since we quotient an analytic manifold by analytic relations, there is more structure. This is not the object here to develop a general theory of such quotients. In our case there will be families of representatives $P_{a,b}$, defined for complex numbers $(a,b)$ varying in an open subset of $\C^2$, whose $4$ coefficients vary holomorphically with $(a,b)$, whose marked points vary holomorphically too, such all equivalence classes are represented, and such that either equivalence classes have only one representative, or at most two. The equivalence relation is still analytic in $(a,b)$ and the quotient is a priori only an orbifold. These orbifold turn out to still be an analytic manifolds in our particular three cases.

The choice of our three families are called here \emph{normalizations}. A more general notion of normalization can certainly be developped but this is not the object of the present article.

\subsection{First family}\label{sub:norm_1st}

We first consider the family of cubic polynomials with one non-critical fixed point marked and both critical points marked, up to affine conjugacy. We choose the following representative: the point $0$ is fixed with multiplier $\lambda \in \C^*$. The critical points are $1$ and $c \in \C^\ast$, the second one is taken as a parameter. It follows from an easy (and classical) computation that
\[ P_{\lambda,c}(z) = \lambda z \left( 1 - \frac{(1 + \sfrac{1}{c})}{2} z + \frac{\sfrac{1}{c}}{3} z^2 \right) \]
and that any marked polynomial is uniquely represented (see \cite{art:Zakeri}).

When $\lambda$ is a root of unity: $\lambda = \opq$, for simplicity we will sometimes denote by $P_{\pq,c}$ the polynomial $P_{\lambda,c}$.

\begin{proposition}
Some easy remarks:
\begin{itemize}
\item Switching the role of the critical points $1$ and $c$ is equivalent to replacing $c$ by $1/c$. We have the symmetry
\begin{equation}\label{eq:sym}
c^{-1}P_{\lambda,c}(cz) = P_{\lambda,1/c} (z).
\end{equation}
\item The map has a double critical point if and only if $c=1$.
\item When $c \to \infty$, $P_{\lambda,c}$ converges uniformly on compact subset of $\C$ to a quadratic polynomial witch fixes $0$ with multiplier $\lambda$, namely to $Q_\lambda(z) := \lambda z \left( 1 - \frac{z}{2} \right)$.
\end{itemize}
\end{proposition}

\subsection{Second family}\label{sub:norm_2nd}

We consider the set of cubic polynomials with one non-critical fixed point marked but the critical points are not marked. It amounts to identifying $c$ and $1/c$. Since $c\neq 0$, because we assume that the fixed point is not critical, we can take
\[ v:=\frac{c+c^{-1}}{2} \]
as a parameter. The map $c\mapsto v$ is a ramified cover from $\C^*$ to $\C$, ramified at $c=1$ and $c=-1$ which are mapped respectively to $v=1$ and $v=-1$.\footnote{These are the same values but we are reluctant to call them fixed points.} The first case $v=1$ corresponds to maps with a double critical point, as noticed above. The second case $v=-1$ corresponds to the maps in the family $P_{\lambda,c}$ that commute with $z\mapsto -z$:
\[ P_{\lambda,-1}(z) = \lambda (z- \frac{1}{3}z^3)\]
This map is conjugated to $z\mapsto \lambda z + z^3$, which is a form that may or may not be more familiar to the reader.
Finally, note that $v=0$ is in the range of the map $c\mapsto v$: it corresponds to $c=\pm i$. However we do not have a special dynamical interpretation for this value of $c$.

\subsection{Third family}\label{sub:norm_3rd}

We consider the set of unmarked cubic polynomials up to affine conjugacy. It can be parametrized by the set of monic centred cubic polynomials up to conjugacy by $z\mapsto -z$, i.e.\ 
\[ 
P(z)= z^3+az+b
\]
with $(a,b)\sim (a,-b)$. The pair of parameters $(a,b^2)$ then gives a bijection from this family with $\C^2$.
We will sometimes denote $b_2 = b^2$ so that
\[(a,b_2) = (a,b^2)\] 
The map $P$ is unicritical if and only if $a=0$. The map $P$ commutes with a non-trivial affine map if and only if it commutes with $z\mapsto -z$, if and only if $b=0$.

\subsection{Topological aspect}

We have identified sets of affine conjugacy equivalence classes of polynomials (marked on not) with open subsets $U\subset\C^2$. We now interpret the sequential convergence for the topology of $\C^2$ in terms of those classes. 

\begin{proposition}
For any of the three families, a sequence of (marked) polynomials classes $[P_n]$ converges in $U$ to the (marked) polynomial class $[P]$ if and only if there exists representatives $Q_n\sim P_n$ and $Q\sim P$ such that $Q_n$ tends to $Q$ as degree $3$ polynomials (i.e.\ coefficient by coefficient; there are $4$ such coefficients) and such that each marked point of $Q_n$ converges to the corresponding marked point of $Q$.
\end{proposition}

\begin{proof}
The ``only if'' direction ($\Rightarrow$) of the equivalence is the easiest. We have families parametrized analytically, hence continuously, by either $(\lambda,c)$ or $(a,b)$. For one choice of the parameter $(\lambda,v)$ there is at most two values of $(\lambda,c)$ that realize it. Similarly, for one choice of the parameter $(a,b^2)$ there is at most two values of $(a,b)$ that realize it. Moreover, given a converging sequence of parameters, one can choose the realization so that it converges too. Last, the marked points depend continuously on the corresponding parameters: $0$ does not move at all and $c\mapsto c$ is trivially continuous\ldots

For the ``if'' direction ($\Leftarrow$), we treat first the $(a,b^2)$-parameter space.
The affine conjugacies from any cubic polynomial $Q=q_3 z^3 + q_2 z^2 + q_1 z + q_0$ to a monic centred polynomial correspond to the change of variable $z=\alpha w +\beta$ such that $\alpha^2 q_3 = 1$ and that send the centre $-q_2/3q_3$ to $0$, a condition that, once $\alpha$ is fixed, determines a unique $\beta$, more precisely $\beta = \alpha q_2/3q_3$.
Then $Q$ is conjugated by this affine map to $w\mapsto w^3 + aw+b$ with $a = q_1-q_2^2/3q_3$ and $b = \frac{2q_2^3/q_3^2+9(1-q_1)q_2/q_3+27q_0}{27\alpha}$. In particular, since $\alpha^2 = 1/q_3$, $(a,b^2)$ is a holomorphic, hence continuous, function of the coefficients $(q_0,\ldots,q_3)\in \C^3\times\C^*$. It follows that given $Q_n\tend Q$, the corresponding conjugate maps will have values of $(a_n,b_n^2)$ that converge to $(a,b^2)$. 

Now we go for the $(\lambda,c)$ and $(\lambda,v)$ spaces. Obviously if $Q_n\tend Q$ and the marked fixed point of $Q_n$ converges then $\lambda_n\tend \lambda$. Moreover, the set of critical points of $Q_n$ converges, by Rouché's theorem, to that of $Q$ (there is no drop of degree that would allow one critical point to tend to infinity). Now if the set of critical points of $Q$ is ${c_0,c_1}$ and the fixed point of $Q$ is $p$ then parameter $v$ is given by $\frac12\left(\frac{c_0-p}{c_1-p}+\frac{c_1-p}{c_0-p}\right)$, which implies the convergence for the $v$-parameter.
If moreover critical point are marked then $c= \frac{c_0-p}{c_1-p}$. This implies the convergence of $c$ if the convergence also concerns critical marked points.
\end{proof}

A more satisfying approach to the topology would be to consider the quotient topologies but we prefer to stay at more basic level.

\subsection{Correspondence between the normalizations}

There is a map associated to forgetting markings, from the first family to the second, and a similar map from the second to the third: they induce maps
\[ (\lambda,c) \mapsto (\lambda, v) \mapsto (a,b_2) \]
where $\lambda$, $c$, $v$, $a$, and $b_2 = b^2$ were introduced in the preceding sections.
Recall that
\[ v=\frac{c+c^{-1}}{2}
\]
A computation gives
\begin{align}
\label{eq:lvab1} a &=\frac{\lambda(1-v)}{2}
\\
\label{eq:lvab2} b_2 &= \frac{\lambda}{3}\cdot\frac{v+1}{2}\cdot\left(1+(v-2)\frac{\lambda}{3}\right)^2
\end{align}

Let us denote
\[
\Theta : 
\left\{
\begin{array}{rcl} 
\C^*\times\C&\to& \C^2
\\
(\lambda,v)&\mapsto& (a,b_2)
\end{array}
\right.
\]
and stress that even though $\Theta$ has a polynomial expression, we only consider its \emph{restriction} to non-vanishing values of $\lambda$.

Of particular interest is the set of \emph{unicritical polynomials}, which as we have seen in $(\lambda,c)$-space correspond exactly to those polynomials for which $c=1$, 
in $(\lambda,v)$-space to those for which $v=1$
and in $(a,b_2)$-space those for which $a=0$. We have $\Theta(\lambda,1) = (0,\frac{\lambda}{3}\left(\frac{\lambda}{3}-1\right)^2)$ i.e.\
\[
(\lambda, c = 1) \mapsto (\lambda, v=1) \mapsto (a = 0, b_2 = \frac{\lambda}{3}\left(\frac{\lambda}{3}-1\right)^2)
\]

Every cubic polynomial in the third family can be marked in at most $3$ different ways as an element of the second family, since there is at most three fixed points: hence the preimage of an element of $\C^2$ by $\Theta$ has at most $3$ elements. Also, a cubic polynomial cannot have all its fixed points critical, in particular $\Theta$ is surjective.
The following lemma sums up a precise analysis.

\begin{lemma}
Let $P$ an affine conjugacy class of cubic polynomials be represented by $(a,b_2)=(a,b^2)\in\C^2$.
The fibre $\Theta^{-1}(P)$ (in $(\lambda,v)$-space) has $3$ elements unless one of the following occurs:
\begin{enumerate}
\item $(a,b^2)=(0,0)$ i.e.\ $P$ is conjugate to $z^3$; then the fibre has $1$ element $(\lambda,v)=(3,1)$;
\item $(a,b^2)=(1,0)$ i.e.\ $P$ is conjugate to $z+z^3$ which has a triple fixed point; then the fibre has $1$ element $(\lambda,v)=(1,-1)$;
\item $(a,b^2)=(3/2,0)$, then $P$ has a symmetry and both critical points are fixed; the fibre has $1$ element $(\lambda,v)=(3/2,-1)$;
\item $b=0$ and $a\notin\{0,1,3/2\}$, then $P$ has a symmetry and the fibre has $2$ elements;
\item $(a,b^2)=(4/3,-4/3^6)$, then $P$ has a double fixed point and a critical fixed point; the fibre has $1$ element $(\lambda,v) = (1,-5/3)$;
\item $P$ has a double fixed point and another fixed point that is not critical; then the fibre has $2$ elements.
\end{enumerate}
\end{lemma}
\begin{proof}
(We only indicate the method of the proof, the details of the computations are not relevant.)
The fibre has less than three element if and only if either less than three different fixed points can be marked, or there exists an affine self-conjugacy sending two different marked points one to the other.
The first case occurs if and only if two fixed points coincide or if a fixed point is critical.
In the second case we must have $b=0$.
\end{proof}

The map from the first to the second family is much simpler. We recall it takes the form $(\lambda,c)\mapsto(\lambda,v=\frac{c+c^{-1}}{2})$. The map $c\mapsto v$ is well-known:\footnote{The map $z\mapsto z+z^{-1}$ it is nowadays often referred to as the \emph{Joukowsky transform}.} it is surjective from $\C^*$ to $\C$ and fibres have two elements unless $v=1$ or $v=-1$.

\medskip

Every polynomial in the second family corresponds to at most $2$ polynomials in the first because there is at most two critical points. Hence a cubic polynomial has at most $6$ representatives in $(\lambda,c)$-space. The map $(\lambda,c)\mapsto(a,b_2)$ is surjective because it is the composition of the surjective maps $\Theta$ and $(\lambda,c)\mapsto (\lambda,v)$.

\medskip

Recall that a mapping is open when it maps open subsets to open subsets.\footnote{Recall that a characterization of continuous functions is that the \emph{preimage} of an open subset is open.}
The following theorem can be found in \cite{book:cha}, Corollary page~328, Section~54.
\begin{theorem}[Osgoode]
Let $n\geq 1$ and consider a holomorphic map $f$ from an open subset $U\subset \C^n$ to $\C^n$.
If all fibres of $f$ are discrete then $f$ is open.
\end{theorem}

It follows that the map $\Theta$ is an open mapping. The simpler map $(\lambda,c)\mapsto (\lambda,\frac {c+c^{-1}}{2})$ is open too.

\begin{definition}\label{def:E}
Let $\cal E$ be the set of affine conjugacy classes of polynomials which have a fixed critical point.
\end{definition}

The set $\cal E$ is characterized by the equation
\[b_2 + \frac{a}{3}\left(1-\frac{2a}{3}\right)^2= 0\]
Its preimage in $(\lambda,c)$-space has equation $3-6\lambda^{-1} \in\{ c,c^{-1}\}$.
Recall that $\lambda=0$ is not part of this space.
Its preimage in $(\lambda,v)$-space has equation $\frac{3\lambda-6}\lambda + \frac\lambda{3\lambda-6} = 2v$ with $\lambda\neq 2$.

\begin{proposition}\label{prop:proper}
The map $\Theta$ is proper over $\C^2\setminus {\cal E}$.
\end{proposition}
\begin{proof}
Consider a sequence of marked polynomials such that their unmarked equivalence classes $(a_n,b_n^2)$ converge in $\C^2\setminus {\cal E}$. Let $(a,b^2)$ be the limit class and $P=z^3+az+b$. We can extract a subsequence such that $a_n$ and $b_n$ converge. Replacing $b$ by $-b$ if necessary, we have $a=\lim a_n$ and $b=\lim b_n$. Let $P_n(z) = z^3+a_n z+b_n$. The fixed points of $P_n$ remain in a bounded subset of $\C$. We can extract a subsequence so that the marked fixed point converges. The limit will then be a fixed point of $P$. Since we assumed that $(a,b^2)$ is not in $\cal E$, it follows that we have a valid marking for $P$. The eigenvalue $\lambda_n$ of the fixed point obviously converges to that of $P$ and by Rouché's theorem the pair of critical points converge as a compact subset of $\C$ to that of $P$. 

By the sequential characterization of compact subsets of metric spaces, it follows that $\Theta$ is proper from $\Theta^{-1}(\C^2\setminus {\cal E})$ to $\C^2\setminus {\cal E}$.
\end{proof}

\begin{remark*}
Two remarks
\begin{itemize}
\item If we had decided that $\Theta$ is defined on all of $\C^2$ by the same polynomial formulae as in \eqref{eq:lvab1} and \eqref{eq:lvab2}, then $\Theta$ would be proper over all the target set $\C^2$.
\item But we defined $\Theta$ on $\C^*\times\C$ and it cannot be proper over a neighborhood of a point $[P]\in\cal E$. Indeed we can find nearby polynomials $P_n\tend P$ with $[P_n]\notin \cal E$ and which have a fixed point of very small eigenvalue, and if we mark this fixed point, the corresponding marked polynomial will not converge (they will converge to a polynomial with $\lambda = 0$ which we decided is outside of the domain of $\Theta$).
\end{itemize}
\end{remark*}

\subsection{Some special subsets of parameter space}\label{sub:spec}

Recall that a polynomial is hyperbolic when all critigcal points lie in basins of attracting periodic points and that this condition is stable by perturbation: the set of parameters for which the polynomials are hyperbolic forms an open subset in the parameter space.
By a \emph{hyperbolic component} for a given parametrization is meant a connected component of this subset.

\begin{definition}\label{def:clmc}
	The \emph{connectivity locus} of the cubic polynomials denoted by $\Con_3$ is the set of
    cubic polynomials up to affine conjugacy which have a connected Julia set.\\
   	The \emph{principal hyperbolic component}, $\Hyp_0$ is the hyperbolic component containing the affine conjugacy class of the polynomial $P(z) = z^3$.    
\end{definition}

We recall a classical theorem\footnote{Whose attribution is not easy to determine.}
\begin{theorem}
    The affine class of a polynomial belongs to $\Hyp_0$ if and only if both critical points belong to the immediate basin of an attracting fixed point. Its Julia set is a quasicircle.
\end{theorem}
\begin{proof}
  Let us consider a path from $z^3$ to $P$ whose class remains in $\Hyp_0$.
  By the Mañé-Sad-Sullivan theory, the Julia set follows a holomorphic motion while the class of $P$ remains in $\Hyp_0$. One consequence is that it is a quasicicle. The other is that the critical points cannot jump out of the immediate basin.
  
  The converse follows from a theorem of Milnor in \cite{art:mil}: every hyperbolic component has a centre, i.e.\ a map that is post critically finite. For the same reason as in the previous paragraph, this other polynomial map has both critical points in the immediate basin of an attracting fixed point $a$. We will prove below that the two critical points coincide with $a$. Since there is only one affine conjugacy class of polynomial with a fixed double critical point (the class of $P(z)=z^3$), the result will follow.
  
  We will use the following topological lemma: given a \emph{connected and non empty} ramified covering $U\to D$ over a topological disk $D$ and a point $b\in D$, if the set of ramification values is contained in $\{b\}$, then $b$ has a unique preimage and $U$ is a topological disk.
  
  Since $a$ is attracting (possibly superattracting), there is a disk $D=B(a,\epsilon)$ small enough so that every point in $D\setminus\{a\}$ have infinite orbit and $P(D)\subset D$. In particular, since $P$ is post critically finite, the first time a critical orbit enters this disk must be by hitting $a$ directly.
  By the topological lemma, it follows by induction that there exists a sequence $D_n$ of topological disks containing $a$ such that $D_0=D$, $D_{n+1}$ is the connected component of $P^{-1}(D_n)$ containing $a$ and $a$ is the only preimage of $a$ in $D_{n+1}$. We will also use the fact, proved by an easy induction, that $D_n$ is a connected component of $P^{-n}(D)$. From $P(D)\subset D$ it follows that $D_{n} \subset D_{n+1}$.
  The basin $B$ of $a$ is equal to the union $U$ of all $D_n$: indeed $U$ is open and the complement of $U$ in $B$ is open (see the next paragraph), hence empty by connectedness of $B$. Since $a$ is the only element of $P^{-1}(a)$ in $D_{n+1}$ it follows that $a$ is the only element of $P^{-1}(a)$ in $B$. We have seen that the critical points, which both belong to $B$ by hypothesis, eventually map to $a$. Hence they are both equal to $a$, Q.E.D.
  
  Let us justify the claim made in the previous paragraph: let $z\in B\setminus U$. Consider $n$ such that $P^n(z)\in D$. Then there is an open disk $V$ of centre $z$ and such that $P^n(V)\subset D$, and thus for all $k\geq 0$, $P^{n+k}(V)\subset D$. Hence the connected set $V$ is contained in a connected component of $P^{-(n+k)}(D)$. It follows that $V \cap D_{n+k} = \emptyset$ for otherwise $V$ would be contained in $D_{n+k}$ hence $z\in D_k$, contradicting $z\notin U$. Hence $V \cap U = \emptyset$. 
\end{proof}

\begin{lemma}\label{lem:top1}
In metric spaces, if a function is continuous open and proper from a non-empty set to a connected set then it is surjective.
\end{lemma}
\begin{proof}
The image is everything because it is non-empty open and closed. The last point is proved by taking a convergent sequence $f(x_n)$ and the compact set $\{f(x_n)\}_{n\in\N}\cup\{\lim f(x_n)\}$. Its preimage being compact we can extract a subsequence of $x_n$ so that $x_n$ converges. Then by continuity $\lim f(x_n)=f(\lim x_n)$.
\end{proof}

The sets $\Hyp_0$, $\Con_3$ and $\cal E$ live in the $(a,b_2)$-space (see \Cref{def:E,def:clmc}: they are respectively the principal hyperbolic domain, the connectivity locus and the maps with a fixed critical point). We denote by $\Con_3'$ the preimage of $\Con_3$ in the $(\lambda,c)$-space. It is not relevant here whether $\Con_3'$ is connected or not.
\begin{proposition}\label{prop:0ra}
The preimage of $\Hyp_0$ in $(\lambda,v)$-space has two connected components. One contains $(\lambda=3,v=1)$ and is contained in $``|\lambda|>1"$ (call it $\Hyp^r$) and the other one contains $(|\lambda|<1, v=1)$ and is contained in $``|\lambda|<1"$ (call it $\Hyp^a$). The map from $\Hyp^a$ to $\Hyp_0\setminus\cal E$ is a homeomorphism.
\end{proposition}
\begin{proof}

The polynomial $(\lambda=3,v=1)$ corresponds to $P(z)=z^3$ and is hence in the preimage of $\Hyp_0$ and satisfies $|\lambda|>1$.
Let us denote $\Hyp^r$ the connected component of the preimage of $\Hyp_0$ that contains it.
The polynomials $(|\lambda|<1, v=1)$ have a unique critical point an it must be in the basin of the attracting fixed point by Fatou's theorem. 
Let us denote $\Hyp^a$ the connected component of the preimage of $\Hyp_0$ that contains this connected subset.

Consider the open subset $\Hyp_0$ of $\C^2$.
Its preimage by $\Theta$ is an open subset of $\C^*\times \C$, its connected components are hence open subsets of $\C^*\times \C$. The image by $(\lambda,v)\mapsto(a,b_2)$ of these components are open because $\Theta$ is an open mapping.

By classical theorems of Fatou, a map in $\Hyp_0$ cannot have a neutral cycle.
Hence the preimage of $\Hyp_0$ cannot meet $|\lambda|=1$: any connected component must be contained either in $|\lambda|<1$ or $|\lambda|>1$.

The map $\Theta$ is injective on the intersection of $|\lambda|<1$ with the preimage of $\Hyp_0$, because there is only one attracting point to mark, and its image is $\Hyp_0\setminus{\cal E}$: the attracting fixed point cannot be marked if and only if it is critical.
By the \emph{invariance of domain} theorem, it follows that $\Theta$ is a homeomorphism from 
$``|\lambda|<1" \cap\Theta^{-1}(\Hyp_0)$ to $\Hyp_0\setminus{\cal E}$. (Alternatively we could have used the fact that $\Theta$ is proper over $\C^2\setminus{\cal E}$---hence over $\Hyp_0\setminus{\cal E}$---because in metric spaces, a proper continuous bijective map is necessarily a homeomorphism.) In particular $``|\lambda|<1" \cap\Theta^{-1}(\Hyp_0)$ is connected and coincides with the set $\Hyp^a$ defined at the beginning of this proof.

Let us prove that $\Theta$ is proper from $``|\lambda|>1"\cap \Theta^{-1}(\Hyp_0)$ to $\Hyp_0$. This has already been proved over $\Hyp_0\setminus{\cal E}$, see \Cref{prop:proper}. The extension to all of $\Hyp_0$ essentially follows from the facts that in the marked point $z=0$ cannot have a multiplier that tend to $0$ since we are in $``|\lambda|>1"$, and that there is a uniformly bounded number of fixed points. Here is a detailed proof:
Let us assume that a sequence of polynomials $P_n$ in $\Hyp_0$ with a marked repelling fixed point is such that the affine conjugacy class of $P_n$, unmarked, converge to some polynomial $P$ in $\Hyp_0$.
Recall that the two repelling fixed points depend holomorphically on polynomials near $P$: there is a neighborhood $V$ of $P$ and two holomorphic functions $\xi_1$, $\xi_2$ from $V$ to $\C$ such that the repelling fixed points of any $Q\in V$ are $\xi_1(Q)$ and $\xi_2(Q)$.
The marked point may be any of these two fixed points and may occasionally jump from one to the other as $n$ varies. But we can extract a subsequence so that the marked point is always $\xi_1(P_n)$ or $\xi_2(P_n)$. It then converges to $\xi_i(P)$ for $i=1$ or $2$.
By the sequential characterization of compact subsets of metric spaces, this proves the claim.

By \Cref{lem:top1}, any component of ``$|\lambda|>1\cap \Theta^{-1}(\Hyp_0)$'' surjects to $\Hyp_0$. Since $z^3$ has only one preimage in  ``$|\lambda|>1\cap \Theta^{-1}(\Hyp_0)$'', it follows that  ``$|\lambda|>1\cap \Theta^{-1}(\Hyp_0)$'' is connected.
\end{proof}

\begin{proposition}
Each of the two components $\Hyp^a$, $\Hyp^r$ in the previous proposition, which sit in $(\lambda,v)$-space, has a preimage in $(\lambda,c)$-space that is connected. The first one contains $(\lambda=3,c=1)$ and is contained in ``$|\lambda|>1$''and the other one contains $(|\lambda|<1, c=1)$ and is contained in ``$|\lambda|<1$''. The map $(\lambda,c)\mapsto (\lambda, v=\frac{c+c^{-1}}{2})$ is a two-to-one covering over these sets minus $``v=1"\cup``v=-1''$.
\end{proposition}
\begin{proof}
The claim on covering properties follows from $c\mapsto v$ being a $2:1$ covering from $\hat \C\setminus\{-1,1\}$ to itself. Moreover this map is proper from $\C^*$ to $\C$, and thus so is too the map $(\lambda,c)\mapsto(\lambda,v)$ from $(\C^*)^2$ to $\C^*\times \C$.
The claim on connectedness then follows from the fact that both preimages contain a point with $c=1$: consider for instance a component $A'$ of the preimage of $\Hyp^a$. Since the map is proper, it is proper from $A'$ to $\Hyp^a$. By \Cref{lem:top1} the restriction $A'\to \Hyp^a$ is surjective. But since $(\lambda,v=1)$ has only one preimage, $(\lambda, c=1)$, there can be only one such component.
\end{proof}
Let us define the following subset of $(\lambda,c)$-space:
\[
\Hyp_0' = \left\{(\lambda,c)\,;\, |\lambda|<1,\ \text{both critical points of  $P_{\lambda,c}$ lie in the immediate basin of $0$}\right\}
\]
Note that $\Hyp_0'$ contains $(\lambda,c=-1)$ for any $\lambda\in \D$. Indeed, the polynomial $P_{\lambda,-1}$ commutes with $z\mapsto -z$, which swaps both critical points. Since at least one critical point is in the immediate basin, which contains $0$ hence is invariant too by $-z$, it follows that both critical points are in the immediate basin. 
Recall that $\cal E$ denotes the subset of polynomials with a fixed critical point in the set $\C^2$ of unmarked classes.
By the analysis above:
\begin{itemize}
\item $\Hyp_0'$ contains $(|\lambda|<1,c=-1)$ and $(|\lambda|<1,c=1)$.
\item $\Hyp_0'$ is connected 
and the map $(\lambda,c)\mapsto (a,b_2)$ sends it to $\Hyp_0\setminus{\cal E}$ as a $2:1$ ramified cover.
\item More precisely the map $(\lambda,c)\mapsto (a,b_2)$ is a
\begin{itemize}
\item $1:1$ homeomorphism from $``c=1"\cap \Hyp_0'$ to $a=0 \cap (\Hyp_0\setminus{\cal E})$,
\item $1:1$ homeomorphism from $``c=-1"\cap \Hyp_0'$ to $b=0\cap (\Hyp_0\setminus{\cal E})$,
\item $2:1$ covering from $\Hyp_0'\setminus ``c\in\{-1,1\}"$ to $(\Hyp_0\setminus{\cal E})\setminus ``a=0, \text{ or } b=0"$.
\end{itemize}
\end{itemize}

We prove the following lemma here, for future reference in this document.

\begin{lemma}\label{lem:cesc}
  For each $\lambda\in\C^*$ there exists $\rho>0$ such that if $|c|>\rho$ then $c$ belongs to the basin of infinity for $P_{\lambda,c}$.
  For $|c|<\rho^{-1}$, then $1$ belongs to the basin of infinity for $P_{\lambda,c}$. 
\end{lemma}
\begin{proof}
  Given a polynomial $f(z) = a_1 z + \cdots + a_3 z^3$, a trap in the basin of infinity is given by $|z|>R$ with
  \[ R=\max(\sqrt{\frac{2}{|a_3|}}, \frac{4|a_2|}{|a_3|}, \sqrt{\frac{4|a_1|}{|a_3|}}).\]
  The first lower bound on $R$ ensures that the term $a_3 z^3$ has modulus $>2|z|$. The other two that the rest has modulus $<\frac{1}{2}|a_3 z^3|$.
  Given the formula of $P_{\lambda,c}$, which we recall here:
  \[ P_{\lambda,c}(z) = \lambda z \left( 1 - \frac{(1 + \sfrac{1}{c})}{2} z + \frac{\sfrac{1}{c}}{3} z^2 \right) \]
  this yields
  \[ R = \max(\sqrt{\frac{6|c|}{|\lambda|}} , 6|c+1| , \sqrt{12 |c|})\] 
  The point $c$ is not in the trap for $c$ big, but its first iterate is:
 indeed one computes
  \[P_{\lambda,c}(c) = \lambda c \frac{3-c}{6}.\]
  This proves the first claim.
  The second claim follows from the symmetry relation $c^{-1}P_{\lambda,c}(cz) = P_{\lambda,1/c} (z)$.
\end{proof}

\section{The linearizing power series}\label{sec:phi}

Given a map $f(z) = \lambda z + \cdots$, with $\lambda\neq 0$, we define by abuse of notations the function
\[
\lambda(z)= \lambda z
\]
There are two variants of the linearizing maps: it is either a map $\phi(z) = z +\cdots$ such that
\[
\phi \circ f = \lambda \circ \phi
\]
holds near $0$, or a map $\psi(z) = z + \cdots$ such that
\[
f\circ \psi = \psi \circ \lambda
\]
holds near $0$.
Douady used to call the first one a \emph{linearizing coordinate}, and the second one a \emph{linearizing parametrization}. We shall use both variants and, by imitation of Douady's conventions for Fatou coordinates, we use the symbol $\phi$ for the first and $\psi$ for the second.

We recall\footnote{without proof, this is very classical} that if $\lambda$ is not a root of unity, then as formal power series normalized by the condition to be of the form $z + {\cal O}(z^2)$ and satisfying the above equations, $\psi$ and $\phi$ exist and are unique. (If $\lambda$ is a root of unity then $\phi$ and $\psi$ may or may not exist, and that if they exist, then they are not unique.)

We recall without proof the following classical facts, valid for all $f$.

\begin{proposition}
  Assume that $\lambda\in\C^*$ is not a root of unity.
  Let $r\in[0,+\infty]$ denote the radius of convergence of $\psi$ and $r'$ the radius of convergence of $\phi$.
  \begin{itemize}
    \item $r>0$ iff $r'>0$
    \item $r>0$ iff the origin is linearizable
    \item if $|\lambda|\neq 1$ then the origin is always linearizable
    \item if $|\lambda|<1$ then a holomorphic extended linearizing coordinate $\phi$ with $\phi'(0)=1$ is defined on the basin of attraction of $f$; it satisfies $\phi\circ f = \lambda \circ \phi$ on the basin but is not necessarily injective if $f$ is not injective; $\phi(z) = \lim f^n(z)/\lambda^n$, which converges locally uniformly on the basin
    \item if $|\lambda|=1$ and $f$ is linearizable then a holomorphic linearizing coordinate $\phi$ with $\phi'(0)=1$ is defined on the Siegel disk of $f$, is necessarily injective thereon and maps it to a Euclidean disk
    \item in the last two cases above, the power series expansion at the origin of the holomorphic map $\phi$ coincides with the formal linearizing power series $\phi$ introduced earlier
  \end{itemize}
\end{proposition}

We will be mainly interested in
$r(\lambda,c) := r(P_{\lambda,c})$ the radius of convergence of the power series $\psi$ associated to $f = P_{\lambda,c}$.
If context makes it clear, we will use the shorter notations $r(c)$ or even $r$.
Even though $\psi$ originally designates a formal power series, we will use the same symbol $\psi$ to also denote the holomorphic function defined on its disk of convergence $B(0,r)$ by the sum of the power series.

\medskip

We recall the following facts, specific to polynomials. Let $P$ be a polynomial of degree at least $2$ with $P(z) = \lambda z + {\cal O}(z^2)$ near $0$ and $\lambda\in\C^*$ which is not a root of unity:

\begin{itemize}
  \item If $|\lambda|\leq 1$ then the sum of the power series $\psi$ on its disk of convergence\footnote{Which may be empty\ldots\ in which case the statements gives no information.} defines an injective function.
  \item If $|\lambda|=1$ and $P$ is linearizable then $r'$ is equal to the distance from $0$ to the boundary of its Siegel disk and $r$ is equal to the conformal radius of the Siegel disk w.r.t.\ $0$.
  \item If $|\lambda|<1$, then $r'$ is equal to the distance from $0$ to the boundary of its attracting basin, and $r$ is equal to the conformal radius w.r.t.\ $0$ of the special subset $U$ defined below in \Cref{prop:U}.
\end{itemize}

\noindent{\bf Note:} Concerning the first point (which also trivially holds if $f$ is linear): if $\lambda$ has modulus one but is not a root of unity, there is a nice proof in \cite{book:Milnor}. And in the case $|\lambda|<1$ it follows from the following argument: the map $\psi$ satisfies $\psi (\lambda z) = f (\psi(z))$ on its disk of convergence. So if $\psi(x)=\psi(y)$ then $\psi(\lambda^n x) = \psi(\lambda^n y)$. Since $\psi'(0)=1$ the map $\psi$ is injective near $0$, so $\lambda^n x=\lambda^n y$ for $n$ big enough. Hence $x=y$.

\medskip

We recall the following classic fact, essentially a consequence of the maximum principle.
\begin{lemma}[folk.]\label{lem:folk}
 if $V\subset \C$ is a bounded Jordan domain and $P$ a non-constant polynomial then all the connected components $U$ of $P^{-1}(V)$ are (bounded) Jordan domains and $P:\partial U\to \partial V$ is a covering whose degree coincides with the degree of the proper map $P:U\to V$.
\end{lemma}

We recall a classical result, the holomorphic dependence of $\phi_P$ on $P$:
\begin{lemma}[folk.]\label{lem:holodep}
Let $U$ be a complex manifold (parameter set).
Consider any analytic family of polynomials $\zeta\in U\mapsto P_\zeta$ (not necessarily of constant degree), all fixing the origin with the same multiplier $\lambda$ with $0<|\lambda|<1$.
Let $\cal B$ denote the fibred union of basins of $0$: $\cal B = \{(\zeta,z)\,;\,z\in B(P_\zeta)\}$. Then $\cal B$ is an open subset of $U\times \C$ and
the map $(\zeta,z)\mapsto \phi_{P_\zeta}(z)$ is analytic.
\end{lemma}
\begin{proof}
Openness follows from the existence of a stable trap near $0$, for small perturbation of the parameter $\zeta$.
Analyticity follows from the local uniform convergence of the following formula:
\[\phi_f(z)=\lim_{n\to\infty} \frac{ f^n(z)}{\lambda^n}\]
\end{proof}

\begin{proposition}\label{prop:U}
  Assume that $P$ is a degree $\geq 2$ polynomial fixing $0$ with attracting multiplier $\lambda\neq 0$. Then
  the map $\psi$ is injective on $B(0,r)$. The set
  \[U:=\psi(B(0,r))\]
 is compactly contained in the basin of $P$. It is a Jordan domain and $P$ is injective on its boundary.
  There is no critical point of $f$ in $U$ and there is at least one critical point of $f$ on $\partial U$.
\end{proposition}
\begin{proof}
Since $P(\psi(z)) = \psi(\lambda z)$ is true at the level of power series, it is true on all $B(0,r)$. It follows that $P^n(\psi(z)) = \psi(\lambda^n z)\tend \psi (0) = 0$ as $n\to+\infty$. Thus the image of $\psi$ is contained in the basin of attraction of $0$ for $P$. Hence $\phi$ is defined on the image of $\psi$. Since $\phi$ has a power series expansion at the origin that is the inverse of that of $\phi$ it follows that $\phi \circ \psi$ is the identity near $0$, and hence on $B(0,r)$ by analytic continuation. In particular $\psi$ is injective on $B(0,r)$. 
The absence of critical point of $P$ in $U$ follows easily from this and $P\circ\psi = \psi\circ \lambda$.
The set $U$ is the preimage by $P$ of $P(U) = \psi(B(0,|\lambda|r))$. Since the latter is compactly contained in $U$, hence in the basin, it follows that all components of $P^{-1}(P(U))$ are compactly contained in their respective Fatou components. In particular $U$ is compactly contained in the immediate basin.
Also, from $P(U) = \psi(B(0,|\lambda|r))$ it follows that $P(U)$ is a Jordan domain (with analytic boundary) and hence $U$, like every connected component of $P^{-1}(P(U))$, must be a Jordan domain by the first part of \Cref{lem:folk}. Moreover, by the second part of the lemma, since $P$ is injective on $U$ it follows that $P$ is injective on $\partial U$.
For the last point we proceed by contradiction.
Consider the sets $U(\rho)=\psi(B(0,\rho))$ and let $U'(\rho)$ be the connected component containing $0$ of $P^{-1}(U(\rho))$.
Then $U(r)=U = U'(|\lambda|r)$.
Assume by way of contradiction that $P$ has no critical point on $\partial U$. Then $U$ sits at a positive distance from the other components of $P^{-1}(P(U))$.
It follows that given a neighborhood $V$ of $\overline U$, then for $\epsilon>0$ small enough and $\rho=|\lambda|r+\epsilon$ we have $U'(\rho)\subset V$ and $U'(\rho)$ does not contain critical points of $P$. Since $P$ is is proper and without critical points from $U'(\rho)$ to $U(\rho)$, it is a cover, and since the image is a topological disk, it is injective. But then we get a contradiction with the definition of $r$: indeed one could extend $\psi$ to $B(0,\rho/|\lambda|)$ by letting $\psi(z)$ be the unique point of $P^{-1}(\psi(\lambda z))$ in $U'(\rho)$.
\end{proof}
Since $\psi'(0)=1$ it follows that $r$ is the conformal radius of $U:=\psi(B(0,r))$ w.r.t.\ the origin. Moreover, we have
$r=|\phi(z)|$ where $z$ is any point on $\partial U$. This will be particularly useful when we take $z$ to be (one of) the critical point(s) on $\partial U$ when we study how $r$ depends on the polynomial, so we number this equation for future reference:
\begin{equation}\label{eq:rpc}
\forall \text{ critical point } c\in\partial U,\ r=|\phi(c)|.
\end{equation}

If only one critical point is on $\partial U$ we call it the \emph{main critical point}. For each $\lambda\in\D^*$ there are values of $c$ such that there is more that one critical point on $\partial U$ for $P_{\lambda,c}$: for instance this is the case for $c=-1$, for which we recall that the polynomial commutes with $z\mapsto -z$.

Two remarks:
\begin{enumerate}
\item
Morally the main critical point, if there is one, is the closest to the attracting fixed point. It is not the closest for the Euclidean distance but it is indeed for some other notion of distance defined using $\phi$. However we will not need this here.
\item
It is important to realize that the main critical point does not necessarily have the least value of $|\phi|$ among all critical points: sometimes there is another critical point, possibly in the immediate basin, that maps under some iterate $P^k$ to a point that is ``closer'' to $0$ that the same iterate $P^k$ applied to the main critical point. It may even happen with $k=1$.
\end{enumerate}

\begin{lemma}\label{lem:scsP}
Let $P_n$ be a sequence of polynomials fixing $0$ with multiplier $\lambda_n\in \ov \D \setminus\{0\}$ not a root of unity and assume that $P_n$ tends to a polynomial $P$ of degree at least $2$ uniformly on compact subsets of $\C$. \begin{itemize}
\item Then
\[r(P)\geq \limsup r(P_n)\]
where $r(P)$ is defined to be $0$ if the the fixed point $0$ of $P$ is parabolic or superattracting.
\item Moreover denoting $r_0=\liminf r(P_n)$, the sequence $\psi_{P_n}$ converges to $\psi_{P}$ on every compact subset of $B(0,r_0)$.
\end{itemize}
\end{lemma}
\begin{proof}
Let $\psi_n=\psi_{P_n}$.
The identity $P_n \circ \psi_n(z) = \psi_n (\lambda_n z)$ holds on $B(0,r(P_n))$. 

Let $\tilde r=\limsup r(P_n)$. If $\tilde r=0$ there is nothing to prove, so we assume $\tilde r>0$.
To prove the first claim of the lemma, let us extract a subsequence such that $r(P_n)$ converges to $\tilde r$.
The maps $\psi_n$ being univalent and normalized by $\psi_n(0)=0$ and $\psi_n'(0)=1$, they form a normal sequence on $B(0,\tilde r)$.\footnote{Usually a normal \emph{family} is defined for maps defined on a common set of definition. By a normal sequence on $B(0,\tilde r)$ we mean the following: the domain eventually contains every compact subset of $B(0,\tilde r)$ and any subsequence has a subsubsequence that converges uniformly on compact subsets of $B(0,\tilde r)$.} 
By continuity, any extracted limit $\ell$ must satisfy $P \circ \ell (z)= \ell (\lambda z)$ for $z\in B(0,\tilde r)$, $\ell(0)=0$ and $\ell'(0)=1$.
In particular $P$ is linearizable, hence $P'(0)$ cannot be $0$ nor a root of unity.
Hence $\ell$ must have the same power series expansion as $\psi_P$, so the limit is unique. Moreover, the radius of convergence of $\psi_P$ is at least $\tilde r$.

Let $\hat r=\liminf r(P_n)$. 
The proof of the second claim is very similar, but this time we do not extract subsequences. The family $\psi_n$ is normal on $B(0,\hat r)$. Any extracted limit must linearize, hence this limit is unique and coincides with the restriction of $\psi_P$ to $B(0,\hat r)$.
\end{proof}

\begin{lemma}\label{lem:psiext}
The function $\psi$ extends as a homeromorphism $\ov\psi : \ov B(0,r) \to \ov U$.
\end{lemma}
\begin{proof}
In \Cref{prop:U} we saw that $P$ is injective on $\partial U$. Hence $P$ is injective on $\ov U$ and since $\ov U$ is compact, $P$ is a homeomorphism from $\ov U$ to its image. Let $g$ be the inverse homeomorphism and let $\ov{\psi}(z) = g\circ \psi (\lambda z)$ for $z\in \ov B(0,r)$. Then $\ov\psi$ is a continuous extension.
\end{proof}

\begin{lemma}\label{lem:cP}
Let $P_n$ be a sequence of polynomials of degree $\geq 2$, fixing $0$ with multiplier $\lambda_n \in \D^*$ and assume $P_n$ tends to a polynomial $P$ of degree at least $2$ uniformly on compact subsets of $\C$. Assume moreover that $\lambda = P'(0) \in \D^*$. Then 
\begin{enumerate}
\item $\phi_{P_n}\tend \phi_P$ uniformly on compact subsets of the basin of $0$ for $P$,
\item $r(P_n) \tend r(P)$,
\item for all sequence $z_n\in \ov B(0,r(P_n))$, such that $z_n$ converges to some $z_\infty\in\C$, then $\ov\psi_{P_n}(z_n) \tend \ov\psi_P(z_\infty)$,
\item if $c_n\in\partial U(P_n)$ is a critical point such that $c_n$ converges, then its limit $c_\infty$ is a critical point of $P$ that belongs to $\partial U(P)$.
\end{enumerate}
\end{lemma}
\begin{proof}
The first point follows from the local uniform convergence\footnote{The proof is classical so we omit it here.} of the following formula:
\[\phi_f(z)=\lim_{n\to\infty} \frac{ f^n(z)}{\lambda^n}.\]

For the second point, we have seen in \Cref{lem:scsP} that $\limsup r(P_n)\leq r(P)$, so there remains to prove that $\liminf r(P_n) \geq r(P)$.
The map $P$ is injective on $U(P)$, so by a variant of Hurwitz's theorem, for all compact subset $K$ of $U(P)$, for $n$ big enough the map $\phi_{P_n}$ is injective on $K$. If we take $K=\psi_P(B(0,r(P)-\eps/2))$ we get that for $n$ big enough, the image of the restriction $\phi_{P_n}|_K$ contains $B'=B(0,r(P)-\eps)$ and since it is injective, its reciprocal is defined on $B'$. As a formal power series, this reciprocal coincides with $\psi_{P_n}$ by uniqueness of the linearizing formal power series, and thus $r(P_n)\geq r(P)-\eps$.

Let us prove the third point.
The case where $|z_\infty|<r(P)$ is already covered by the last point of \Cref{lem:scsP}, so we assume that $|z_\infty|=r(P)$.
Denote $w_\infty = \ov\psi_P(z_\infty)$ and $w_n = \ov\psi_{P_n}(z_n)$.
Then $w_\infty\in \partial U(P)$ and the objective is to prove that $w_n\tend w_\infty$.

Let us first treat the case where $w_\infty$ is not a critical point of $\phi_P$.
Then there exists $\eps$ such that for $n$ big enough, $\phi_{P_n}$ has an inverse branch $h_n$ defined on $B(z_\infty,\eps)$ that converges uniformly to an inverse branch $h$ of $\phi_P$ that satisfies $h(z_\infty) = w_\infty$.
Note that $h(z') = \psi_P(z')$ for all $z'$ in the non-empty connected set $L = B(0,r(P)) \cap B(z_\infty,\eps)$. 
Then by the above and by by \Cref{lem:scsP}, $h_n$ and $\psi_{P_n}$ both tend to $h$ uniformly on compact subsets of $L$.
Since $h$ is non-constant, we in particular have that for all ball $B$ compactly contained in $L$, then for $n$ big enough: $\psi_{P_n}(B) \cap h_n(B)\neq \emptyset$.
But $\psi_{P_n}$ and $h_n$ are both inverse branches of the same map $\phi_{P_n}$.
It follows that they coincide on the connected component $W_n$ of the intersection of their domain that contains $B$.
The set $W_n = B(z_\infty,\eps) \cap B(0,r(P_n))$ eventually contains $z_n\tend z\infty$ and we have $\psi_{P_n}(z_n) = h_n(z_n) \tend h(z_\infty) = \psi_P(z_\infty)$.

We now treat the case where $w_\infty$ is a critical point of $\phi_P$.
Note that $\phi_P$ has only finitely many critical points on $\partial U$.
For $\eps>0$ let $K = \ov B(0,r(P)) \cap \ov B(z_\infty,\eps)$ and choose $\eps$ small enough so that in $\ov\psi_P(K)$ the point $w_\infty$ is the only critical point of $\phi_P$  and the only preimage of $z_\infty$ by $\phi_P$.
Let us proceed by contradiction and assume that $w_n = \psi_{P_n}(z_n) \centernot\tend w_\infty$.
Then either $w_n$ leaves every compact of the basin of $0$ for $P$, or it has a subsequence that converges to some $w'$ in this basin, and then by the first point of the current lemma, $\phi_P(w') = z_\infty$.
In both cases $w_n$ is eventually out of some fixed neighborhood $V$ of $\ov\psi_P(K)$. Since $K$ is connected, this means there exists another sequence $z'_n$ with $w'_n:=\psi_{P_n}(z'_n)$ that satisfies: $w'_n \tend w' \notin V \cup \phi_P^{-1}(z_\infty)$.
By the first point of the current lemma, $z'_n = \phi_{P_n}(w'_n) \tend z':= \phi_P(w')$. Since $K$ is closed, we have $z'\in K$.
By definition $\phi_P(w')\neq z_\infty$, i.e.\ $z'\neq z_\infty$.
Let us prove that $w'=\ov\psi_P(z')$:
if $|z|<r(P)$ then this is covered by the last point of \Cref{lem:scsP};
if $|z'|=r(P)$, since $z'\neq z_\infty$ then by the choice of $K$, the point $\ov\psi_P(z')$ is not a critical of $\phi_P$ and by the analysis in the previous paragraph, $w'=\ov\psi_P(z')$. So $w' \in \ov\psi_P(K)$, which contradicts $w'\notin V$.

Last, we prove the fourth point. By passing to the limit in $P_n(c_n)=0$ we get that $c_\infty$ is a critical point of $P$. Let $z_n = (\ov\psi_{P_n})^{-1}(c_n)$. By the first point of \Cref{lem:scsP}, $z_n$ is a bounded sequence and any extracted limit $z'$ satisfies $|z'|\leq r(P)$. By the third point of the present lemma, for the extracted sequence we have $c_n \ov \psi_{P_n}(z_n) \tend \ov \psi_P(z')$, hence $\ov\psi_P(z') = c_\infty$.  
Hence $c_\infty$ belongs to $\ov \psi_P(\ov B(0,r(P)) = U(P)$. Since it is critical it cannot belong $U(P)$ so $c_\infty \in \partial U(P)$.
\end{proof}

\subsection{Applications to our family of cubic polynomials}

Let $\U_\Q$ denote the set of roots of unity. If $\lambda\in\U_\Q$ then the linearizing power series is not defined. We set
\[r(P_{\lambda,c})=0\]
in this case. This is a natural choice because we know in advance that $P_{\lambda,c}$ is not linearizable.\footnote{No rationally indifferent periodic point of a degree $\geq 2$ rational map can be linearizable.}

The following lemma is a direct application of \Cref{lem:scsP}:
\begin{lemma}
The map $(\lambda,c)\mapsto r(\lambda,c)$ restricted to values of $\lambda \in \ov{\D}^*$ is upper semi-continuous.
\end{lemma}

\section{About the radius of convergence of the linearizing parametrization}\label{sec:3}

Given $f(z)=\lambda z +\cal O(z^2)$ with $\lambda \neq 0$ nor equal to a root of unity, we noted $\psi$ the formal power series solution of $\psi = z + \cal O(z^2)$ and $f\circ \psi = \psi \circ \lambda$ where by abuse of notation $\lambda$ denotes the function $z\mapsto \lambda z$. Let us write $\psi = \psi_f$ to highlight the dependence on $f$ and let
\[f = \sum a_n z^n\]
\[\psi_f = \sum b_n z^n\]
be the respective (formal) power series expansions. We will sometimes write $b_n(f)$ to emphasize the dependence on $f$.

Below we will denote, for a given formal power series $s$ in $z$, its $z^n$ coefficient  by $[s]_n$. We recall here a few well-know facts:
\begin{itemize}
\item (Cauchy-Hadamard formula) The radius of convergence $r$ of $\psi_f$ is given by \[\frac{1}{r}=\limsup |b_n|^{1/n}\]
\item $b_n$ is uniquely determined by the strong recursion formula
\[\lambda^n b_n = \lambda b_n + \sum_{k=2}^{n} a_n [\psi^k(z)]_n\]
where $\psi^k$ stands for the multiplicative $k$-th power (not the $k$-th iterate).
Note that the sum starts with $k=2$ and that $[\psi^k(z)]_n$ only depends on the coefficients $b_1,\ldots,b_{n-k+1}$ (here $b_1=1$).
\end{itemize}
Let us apply this to the family $P_{\lambda,c}$ with a fixed $\lambda$. We recall the definition:
\[ P_{\lambda,c}(z) = \lambda z \left( 1 - \frac{(1 + \sfrac{1}{c})}{2} z + \frac{\sfrac{1}{c}}{3} z^2 \right) \]
We thus have $a_1 = \lambda$, $a_2 = -\lambda\frac{1+c^{-1}}{2}$, $a_3 = \lambda c^{-1}/3$ and all other $a_n$ are equal to $0$. It follows that
\begin{lemma}
  For a fixed $\lambda$ that is not a root of unity, nor $0$, the coefficient $b_n(P_{\lambda,c})$ is a polynomial in $c^{-1}$, of degree at most $n-1$.
\end{lemma}
\begin{proof}
This is proved by induction on $n$. By definition $b_1=1$.
Then $b_n = (\lambda^n-\lambda)^{-1}(a_2\sum_{i+j=n} b_ib_j + a_3\sum_{i+j+k=n}b_ib_jb_k)$.
If $n\geq 2$ and the claim holds up to $n-1$ then the term involving $a_3$ has degree at most $1+(i-1)+(j-1)+(k-1) = n-2$ and the term involving $a_2$ at most $1+(i-1)+(j-1)=n-1$.
\end{proof}

\begin{lemma}\label{lem:cont}
Let $r(P_{\lambda,c})$ denote the radius of convergence of $\psi_{P_{\lambda,c}}$.
If either $0<|\lambda| <1$ or $\theta\in\R$ is a Brjuno number and $\lambda=e^{2\pi i\theta}$, 
then
\[c\mapsto r(P_{\lambda,c})\]
is a continuous function of $c\in\C^*$.
\end{lemma}
\begin{proof}
In the Brjuno case, it is a direct application of a theorem in \cite{thesis:Cheritat}: the proposition on page 79, Chapter 3.

In the attracting case, by \cref{eq:rpc} we have $r=|\phi_P(c_P)|$ where $c_P$ is one of the two critical points of $P=P_{\lambda,c}$.\footnote{So $c_P = c$ or $c_P =1$, depending on the value of $c$. It is \emph{not} true that $c_P$ depends continuously on $P$, not even locally (when the two critical points belong to $\partial U$ but are distinct, the point $c_P$ may jump from one to the other for nearby parameters. And it will: it follows from the analysis that we will make later of the curve $Z_\lambda$, see \Cref{sec:attr}.}
We saw in \Cref{lem:holodep} that $\phi_P$ depends continuously (holomorphically!) on $P$.
The claim then follows from the fourth point of \Cref{lem:cP}.
\end{proof}

By the symmetry $c^{-1}P_{\lambda,c}(cz) = P_{\lambda,1/c} (z)$ valid for all $c\in\C^*$ and $z\in\C$, we get that 
\begin{equation}\label{eq:symr}
  -\log r(P_{\lambda,c}) + \log |c|= -\log r(P_{\lambda,1/c}).
\end{equation}

We refer to \cite{Ra} for the definition of a subharmonic function (definition~2.2.1 page~28): it is a function from some open subset $U$ of $\C$ to $[-\infty,+\infty)$ that is upper semi-continuous and satisfies the \emph{local submean inequality}.
\begin{proposition}\label{prop:sh}
Under the same assumptions,
\[c\mapsto -\log r(P_{\lambda,c})\]
is a subharmonic function of $c\in\C^*$.
\end{proposition}
\begin{proof}
First it is upper semi-continuous: indeed it is continuous according to the previous lemma.
We then have to check the local submean inequality.
Recall that
\[-\log r = \limsup \frac1n\log b_n\]
So $-\log r$ is the decreasing limit of the functions $u_n = \sup_{k\geq n} \frac1k\log b_k$.
We can't claim that $u_n$ is subharmonic because we do not know if it is upper semi-continuous.
However we will still check that $u_n$, then $-\log r$, satisfy the submean inequality on any disk $B(0,\rho)$ compactly contained in $\C^*$.
Consider such a disk. Then the continuous function $r$ reaches a minimum $r_0>0$ on its closure. Recall that the power series $\psi_r$ defines an \emph{injective} function on its domain of convergence.
It follows then from the Bieberbach-De Branges theorem\footnote{Most previously known bounds on the coefficients of univalent function are easier to prove and are also sufficient for this purpose.} that $|b_n|\leq n/r_0^n$. In particular : $\frac1n\log b_n$ is a sequence of functions on $B(0,\rho)$ that is bounded from above by some constant $M\in \R$.
The functions $u_n$ are hence bounded from above by the same $M$.
Each function $\frac1k\log b_k$ being subharmonic, it satisfies the submean inequality on any disk contained in $\C$. It easily follows that $u_n$ does too. 
Since this weakly decreasing\footnote{By this we mean $u_{n+1}\leq u_n$.} sequence is bounded from above, the monotone convergence theorem holds and its limit thus satisfies the submean inequality.

As an alternative proof, one can use the extension of the Brelot-Cartan theorem called Theorem~3.4.3 in \cite{Ra}. The limsup $u$ satisfies $u^*=u$ since $u$ is continuous. We saw in the previous paragraph that the sequence of functions is locally uniformly bounded from above. So the hypotheses of the theorem are satisfied.
\end{proof}

As $c\tend\infty$, the map  $P_{\lambda,c}$ converges uniformly on every compact subset of $\C$ to $Q_\lambda:z\mapsto \lambda z (1-\frac{z}{2})$.
\begin{lemma}\label{lem:qlr}
If $|\lambda|\leq 1$ and
\[|c|\geq \frac{7^2}{7^2/2-7-10}\left(\frac 12 + \frac{7}{3|\lambda|}\right)\]
then
$P_{\lambda,c}$ has a quadratic-like restriction $P_{\lambda,c}: V\mapsto B(0,10/|\lambda|)$ where $V$ is the connected component containing $0$ of $P^{-1}(B(0,10/|\lambda|))$. The critical point of this quadratic-like restriction is $z=1$.
\end{lemma}
\begin{proof}
We first change variable with $z= \lambda^{-1} w$.
Then $P_{\lambda,c}$ is conjugated to
\[F(w) = \lambda w -\frac{w^2}{2} + c^{-1} w^2 \left(\frac{-1}{2} +\frac{w}{3\lambda}\right)\]
The condition above ensures that the rightmost term has modulus $\leq 7^2/2-7-10$ whenever $|w|\leq 7$. When $|w|=7$ we have
\[|\lambda w -\frac{w^2}{2}| \geq 7^2/2-7 \]
and hence
\[|F(w)|\geq 10.\]
Moreover, by Rouché's theorem, $F(w)$ winds the same number of times around $0$ than $-\frac{w^2}{2}$ does, i.e.\ $2$ times.
Let $V$ be the preimage of $B(0,10)$ by $F$ restricted to $B(0,7)$.
Then $F$ is proper from $V$ to $B(0,10)$ and has degree $2$.
Either set $V$ has two connected components and $F$ has no critical point on $V$, or $F$ has a critical point on $V$ and $V$ has one component.
Now the point $w=\lambda$ is critical and
\[F(\lambda) = \lambda^2\left(\frac12-\frac{c^{-1}}{6} \right).\]
Note that 
\[|c|\geq \frac{7^2}{7^2/2-7-10}\left(\frac 12 + \frac{7}{3}\right) = \frac{833}{45}\]
hence
\[|F(\lambda)| \leq \frac{1}{2} + \frac{1}{6}\cdot\frac{45}{833} < 10\]
and of course $|\lambda|<7$.
Hence we are in the second case and $P:\lambda^{-1}V\to B(0,10/|\lambda|)$ is quadratic-like, with $\lambda^{-1}V \subset B(0,7/|\lambda|)$.
By hypothesis $|c| \geq \frac{7^2}{7^2/2-7-10}\left(\frac 12 + \frac{7}{3|\lambda|}\right)$, hence $|c| \geq \frac{7^2}{7^2/2-7-10}\cdot\frac{7}{3\lambda } > 7$, hence $z=c$ cannot be the critical point of the quadratic-like restriction, so this has to be $z=1$.
\end{proof}

\begin{remark*}
 In fact, if $|\lambda|\leq 1$, to ensure the existence of a quadratic-like restriction it is enough that one critical point escapes to infinity.
  We could then have used \Cref{lem:cesc}.
  One advantage of \Cref{lem:qlr} is that we have an explicit and simple range for the quadratic-like restriction.
\end{remark*}

According to \cite{Ra}, section 3.7, the Laplacian, in the sense of distributions, of a subharmonic function is represented by a Radon\footnote{The general definition of a Radon measure is elaborate, but on $\R^n$ it is just a (positive) measure on the Borel sets that is locally finite, i.e.\ finite on every compact set.
There is a correspondence between such Radon measures and positive linear operators on the set of continuous functions $\R^n \to \R$ with compact support.} measure, let us call it $\mu_\lambda$.

\begin{proposition}\label{prop:harmo}
Assume that we have an open subset $W$ of $\C$,
and a family of quadratic-like maps for $c\in W$, $f_c : U_c \to V_c$ that all satisfy $f_c(z) = \lambda z + \cal O(z^2)$. Assume that the fibred union $\bigcup_{c\in W} \{c\}\times U_c$ is open in $\C^2$ and that $f_c(z)$ varies analytically with $c$.
Then the function $c\mapsto -\log r(f_c)$ is harmonic on $W$.
\end{proposition}
\begin{proof}
  
Case 1: $|\lambda|<1$. Then the immediate basin of $0$ for $f_c$ is the basin of $0$ for its restriction. We thus have $r(f_c) = |\phi_{f_c}(1)|$, whence the harmonicity of $-\log(r)$ in $1/c$.

Case 2: $|\lambda|=1$ and $\theta$ is Brjuno. 
  The Julia set $J$ of the restriction undergoes a holomorphic motion as $c$ varies locally in $W$: its multiplier at the origin remains constant and non-repelling, so all other cycles remain repelling, hence undergo a holomorphic motion, so their closure $J$ too by the $\lambda$-lemma. We can then apply the analysis of Sullivan (see \cite{Zakeri2},  or \cite{BC1}, proposition~2.14): when a holomorphic family of maps with an indifferent fixed point $0$ has a Siegel disk whose boundary undergoes a holomorphic motion w.r.t.\ the parameter, then $-\log r$ is a harmonic function of $1/c$.
\end{proof}

\begin{proposition}\label{prop:p2}
Under the same assumptions as \Cref{lem:cont},
the function $f:c\in\C^*\mapsto -\log r(P_{\lambda,c})$ is harmonic near $0$ and near $\infty$.
The support of the measure $\mu_\lambda = \Delta f$ is bounded away from $0$ and $\infty$.
The function $f$ has a limit as $c\tend \infty$.
The function $f(1/c)$ has an harmonic extension near $0$ whose value
is $-\log r(Q_\lambda)$.
\end{proposition}
\begin{proof}
The second claim follows from the first.
The last claim follows from the third, but we will prove it directly.

From \cref{eq:symr} we can deduce harmonicity near $0$ from the harmonicity near $\infty$. We will simlutaneously prove the existence of a limit as $c\to\infty$.

These follow from the following remark, already done by Yoccoz in \cite{Yoccoz}.
By \Cref{lem:qlr} there is some $R>0$, depending on $\lambda$, such that $|c|>R$ implies that the following map is quadratic-like with critical point $z=1$: the restriction $P_{\lambda,c}: V\mapsto B(0,10/|\lambda|)$ where $V$ is the connected component containing $0$ of $P^{-1}(B(0,10/|\lambda|))$. Moreover, its domain converges as $c\tend\infty$ and the restriction depends holomorphically on $1/c$, including when $c=\infty$, as $P$ tends to $Q$ on every compact subset of $\C$.
We can then apply \Cref{prop:harmo}.
\end{proof}

We stress the following fact, that is kind of hidden in the previous statement:
\begin{equation}\label{eq:limr}
  r(P_{\lambda,c}) \underset{c\to\infty}{\tend} r(Q_\lambda).
\end{equation}

Note that the proof of \Cref{prop:p2} and the formula in the statement of \Cref{lem:qlr} gives that the support of $\mu_\lambda$ is contained in the closed ball of radius
\begin{equation}\label{eq:R}
R = \kappa_0 + \kappa_1/|\lambda|
\end{equation}
for some explicit positive constants $\kappa_0$, $\kappa_1$.

\medskip

From \Cref{prop:p2}, it follows that the measure $\mu_\lambda$ has a support that is a compact subset of $\C^*$. Since it is locally finite, it follows that it is finite.

\begin{proposition}\label{prop:totalmass}
The total mass of $\mu_\lambda$ is $2\pi$.
\end{proposition}
\begin{proof}
As $c\to \infty$, the map $P_{\lambda,c}$ converges on every compact subset of $\C$ to the polynomial $Q_\lambda(z)=\lambda z(1-\frac{z}2)$, whose unique critical point is $z=1$.
We saw that the function $-\log r$ has a limit
\[a=-\log r(Q_\lambda)\]
as $c\to\infty$.
Using the symmetry $c^{-1}P_{\lambda,c}(cz) = P_{\lambda,1/c} (z)$ valid for all $c\in\C^*$ and $z\in\C$, we get that 
\[-\log r(P_{\lambda,c}) + \log |c|= -\log r(P_{\lambda,1/c}).\]
Hence
\[-\log r(P_{\lambda,c}) \underset{c\to 0}= -\log |c| +a+ o(1).\]
Let
\[f(c) = \big\langle \mu_\lambda , \frac{1}{2\pi}\log |z| \big\rangle = \int_{z\in\C} \frac{1}{2\pi}\log|c-z| d\mu_\lambda(z).\]
Then $\Delta f = \mu_\lambda$ (\cite{Ra}, Theorem 3.7.4).
By Weyl's lemma, \cite{Ra} Theorem 3.7.10, two subharmonic functions with the same Laplacian differ by a harmonic function.
In particular $-\log r(P_{\lambda,c})+\log|c|-f(c)$ is a harmonic function of $c\in\C^*$ that has a limit when $c\to 0$ and is $= (1-\frac{\mathrm{mass}}{2\pi})\log|c|+o(\log|c|)$ when $c\to\infty$. Such a harmonic function is necessarily constant, hence the mass is $2\pi$.
\end{proof}

As a bonus, we get the representation formula below:
\begin{equation}\label{eq:logr}
-\log r(P_{\lambda,c}) = -\log r(Q_\lambda) - \log|c| + \int_{z\in\C} \frac{1}{2\pi}\log|c-z| d\mu_\lambda(z)
\end{equation}
where $Q_\lambda(z)=\lambda z(1-\frac{z}2)$.

\section{Attracting slices}\label{sec:attr}

Recall the set $\cal H_0$, which is the principal hyperbolic component in the family of unmarked affine conjugacy classes of cubic polynomials, and one of whose charaterization is that there is an attracting fixed point whose immediate basin contains all critical points, see \Cref{sub:spec}. We see $\cal H_0$ as a subset of $\C\times\C$ via the $(a,b^2)$ parameterization of such conjugacy classes, see \Cref{sub:norm_3rd}.

\begin{figure}[ht]
	\begin{center}
   \includegraphics[width=11.5cm]{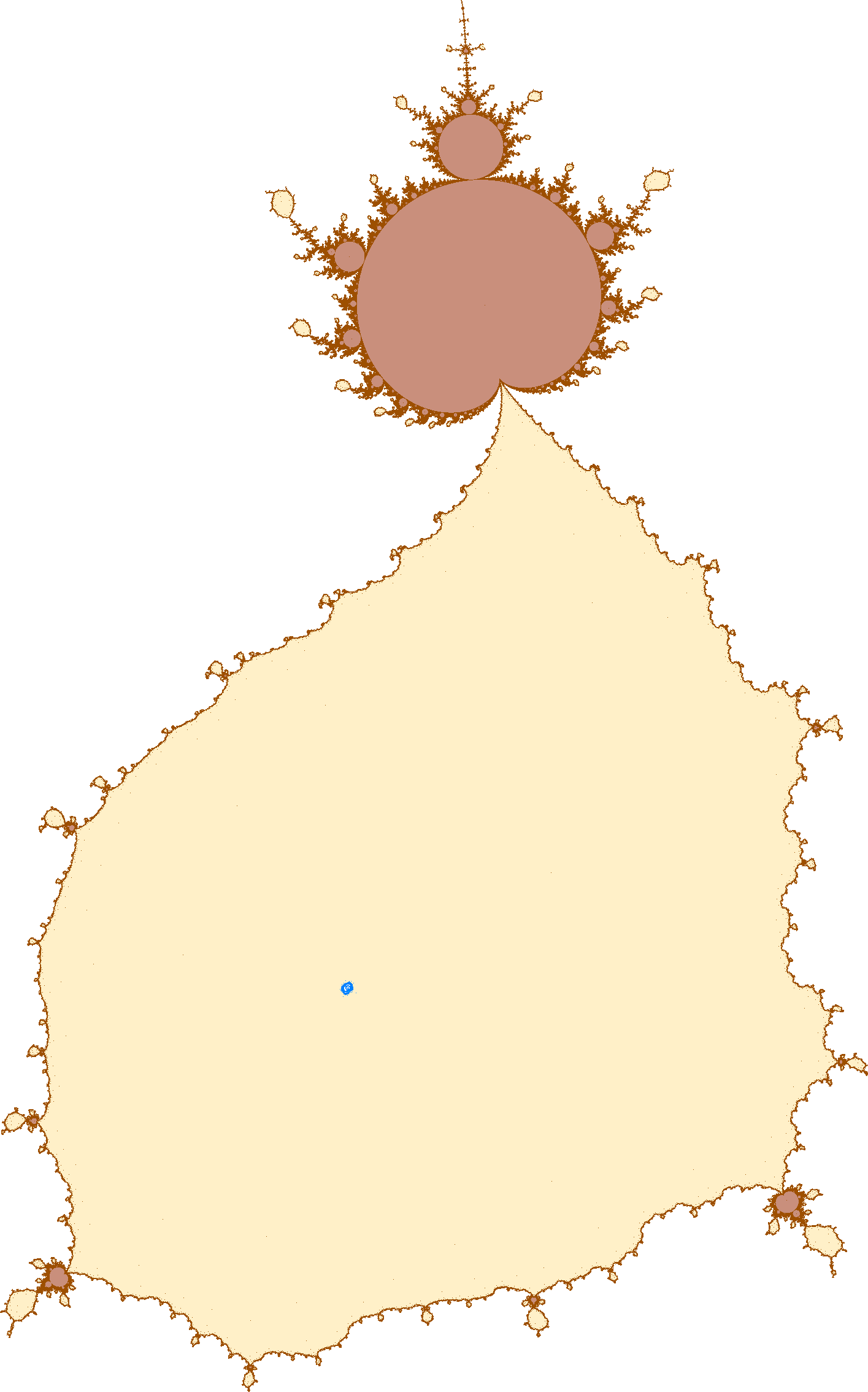}
   \end{center}
   \caption{The bifurcation locus of the $\lambda$-slice for $\lambda = 0.4i$, represented in the $c$-coordinate. The critical point $z=1$ bifurcates on the dark red set. The critical point $z=c$ bifurcates on the dark blue set, barely visible in the centre of the picture. This set is the image of the dark red one by the inversion $c\mapsto 1/c$. The isolated dark dots are artifacts of the method used to detect the bifurcation locus.}
   \label{fig:bifattrc}
\end{figure}

Petersen and Tan Lei described in \cite{PT} the fibres of the map: $\pi : \Hyp_0 \longrightarrow \D$ which associate to each polynomial the multiplier of the attracting fixed point. Denote
\[\Hyp_0(\lambda) =\pi^{-1}(\lambda)\]
\begin{theorem}[Petersen and Tan]
Let $\lambda \in \D^\ast$, then $\Hyp_0(\lambda)$ is a topological disk.
\label{thm:attr}
\end{theorem}

Recall the sets $\Hyp_0^a$ and $\Hyp_0'$ : they are the connected components contained in ``$|\lambda|<1$'' of the preimage of $\Hyp_0$ by respectively the maps $(\lambda,v)\mapsto (a,b^2)$ and $(\lambda,c)\mapsto (a,b^2)$ introduced in \Cref{sec:normz}.
For $\lambda\in\D^*$ let $\Hyp_0'(\lambda) \subset \C$ denote the $\lambda$-slice of $\Hyp_0'$, and $\Hyp_0^a(\lambda)$ be defined similarly, i.e.\ 
\begin{align*}
\Hyp_0^a(\lambda) & = \{v\in\C\,;\,(\lambda,v)\in\Hyp_0^a\}
\\
\Hyp_0'(\lambda) & = \{c\in\C^*\,;\,(\lambda,c)\in\Hyp_0'\}
\end{align*}
Let $F$ denote the rational map given by $F(z) = \frac{z+z^{-1}}{2}$. Then
\begin{equation}\label{eq:hap}
\Hyp_0'(\lambda) = F^{-1}\left(\Hyp_0^a(\lambda)\right)
\end{equation}

\begin{corollary}
The set $\Hyp_0^a(\lambda)$ is a topological disk.
The set $\Hyp_0'(\lambda)$ is a topological annulus.
\end{corollary}
\begin{proof}
Recall that ${\cal E}\subset \C^2$ denotes the unmarked polynomial classes that have a fixed critical point. This set is disjoint from $\Hyp_0$.
In the $(\lambda,v)$-coordinate, $\lambda$ is precisely the multiplier of the attracting cycle, so $\Hyp_0^a(\lambda)$ is exactly the preimage of $\pi^{-1}(\lambda)$ by $(\lambda,v)\mapsto (a,b^2)$.
The first statement is thus immediate since $(\lambda,v)\mapsto(a,b^2)$ is a homeomorphism from $\Hyp_0^a$ to $\Hyp_0$ by \Cref{prop:0ra}.

The second statement then follows from topological properties of $F$ and the fact that $c=-1$ and $c=1$ both belong to $\Hyp_0'(\lambda)$. Indeed $P_{\lambda,c=-1}$ is conjugate to $\lambda z+z^3$ which commutes with $z\mapsto -z$; since one critical point $z_0$ is in the immediate basin $0$, the (distinct) critical point $-z_0$ is too; concerning the polynomial $P_{\lambda,c=1}$, it is unicritical and hence all critical points are in the immediate basin. The preimage by $F$ of a topological disk $D$ containing $-1$ and $1$ is connected (because it is a ramified cover over $D$ and $1$ has only one preimage) and we conclude using the Riemann Hurwitz formula.
\end{proof}

In \Cref{sec:phi} we introduced the quantity $r=r(P_{\lambda,c})$, which is the conformal radius of a special subset $U=U(P_{\lambda,c})$ of the basin of attraction of $0$ for $P_{\lambda,c}$. We recall that set $U$ contains $0$, its boundary contains at least one critical point but $U$ contains none and the linearizing coordinate $\phi(z)=\phi_{\lambda,c}(z) = z+\cdots$ is a bijection from $U$ to the round disk $B(0,r)$.
Let
\[
Z_\lambda = \{c\in\C^*\,;\,\text{both critical points of $P_{\lambda,c}$ belong to }\partial U\}.
\]
We call this a Z-curve, chosen after the name of Zakeri, who defined a similar set in the bounded type indifferent case instead of the attracting case, and proved, using quasiconformal surgery, that his set is a Jordan curve. Another name for this set could have been the Petersen-Tan set, as it is the pull-back of the seam that they define (see \Cref{sub:PT}).
In this section we will prove the following:
\begin{theorem}\label{thm:ll}
The set $Z_\lambda$ is a Jordan curve. Let $c\in\C^*$. If $c$ lies in the bounded component of $\C\setminus Z_\lambda$ then the unique critical point of $P_{\lambda,c}$ that belongs to $\partial U$ is $c$, and it is $1$ if $c$ belongs to the unbounded component.
For each fixed $\lambda\in\D^*$, the function $c\in\C^*\mapsto -\log r(P_{\lambda,c})$ is subharmonic and continuous. It is harmonic on $\C^* \setminus Z_\lambda$.
Its Laplacian has total mass $2\pi$ and its support is equal to $Z_\lambda$.
\end{theorem}
Recall that $c$ and $1$ are the two critical points of $P_{\lambda,c}$ so on $Z_\lambda$ we have $|\phi(c)|=|\phi(1)|$, but the converse does not hold.
We will see that $Z_\lambda$ is naturally parametrized by $\arg\left(\phi(c)/\phi(1)\right)$.

\begin{figure}[ht]
	\begin{center}
   \includegraphics[width=12cm]{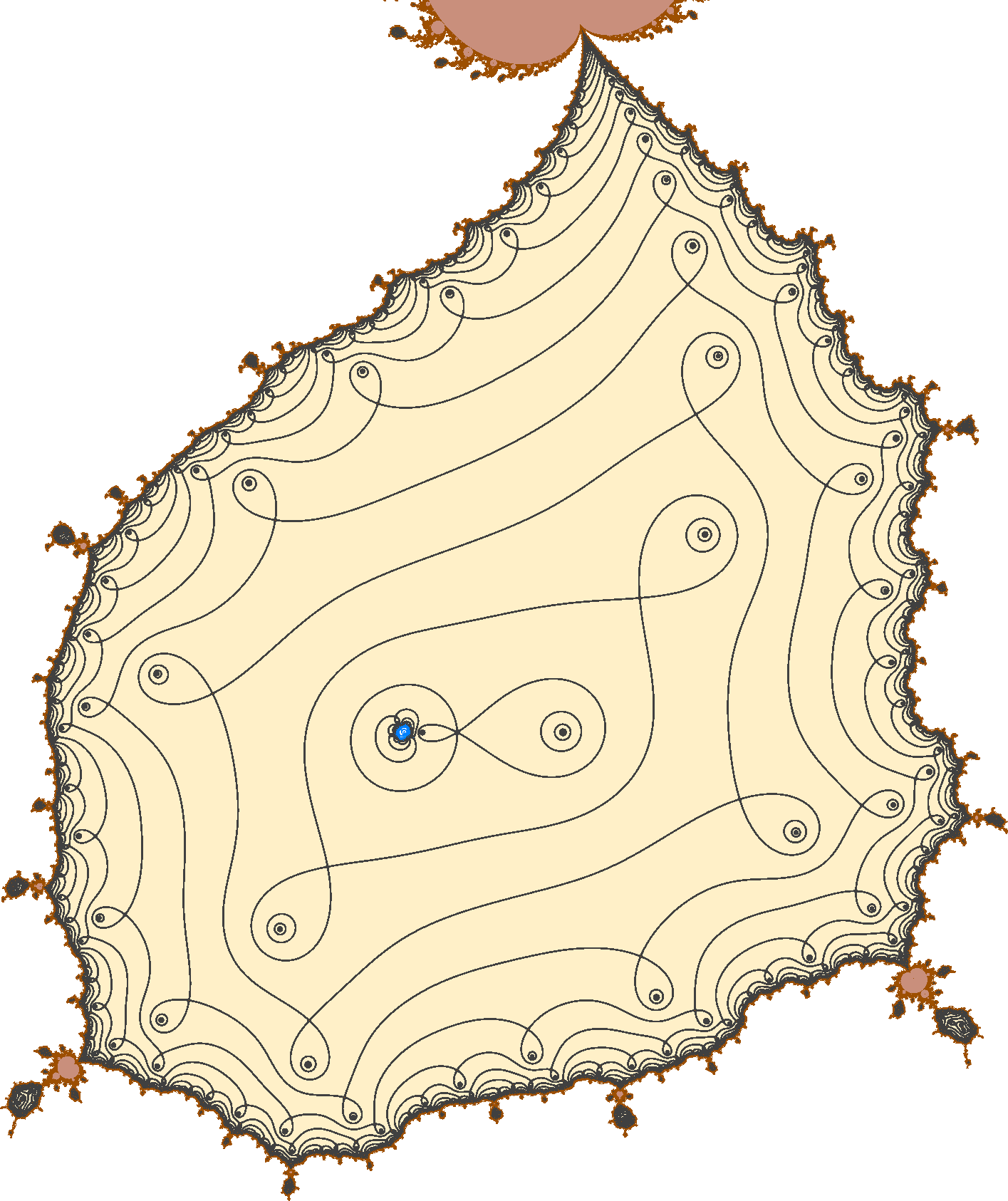}
   \end{center}
   \caption{We enriched \Cref{fig:bifattrc} with lines corresponding to the locus where both critical points $z=1$ and $z=c$ belong to the (whole) basin of $0$ and have the following property: $|\phi_P(c)/\phi_P(1)|\in|\lambda|^\Z$.}
   \label{fig:bifattrclines}
\end{figure}

\begin{figure}[ht]
	\begin{center}
   \includegraphics[width=\textwidth]{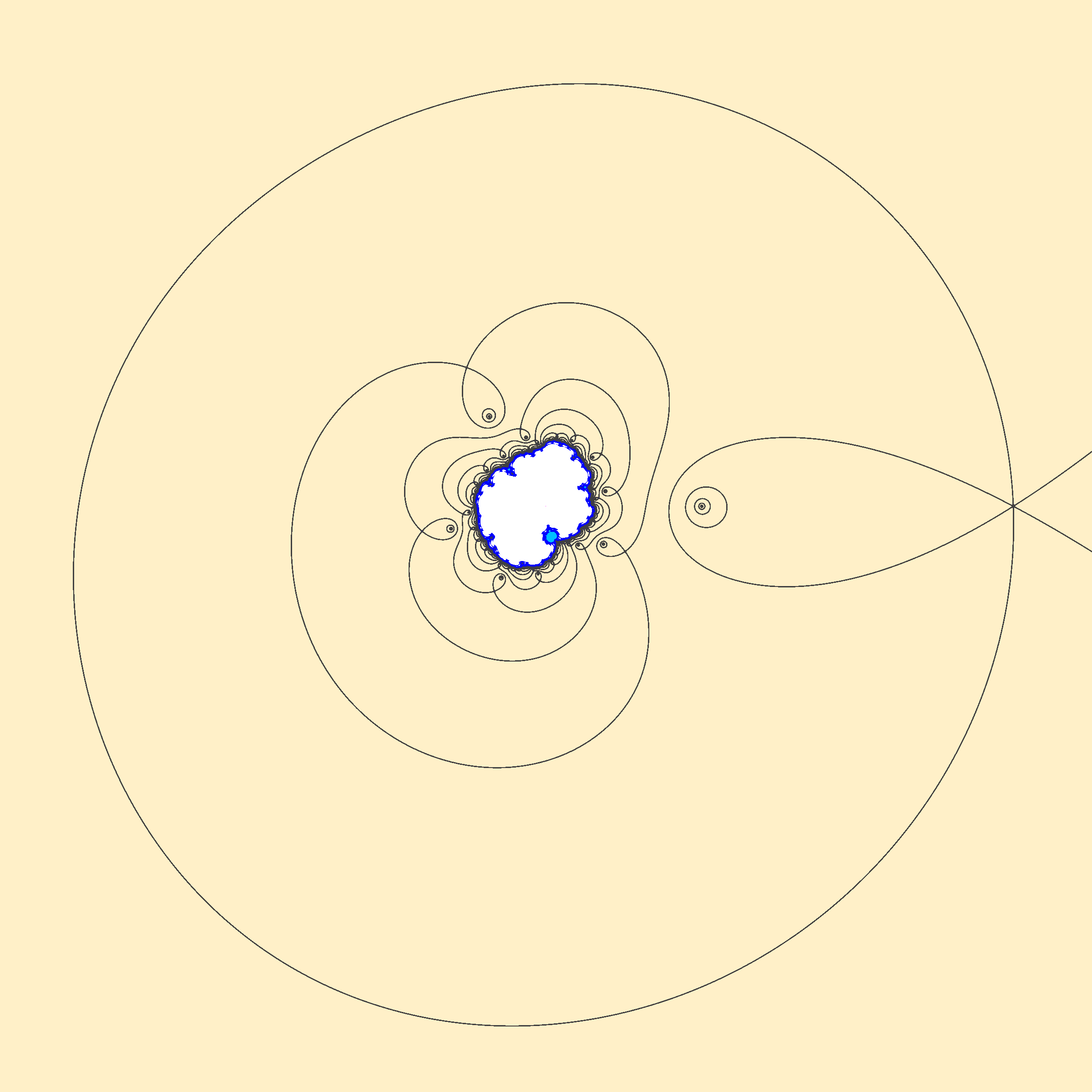}
   \end{center}
   \caption{Zoom on the central part of \Cref{fig:bifattrclines}. The set $Z_\lambda$ is the outermost circular-shaped curve.}
   \label{fig:zoom1}
\end{figure}

\begin{figure}[ht]
	\begin{center}
   \includegraphics[width=12cm]{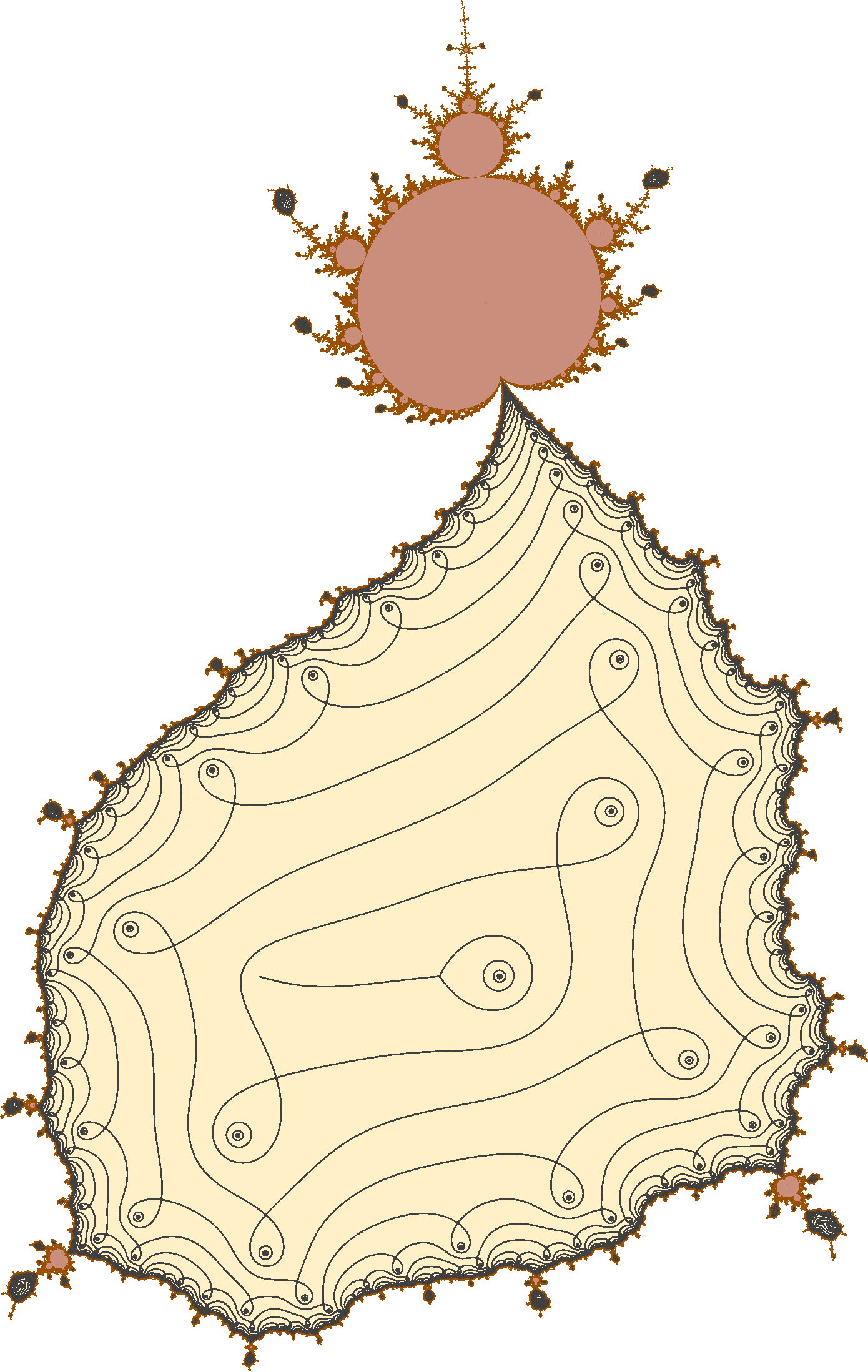}
   \end{center}
   \caption{Analog of \Cref{fig:bifattrclines} but this time drawn in the $v$-coordinate. It is not anymore possible to decree which critical point is red or blue, so we used only red for the bifurcation locus. The scar is quite visible: it is the tail of the tadpole-shaped figure in the centre.}
   \label{fig:bifattrvlines}
\end{figure}

\subsection{About the Petersen-Tan theorem}\label{sub:PT}

Recall that $\pi^{-1}(\lambda)$ is the set of affine conjugacy classes of cubic polynomials in ${\Hyp}_0$ whose attracting fixed point has multiplier $\lambda$.
In \cite{PT} is defined a bijection from $\pi^{-1}(\lambda)$ to a topological disk $D_\lambda$. More precisely $D_\lambda$ is obtained as follows: take the basin of attraction $B(Q_\lambda)$ of $0$ for the quadratic polynomial $Q_\lambda(z)=\lambda z + z^2$.
To simplify notations we write
\[P=P_{\lambda,c}\text{ and }Q=Q_\lambda.\]
Remove $U(Q)$ from it (the set $U(\cdots)$ has been defined in \Cref{prop:U}). Close the hole thus created by gluing $\partial U(Q)$ to itself according to the following rule: $z_1\sim z_2$ iff ($z_1,z_2\in\partial U(Q)$ and $\phi_\lambda(z_1)\phi_\lambda(z_2) = \phi_Q(c_Q)^2$) where: $c_Q$ denotes the critical point of $Q$, $\phi_\lambda:B(P)\to \C$ the extended linearizing coordinate of $P$, and $\phi_Q$ the analogue for $Q$ (see \Cref{sec:phi}).
This last relation can be understood as follows: the three complex numbers $\phi_\lambda(z_1)$, $\phi_\lambda(z_2)$ and $\phi_Q(c_Q)$ all belong to a circle of centre $0$ and we ask the first two to be symmetric with respect to the reflection along the line passing through $0$ and $\phi_Q(c_Q)$.\footnote{There is a difference with \cite{PT} because we took another normalizing convention for the functions $\phi$.}
Let
\[ D_\lambda:=(B(Q)\setminus U(Q))\,/\!\sim\]
and let
\[ \Pi: B(Q)\setminus U(Q) \to D_\lambda\]
be the quotient map.

The target $D_\lambda = (B(Q)\setminus U(Q))\,/\!\sim$ is a priori just a topological disk. We give it a complex structure with an atlas as follows: one chart is the identity on the following open subset of $\C$: $B(Q)\setminus \ov U(Q)$. For $z\in \partial U(Q)$ such that $\phi_Q(z)/\phi_Q(1)\neq \pm 1$, so that there is a point $z'\neq z$ on $\partial U(Q)$ that is equivalent to $z$. We can use the map $z\mapsto F(\phi_Q(z)/\phi_Q(1))$, where $F(z)=(z+z^{-1})/2$,  to define a chart near $z$ in the quotient, since the map $\phi_Q$ is invertible near $z$ and $z'$ and maps nearby points in $B(Q)\setminus U(Q)$ to points in $\C\setminus\D$. See \Cref{fig:seam-2}. The same also works if $\phi_Q(z)/\phi_Q(1) = -1$ because $\phi_Q$ is a bijection near $z=z'$, but will not work near $z=1$ where $\phi_Q$ has a critical point. There, one can use instead a branch of $z\mapsto\sqrt[3]{1-F(\phi_Q(z)/\phi_Q(1))}$. See \Cref{fig:seam-2}.
With this atlas, the map $\Pi$ is holomorphic on $B(Q)\setminus\ov U(Q)$ and has a holomorphic extension to neighborhoods of points of $\partial U(Q)\setminus\{1\}$.

\begin{figure}[ht]
	\begin{tikzpicture}
	\node at (0,0) {\includegraphics[width=\textwidth-.5cm]{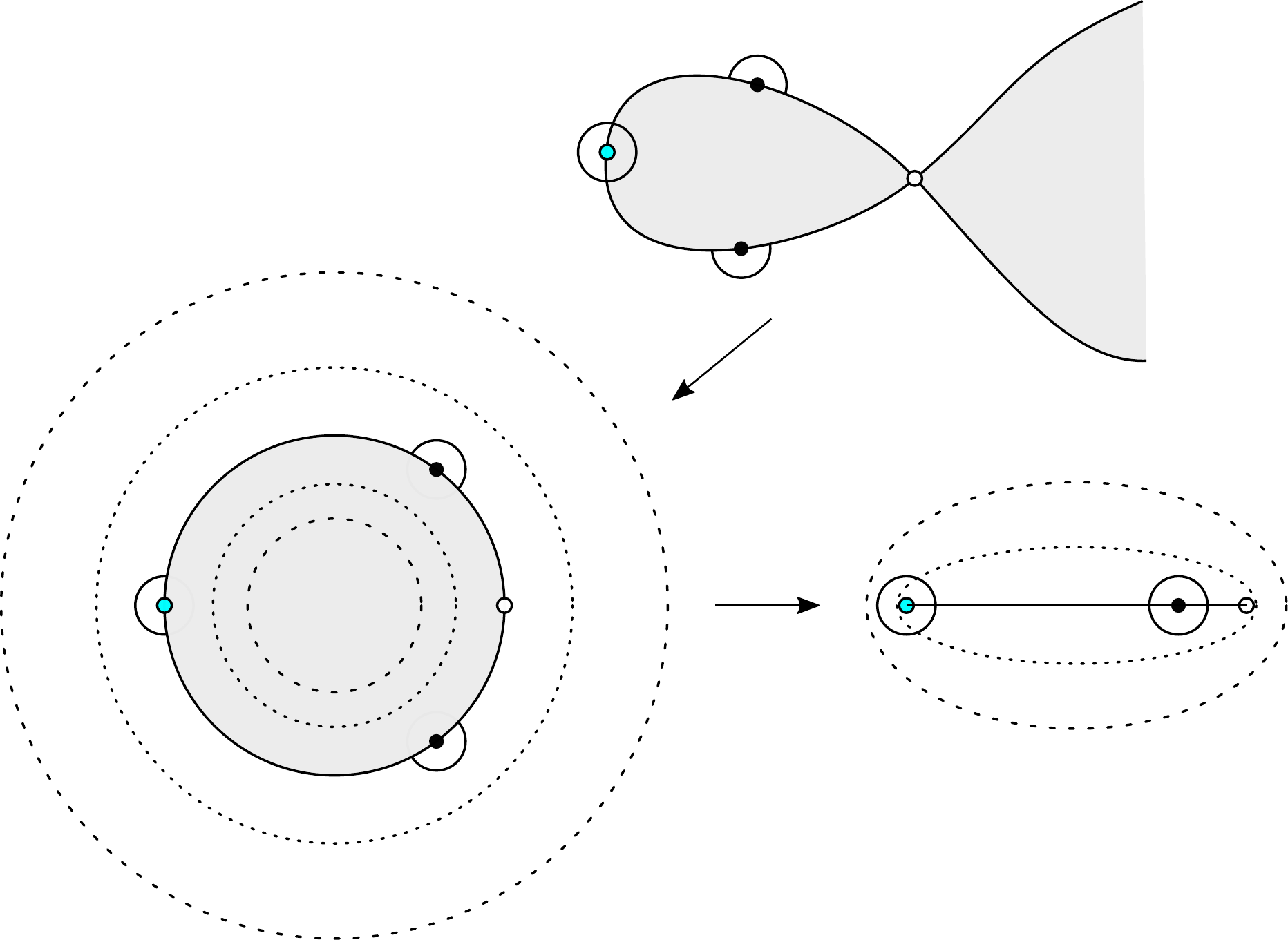}};
	\node at (2.3,0.7) {$z\mapsto \phi_Q(z)/\phi_Q(1)$};
	\node at (1.3,-2.5) {$F: z\mapsto \frac{z+z^{-1}}{2}$};
	\end{tikzpicture}
	\caption{Holomorphic charts for the quotient $(B(Q)\setminus U(Q))\ /\sim $ where $Q(z)=\lambda (z-\frac{z^2}{2}$. Illustration near a point not in $\ov U(Q)$ (green point), near a cardinal two fibre on $\partial U(Q)$ (pair of black points), near the point on $\partial U(Q)$ such that $\phi_1(z)/\phi_Q(1)=-1$. The case of the critical point $z=1$ is treated in \Cref{fig:seam-2}}
	\label{fig:seam-1}
\end{figure}

\begin{figure}[ht]
	\begin{tikzpicture}
	\node at (0,0) {\includegraphics[width=\textwidth-.5cm]{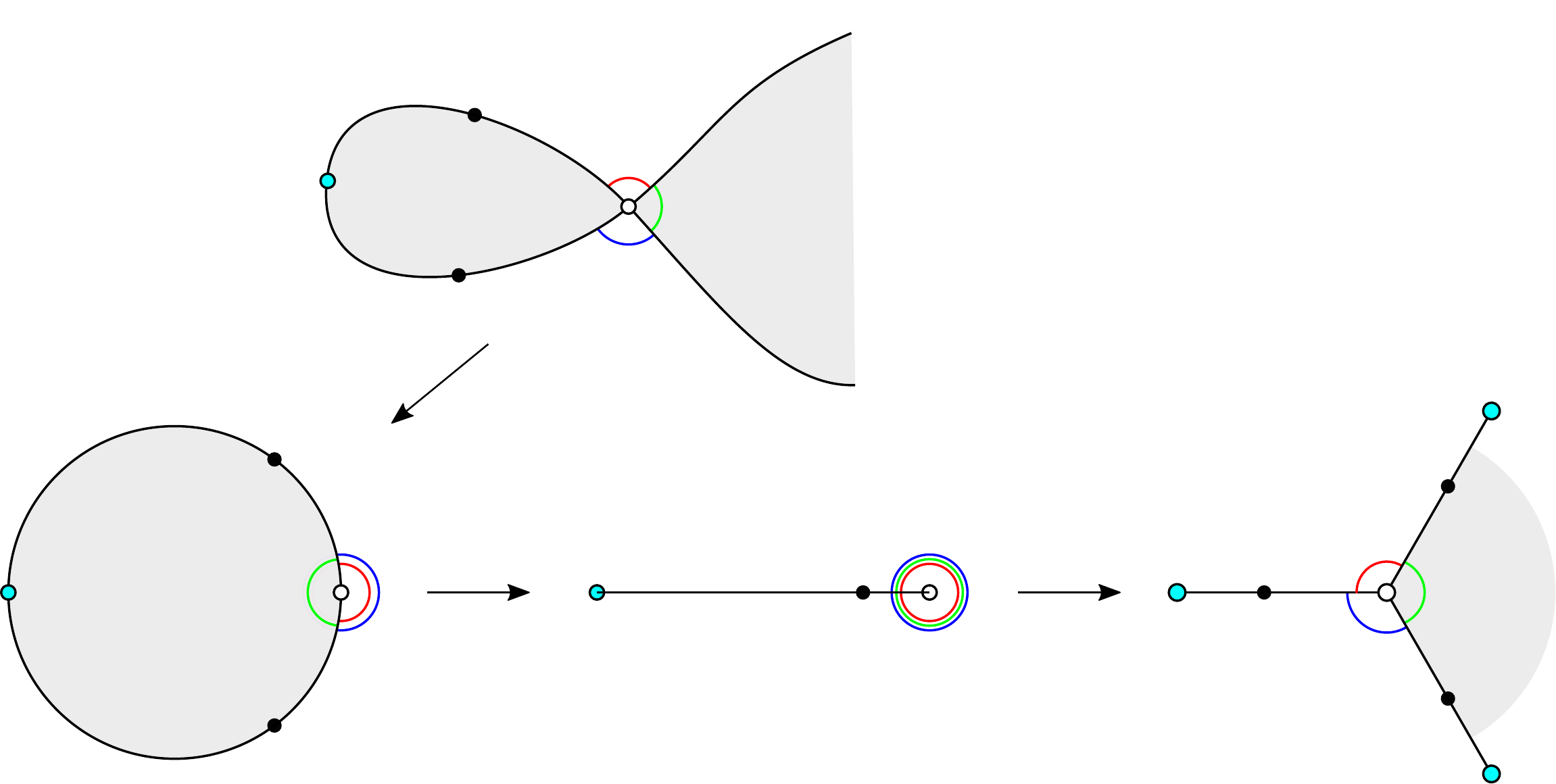}};
	\node at (-1.6,-0.1) {$\phi_Q/\phi_Q(1)$};
	\node at (-2.4,-1.2) {$F$};
	\node at (2.3,-1.05) {$z\mapsto \sqrt[3]{1-z}$};
	\end{tikzpicture}
	\caption{Continuation of \Cref{fig:seam-1}; holomorphic charts near the critical point on $\partial U(Q)$.}
	\label{fig:seam-2}
\end{figure}

The \emph{seam} $\partial U(Q)\ / \sim$ is a Jordan arc (it is homeomorphic the quotient of a circle by a reflection).
Petersen and Tan call it the \emph{scar}.

\begin{figure}[ht]
	\begin{center}
   \includegraphics[width=10cm]{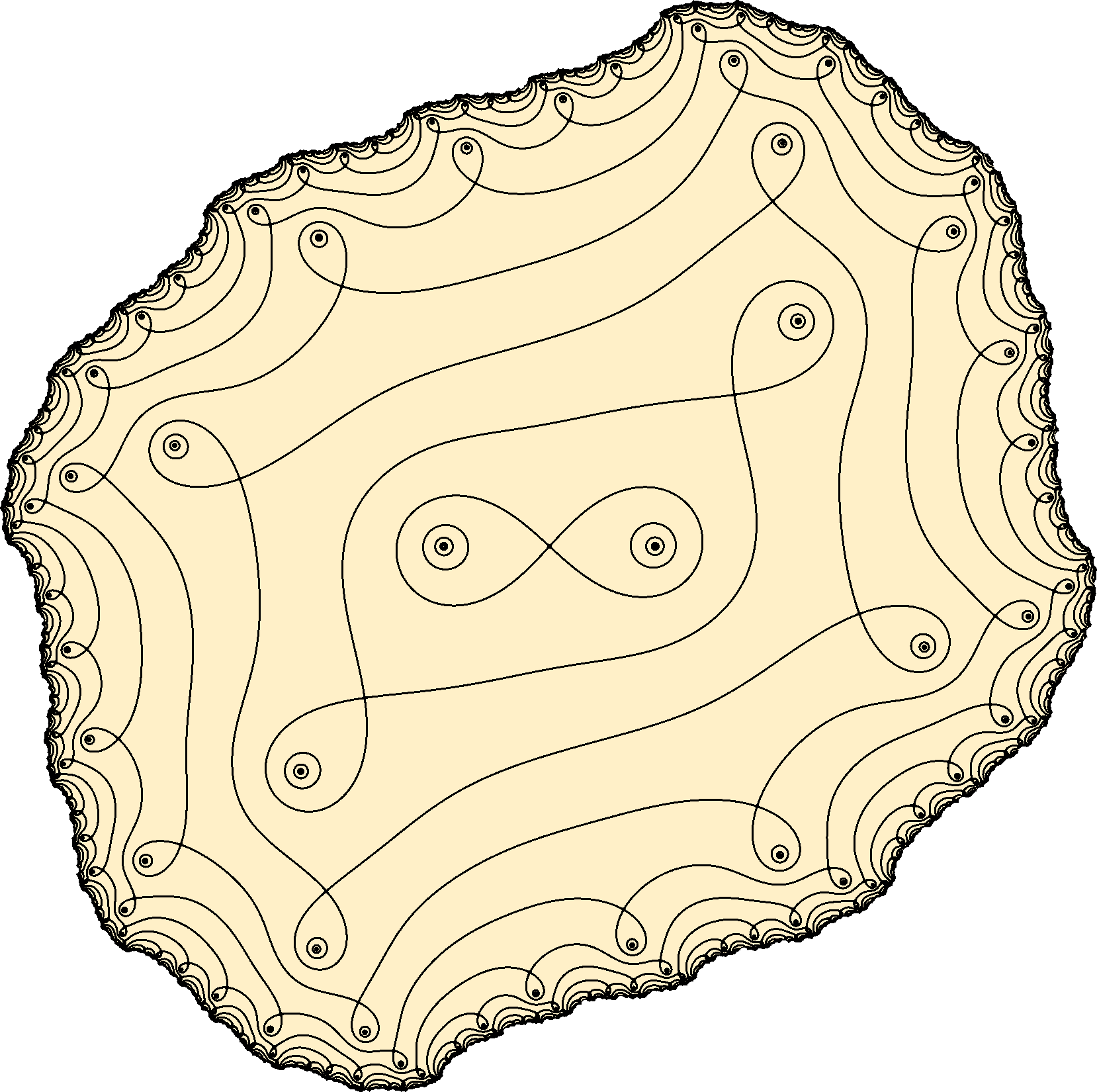}
   \end{center}
   \caption{The filled-in Julia set of $Q:z\mapsto \lambda(z-\frac{z^2}2)$ for $\lambda = 0.4i$ has been drawn in yellow and in its interior we represented ``equipotential'' lines, defined as the locus where $|\phi_Q(z)|/|\phi_Q(c_Q)| \in |\lambda|^\Z$, where $c_Q=1$ is the critical point of $Q$.
   The eye is the set $U(Q)$, which is the left lobe of the central lemniscate shaped curve.
   The attracting fixed point is at the centre of the pupil of the eye.}
   \label{pic:Q}
\end{figure}

\begin{figure}[ht]
	\begin{center}
   \includegraphics[width=7cm]{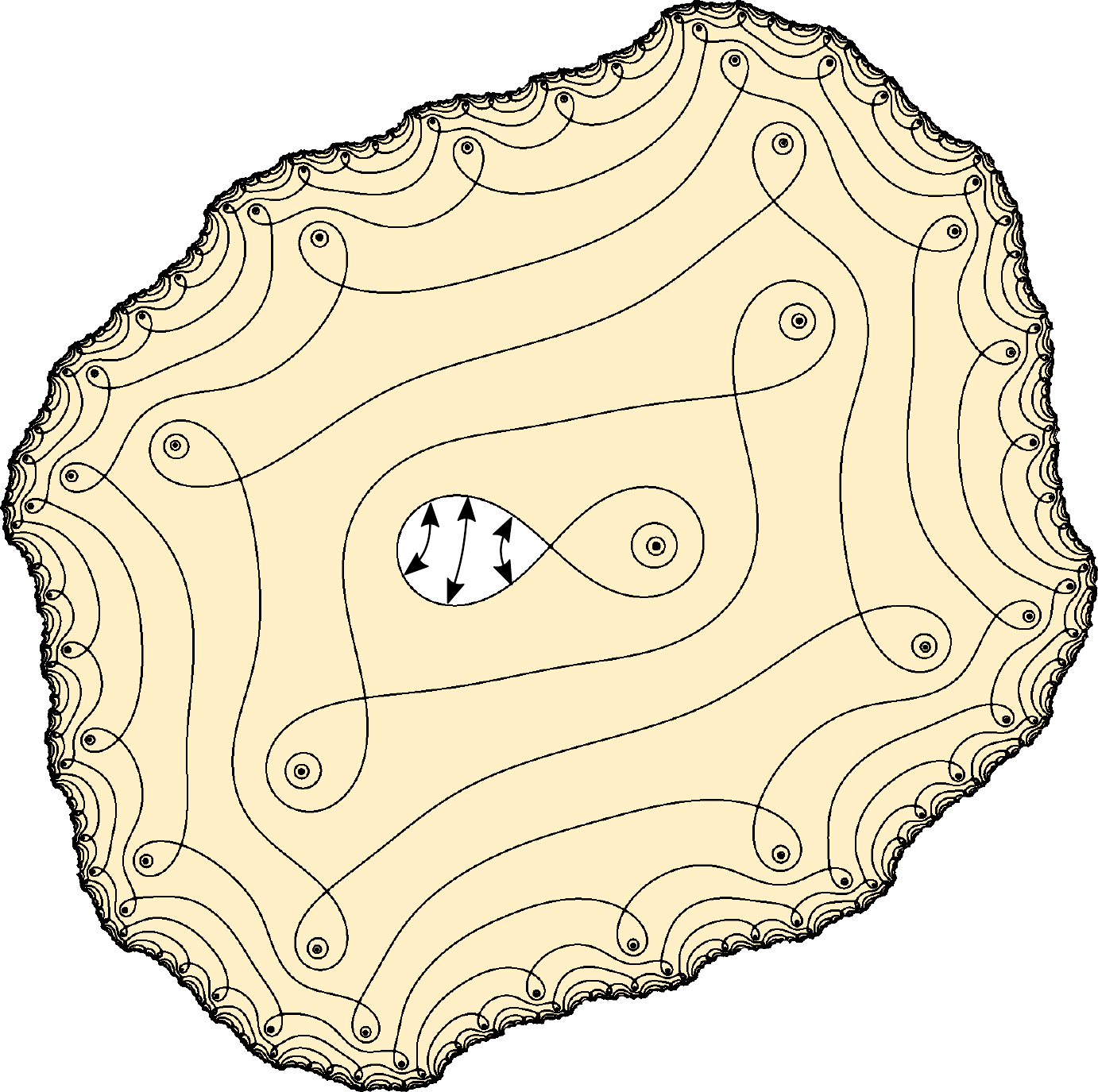}
   \end{center}
   \caption{To define the model space, remove the eye $U(Q)$ and glue the top and bottom eyelids together.}
   \label{pic:Qglue}
\end{figure}

Then they define a bijection from $\Hyp_0(\lambda)$ to $D_\lambda = (B(Q)\setminus U(Q))\,/\!\sim$ 
and prove that it is holomorphic (they also prove more properties).
We explain here without proof the definition of the bijection, the interested reader may look at \cite{PT} for more details.

\medskip

Let $P=P_{\lambda,c}$ with $[P] \in \Hyp_0$, i.e. both critical points in the immediate basin of the attracting fixed point $0$. Denote $U = U(P)$. 
Recall that there must be at least one critical point on $\partial U$. Sometimes both are on $\partial U$.
Denote by $c_0$ such a critical point and denote $c_1$ be the other one (possibly equal to $c_0$ if $P$ is unicritical). If there are two critical points on $\partial U$ then there is a choice of which one we call $c_0$.
The point $c_0$ is called the \emph{first critical point}.
We will also consider the \emph{co-critical} points $\co_0$ and $\co_1$, defined by $\{c_i,\co_i\} = P^{-1}(P(c_i))$, for $i=1,2$. If $c_0=c_1$ then $\co_0=c_0=c_1=\co_1$.

Recall that $\psi_P$ is the linearizing parametrization, maps its disks of convergence to $U$ and has a continuous extension $\ov \psi_P$ to a homeomorphism from the closed disk to $\ov U$, whose reciprocal is the restriction of $\phi_P$ to $\ov U$. The same statements hold for $Q$.
There is hence a natural conjugacy $\tilde \eta_P : \ov{U} \to \ov{U}(Q)$ of $P$ to $Q$ sending $c_0$ to $c_Q$ and obtained by
\[\tilde\eta_P(z) = \ov\psi_Q \left(\frac{\phi_Q(c_Q)}{\phi_P(c_0)}
\phi_P(z)\right)\]
It only depends on the affine conjugacy class of $P$: if $s\in\C^*$ and $f(z) = sP(s^{-1}z)$ then $\tilde\eta_f(z) = \tilde\eta_P(s^{-1}z)$.

Petersen and Tan prove that there exist:
\begin{itemize}
\item a special connected, simply connected and compact subset $\Omega$ of the basin $B(P)$ containing the closure of $U(P)$ and both critical points and such that $\co_0$ is either not in $\Omega$, or is a non-separating point of $\Omega$ contained in its boundary;
\item a semi-conjugacy $\eta_P:\Omega\to B(Q)$ of $P$ to $Q$, extending $\tilde \eta_P$.
\end{itemize}
The construction above only depends on the affine conjugacy class of $P$: if $s\in\C^*$ and $f(z) = sP(s^{-1}z)$ then $\eta_f(z) = \eta_P(s^{-1}z)$.
\begin{definition}[Petersen-Tan bijection]
To a an affine class $[P]\in\Hyp_0(\lambda)$ (without marked point), let $\Phi$ associate the point
\[\Phi([P]):=\Pi(\eta_P(c_1))\in (B(Q)\setminus U(Q))\ /\sim\]
where $\eta$ is the extension mentioned above.
\end{definition}
Note that this value does not depend on the chosen representative $P$ of the class.
The following immediate consequence will be useful later.
\begin{lemma}\label{lem:ppu}
If both critical points of $P$ are on $\partial U$, then
\[\eta_P(c_1) = \ov\psi_Q \left(\frac{\phi_Q(c_Q)\phi_P(c_1)}{\phi_P(c_0)}\right).\]
\end{lemma}

Since Petersen and Tan proved that $\Phi$ is a homeomorphism\footnote{They even proved that it is analytic for some natural complex structures on the domain and the range of the map $\Phi$.} it follows obviously that
\begin{equation}\label{eq:inj}
\text{the map $\Phi$ is injective.}
\end{equation}
We numbered that fact for future reference.

Curiously, the following statement is not present in \cite{PT}. For completeness we give here a proof using quasiconformal deformation and injectivity of the map $\Pi$.

\begin{proposition}\label{prop:PTo}
Let $[P]\in\Hyp_0(\lambda)$. Then $\Phi([P])$ belongs to the seam if and only if both critical points of $P$ belong to $\partial U(P)$.
\end{proposition}
\begin{proof}
If both critical points are on $\partial U(P)$ then $|\psi(c_1)|=|\psi(c_0)|$ hence $\frac{\phi_Q(c_Q)\phi_P(c_1)}{\phi_P(c_0)}$ has the same modulus as $\phi_Q(c_Q)$ hence
$\eta_P(c_1)\in\partial U(Q)$, so $\Phi([P])$ belongs to the seam. 

For the converse, we will prove below that for all $\theta\in\R$ there is a cubic polynomial $P=P_\theta$ whose critical points $c_0$, $c_1$ both belong $\partial U(P)$ and such that $\phi_P(c_1)/\phi_P(c_0) = e^{i\theta}$.
Once this claim is proved, consider any $z\in\partial U(Q)$.
Then $\phi_Q(z)/\phi_Q(c_Q)$ has modulus one, hence is of the form $e^{i\theta}$ for some $\theta\in\R$. 
From \Cref{lem:ppu} we get $\Phi([P_\theta]) = z$.
In other words: any point on the seam is the image by $\Pi$ of the class of some of the maps $P_\theta$ of the claim.
Now injectivity of $\Pi$ implies that a cubic map $P$ whose class is mapped to the seam by $\Pi$ must be one of the $P_\theta$ so must have both critical points on $\partial U(P)$.
\end{proof}

The proposition above used the following fact, that we prove now.

\begin{lemma}\label{lem:alltheta}
For all $\theta\in\R$ there is a cubic polynomial $P=P_\theta$ whose critical points $c_0$, $c_1$ both belong $\partial U(P)$ and such that $\phi_P(c_1)/\phi_P(c_0) = e^{i\theta}$.
\end{lemma}
\begin{proof}
For $\theta=0$ the map $P_{\lambda,c=1}$ satisfies the assumption: its critical points coincide.

For $\theta=\pi$, the map $P_{\lambda,c=-1}$, for which $c_1=-c_0$, satisfies the assumption: it has at least one critical point $c_0\in\partial U$ and since $U$ is invariant by $z\mapsto -z$, the other critical point $-c_0$ is also on $\partial U$. Moreover $\Phi_P$ commutes with $z\mapsto -z$, whence the claim.

For another value of $\theta$, we build $P$ by quasiconformal deformation of $P_0 = P_{\lambda,-1}$, i.e.\ the map $P$ will by the conjugate of $P_0$ by the straightening of a $P_0$-invariant Beltrami form $\mu_1$. Such a conjugate is holomorphic and is a self-map of $\C$ of topological degree $3$, hence a cubic polynomial.

To find $\mu_1$, we proceed as follows: The map $\phi_{P_0}$ sends $\ov U(P_0)$ to a closed round disk $B(0,R)$ and sends both critical points $c_0$, $c_1$ to antipodal points $a_0$, $a_1$ on its bounding circle.
Let $f$ be a (Lipschitz) homeomorphism of $[0,2\pi]$ fixing both ends and sending $\pi$ to $\theta$. We can assume that $f$ is linear on $[0,\pi]$ and $[\pi,2\pi]$ but it is not necessary. We can periodize $f$ into a Lipschitz homeomorphism of $\R$ commuting with $x \mapsto x+2\pi$.
Let
\[b = -\log \lambda\]
(we can take any determination of its logarithm).
Then $\Re b>0$.
Let
\[a = \frac{\Im b}{\Re b}.\]
Then let
\[W:\begin{array}{rcl}
\C&\to&\C
\\
x+iy &\mapsto& x+i (ax+f(y-ax))
\end{array}\]
Which commutes with $z\mapsto z+b$.
The map $W$ is semi-conjugate via $\exp$ to a map $V:\C\to\C$:
\[\exp\circ W = V \circ \exp\]
The map $V$ has been designed to commute with $z\mapsto \lambda z$,
to be quasiconformal and to send $-1$ to $e^{i\theta}$.
Let $\mu_V$ be the pull-back by $V$ of the null Beltrami form.
Let $\mu_0$ be defined on $U(P_0)$ as the pull-back of $\mu$ by the map $z\mapsto \phi_P(z /\phi_P(c_0))$.
Since $\mu_0$ is invariant by $P_0$ on the forward invariant set $U(P_0)$, we can complete $\mu_0$ into a $P_0$-invariant Beltrami form $\mu_1$ on $\C$ in the usual way: it is null outside the basin of attraction of $0$ and in the basin it is obtained by iterated pull-backs of $\mu_0$ by $P_0$.

Now let $S$ be the straightening of $\mu_1$, i.e.\ $S$ sends $\mu_1$ to the null form.
Let $P_1 = S\circ P_0 \circ S^{-1}$. The map $H= V\circ \phi_{P_0}\circ S^{-1}$ sends the null-form to the null-form hence is holomorphic.
It is defined on the basin of $P_1$ and moreover a direct computation shows that it conjugates $P_1$ to the multiplication by $\lambda$. Hence by uniqueness of the linearizing maps, we get that $H = a' \phi_{P_1}$ for some $a'\in\C^*$.
In particular the image of $U(P_0)$ by $S$ is $U(P_1)$: indeed it contains a critical point (in fact, both) of $P_1$ and $H$ is injective on it and maps it to a round disk.
One also checks that $H$ sends the critical points of $P_1$ to two points whose quotient is $e^{i\theta}$.
The lemma follows.
\end{proof}

\subsubsection{About semi-conjugacies}

Here, we discuss the impossibility of having a semi-conjugacy on the whole basin, this section has no application in the present document.

Let $P=P_{\lambda,c}$ with both critical points in the immediate basin of the attracting fixed point $0$. Let $U_n = U_n(P_{\lambda,c})$ denote the connected component containing $0$ of $ P_{\lambda,c}^{-n}(U)$. One can prove by induction that $U_n$ is simply connected.\footnote{If $U_{n-1}$ simply connected but not $U_n$ then the complement of $U_n$ in the Riemann sphere would have a bounded component $C$, whose image $P(C)$ is disjoint from $U_{n-1}$ and whose boundary is contained in $\partial U_{n-1}$. With these properties, $P(C)$ must contain infinity, and since $P$ is a polynomial, $C$ must contain infinity, but $C$ is bounded, leading to a contradiction.} 
We have $U_0 = U$ and $U_{n-1}\Subset U_n$.
Note that no critical points belong to $U_0$ and that at least one critical point belongs to $U_1$ because it contains $\partial U$.
There is some $n_0\geq 1$ such that both critical points belong to $U_n$ iff $n\geq n_0$. There is at least one critical point of $P$ on $\partial U_0$, let $c_0$ denote one of them ($c_0$ is a first critical point according to the terminology above) and let $\{c_0,c_1\}$ be the set of critical points of $P$.
The map $P$ is a ramified covering from $U_n$ to $U_{n-1}$ of degree $3$ if $n\geq n_0$, and of degree $2$ otherwise.\footnote{The value of degree can be deduced from the Riemann-Hurwitz formula: since the sets $U_n$ are simply connected, the degree is $1+$ the number of critical points of $P$ in $U_n$.} Note that $\co_0$ and $\co_1$ also belong to $U_{n_0}$, since $P$ has degree $3$ on $U_{n_0}$ and $U_{n_0-1}$ contains the two critical values.

For $n=0$, there are non-unique conjugacies $\zeta_0:U_0\to U_0(Q)$ of $P$ to $Q$. They are also the conformal maps from $U_0$ to $U_0(Q)$ that map $0$ to $0$. They all have a continuous extension to $\partial U_0$ because $U_0$ and $U_0(Q)$ are Jordan domains.
There is a unique conjugacy whose extension maps $c_0$ to the critical point of $Q$, and we now call this one $\zeta_0$.

As long as $n<n_0$ there is a conjugacy $\zeta_n$ of $P_{\lambda,c}$ to $Q$, extending $\zeta_{0}$, defined on $U_n$ and mapping to $U_n(Q)$ and if $n>0$ then $\zeta_{n}$ extends $\zeta_{n-1}$. 

Now there is a complication: there is no conjugacy $\zeta_{n_0}:U_{n_0}\to U_{n_0}(Q)$ from $P$ to $Q$ extending $\zeta_{n_0}$. Indeed $P$ is a degree $3$ ramified covering from $U_{n_0}$ to its image whereas $Q$ is a degree $2$ ramified covering from $U_{n_0}(Q)$ to its image. One may hope that allowing $\zeta_{n_0}$ to have critical points may solve the problem, however:

\begin{lemma}There is no holomorphic extension $\zeta$ of $\zeta_{n_0-1}$ to $U_{n_0}$.
\end{lemma}
\begin{proof}
 Let us work by contradiction and assume there is such a $\zeta$. 
 By holomorphic continuation, the relation $\zeta \circ P = Q\circ \zeta$ holds on $U_{n_0}$.
 This implies, denoting $\deg(f,z)$ the local degree at $z$ of a holomorphic map $f$ :
  \begin{equation}\deg(\zeta,P(z)) \deg(P,z) = \deg(Q,\zeta(z)) \deg(\zeta,z)\end{equation}
  Now since $\zeta_{n_0-1}$ is injective on $U_{n_0-1}=P(U_{n_0})$, it follows that $\forall z\in U_{n_0}$, we have $\deg(\zeta,P(z)) = 1$ so
  \begin{equation}\label{eq:degs2} \deg(P,z) = \deg(Q,\zeta(z)) \deg(\zeta,z)\end{equation}
 in particular if $z\in U_{n_0}$ and $\zeta(z)$ is the critical point of $Q$ then $z$ is a critical point of $P$, and its local degree is even, thus equal to $2$.
  This immediately rules out the possiblity that $c_0=c_1$, for the local degree of this double critical point would be $3$. Otherwise $P(\co_0) = P(c_0)$ hence $Q(\zeta(\co_0)) = \zeta(P(\co_0)) = \zeta(P(c_0)) = Q(\zeta(c_0)) = Q(c)$ where $c$ denotes the critical point of $Q$, so $\zeta(\co_0) = c$ because $c$ is the only preimage of $Q(c)$ by $Q$. Hence $\co_0$ must be critical as we already remarked. But this is not the case, leading to a contradiction.
\end{proof}

As the proof above shows, the obstruction is essentially due to the co-critical point $\co_0$.
This is why Petersen and Tan had to extend the conjugacy $\zeta_{n_0-1}$ into a semi-conjugcacy $\zeta$ defined only on some subset of $U_{n_0}$ containing $c_1$ and either not containing $\co_0$ or at least with $\co_0$ not ``in the way''. For this, they had to consider many cases, and we will not review them here.

\subsection{Proof of Theorem~\ref{thm:ll}}
First note that the claim on the total mass of $\mu_\lambda$ was proven in \Cref{prop:totalmass}.

Recall that $0<|\lambda|<1$ and that
\[
Z_\lambda = \{c\in\C^*\,;\,\text{both critical points of $P_{\lambda,c}$ belong to }\partial U\}
\]
where $U = U(P_{\lambda,c})$ is the set defined in \Cref{prop:U}.
The critical points of $P_{\lambda,c}$ are $z=1$ and $z=c$. The set $Z_\lambda$ contains $1$ and when $c\in\C^* \setminus Z_\lambda$, there is only one critical point on $\partial U$.

We invite the reader to read the statement of the fourt point of \Cref{lem:cP} again, about limits of critical points on $\partial U$ when the polynomial varies.
\begin{assertion}
The set $Z_\lambda$ is closed in $\C^*$. 
On any connected component of the complement of $Z_\lambda$, it is always the same critical point that belongs to $\partial U$.
\end{assertion}
The two assertions follow from the fourth point of \Cref{lem:cP} and
continuity of the two critical points of $P_{\lambda,c}$ with respect to $c$.

\medskip

Let us prove that $Z_\lambda$ is a Jordan curve (which gives an independent proof of the fact that it is closed).
Recall that $P_{\lambda,c}$ and $P_{\lambda,1/c}$ are conjugate by an affine map fixing $0$, in particular $Z_\lambda$ is invariant by $c\mapsto 1/c$.
It contains $c=1$ because both critical points are then identical.
It also contains $c=-1$ for then $P_{\lambda,-1}$ commutes with $z\mapsto -z$, hence $U$, which contains $0$, is invariant by $-z$ too and the two critical point $-1$ and $1$ thus belong to $\partial U$ simultaneously.

Let $I_\lambda$ denote the image of $Z_\lambda$ by $c\mapsto v = \frac{c+c^{-1}}{2}$. It contains $v=1$ and $v=-1$ and $Z_\lambda$ is its whole preimage.

\begin{lemma}\label{lem:Il}
  The set $I_\lambda$ is a Jordan arc.
\end{lemma}
\begin{proof}
  Recall that $\Hyp_0(\lambda)$ can be seen as a subset of $\C^2$ via the $(a,b^2)$ coordinates.
  By \Cref{prop:0ra} this set is homeomorphic by $(\lambda,v)\mapsto(a,b^2)$ to the subset that we denoted $\Hyp^a(\lambda)$ of $(\lambda,v)$-space, of classes of polynomials with an attracting fixed point marked of multiplier $\lambda$. According to Petersen and Tan (see \Cref{sub:PT}, \Cref{prop:PTo}), the fact that both critical points are on $\partial U(P)$ is equivalent to the fact that $\Phi([P])$ belongs to the seam $\Pi(\partial U(Q))$. Since the seam is a Jordan curve, and $\Phi$ is a homeomorphism from $\Hyp_0(\lambda)$ to its image, the lemma follows.
\end{proof}
 
\begin{corollary}
  The set $Z_\lambda$ is a Jordan curve.
\end{corollary}
\begin{proof}
  By lifting properties of coverings, the Jordan arc minus its ends has two disjoint lifts starting from its middle point by $c\mapsto v$. The ends of each lift must converge to $-1$ and $1$. The union of these two Jordan arcs is then a simple closed curve and equal to $Z_\lambda$.
\end{proof}

In particular $Z_\lambda$ is bounded. Since it is invariant by $c\mapsto 1/c$, it is also bounded away from $0$.
In particular, the complement of $Z_\lambda$ in $\C$ has two components, one that is bounded and one that is unbounded.

\begin{assertion}
The unique critical point of $P_{\lambda,c}$ that belongs to $\partial U$ is $c$ if $c\neq 0$ belongs to the bounded component of $\C\setminus Z_\lambda$, and it is $1$ if $c$ belongs to the unbounded component.
\end{assertion}

By the discussion at the beginning of this section, it is enough to prove that the unbounded component contains \emph{at least one parameter} for which $c\in \partial U$ and similarly that the bounded component minus $0$ contains at least one parameter for which $1\in \partial U$.

When $c$ tends to infinity, the map $P_{\lambda,c}$ tends on every compact subset of $\C$ to the quadratic polynomial $Q_\lambda(z)=\lambda z(1-\frac{z}2)$.
We have $J(Q_{\lambda}) \subset B(0,10)$.
The restriction of $Q_\lambda$ as a map from $Q^{-1}(B(0,10))$ to $B(0,10)$ is quadratic-like. For $|c|$ big enough, there is a quadratic-like restriction of $P_{\lambda,c}$ whose domain contains $0$ and is contained in $B(0,10)$ (a perturbation of a quadratic like map is still quadratic like, up to reducing its domain). This restriction has an attracting fixed point $z=0$, hence there is a critical point of the restriction in its basin. This critical point must belong to $B(0,10)$ hence cannot be equal to $c$ if $c$ is big enough. Hence it must be the critical point $z=1$. The boundary of the basin is contained in the Julia set of the restriction, hence in the Julia set of the full polynomial. It implies that $c$ is not in the immediate basin of $0$ for $P_{\lambda,c}$, a fortiori not in $\partial U$. 

Recall that $P_{\lambda,1/c}$ is conjugate to $P_{\lambda,c}$ by $z\mapsto cz$:
\[c^{-1}P_{\lambda,c}(cz) = P_{\lambda,1/c} (z).\]
The conjugacy $z\mapsto cz$ sends respectively the critical points $1$ and $1/c$ of $P_{\lambda,1/c}$ to the critical points $c$ and $1$ of $P_{\lambda,c}$.
Applying this change of variable, it follows from the above discussion that the critical point on $\partial U$ is $c$ when $|c|$ is small.

\medskip

\begin{assertion}
  The function $c\mapsto -\log r(P_{\lambda,c})$ defined on $\C^*$ is subharmonic and continuous. It is harmonic on $\C^*\setminus Z_\lambda$.
\end{assertion}
Continuity has been proven in \Cref{lem:cont} and subharmonicity in \Cref{prop:sh}.
To prove harmonicity on the complement of $Z_\lambda$, we use \cref{eq:rpc} on page~\pageref{eq:rpc}, according to which $r(P)=|\phi_P(c_P)|$ where $c_P$ is the critical point on $\partial U(P)$, and we use holomorphic dependence of $\phi_P$ on $P$, \Cref{lem:holodep}. We saw that $c_P=1$ on one component and $c_P=c$ on the other component.

\medskip

The (distribution) Laplacian of a subharmonic function is known to be a Radon measure, let us call it $\mu_\lambda$.
\begin{lemma}
The support of $\mu_\lambda$ is equal to $Z_\lambda$.
\end{lemma}
\begin{proof}
The support is the complement of the biggest open set on which $-\log r$ is harmonic.
We proved that $-\log r$ is harmonic on the complement of the closed set $Z_\lambda$, hence the support of $\mu_\lambda$ is contained in $Z_\lambda$.

To prove the converse inclusion, we adapt to our setting an argument that was explained to us by \'Avila in the setting of Siegel disks. We will proceed by contradiction and assume that there is some ball $B=B(c_0,\rho)$ with $c_0\in Z_\lambda$ and on which the function $-\log r$ is harmonic.
Let us deduce from this that one of the critical points is on $\partial U$ for all $c\in B$, i.e.\ that either $\forall c\in B$, $1\in\partial U(P_{\lambda,c})$ or $\forall c\in B$, $c\in \partial U(P_{\lambda,c})$.
This leads to a contradiction since $c_0\in Z_\lambda$ is accumulated by points in each of the two complementary components of $Z_\lambda$, and on one of those components the only critical point on $\partial U(P_{\lambda,c})$ is $z=1$ and on the other it is $z=c$.

A harmonic function on a simply connected open set is the real part of some holomorphic function, so $\log r(P_{\lambda,c}) = \Re g(c)$ with $g : B\to\C$ holomorphic.
In other words,
\[\forall c\in B,\ r(P_{\lambda,c}) = | h(c) |\]
for the non-vanishing holomorphic function $h=\exp \circ \,g$. On the other hand, $r(P) = |\phi_P(c_P)|$ for any critical point $c_P\in\partial U(P)$. Recall that the critical point $c_P$ with $P=P_{\lambda,c}$ is unique if $c$ is not in $Z_\lambda$ and that is a holomorphic function of $c$ in the complement of $Z_\lambda$.
Since $Z_\lambda$ is a Jordan curve, $c_0$ is in the closure of both complementary components of $Z_\lambda$.
Taking the intersection of a complementary component with $B$ may disconnected it, but at least on each component $C'$ of this intersection, the function $\phi(c_P)/h(c)$ has constant modulus equal to $1$, so is constant on $C'$. Choose one such component $C'$ and call $u$ this constant:
\[|u|=1,\text{ and}\]
\[\forall c\in C',\ \phi(c_P) = u h(c).\]
Recall that $\psi_P$ has a continuous extension $\ov\psi_P$ to a homeomorphism from $\ov B(0,r(P))$ to $\ov U(P)$ (\Cref{lem:psiext}).
For $c\in B$ let
\[\zeta(c) = \ov{\psi}_{P_{\lambda,c}}(uh(c)),\]
which is defined since by assumption $r(P_{\lambda,c}) = |h(c)|$ and $|u|=1$.
The function $\zeta$ is continuous by the third point of \Cref{lem:cP}.
For $c\in C'$ we have $\zeta(c)=c_P$.

Let us prove that the function $\zeta$ is homlomorphic.
It is the pointwise limit as $\epsilon\to 0$ of the holomorphic functions $c \mapsto  \ov{\psi}_{P_{\lambda,c}}(u(1-\epsilon)h(c))$. These functions are uniformly bounded on compact subsets of $B$: one argument for that is that they take value in the filled-in Julia set, which are contained in a common ball when the parameter varies little. A uniformly bounded pointwise limit of holomorphic functions is holomorphic. It follows that $\zeta$ is locally holomorphic, hence holomorphic.

Now note that $\zeta$ coincides with one of the two critical points $z=c$ or $z=1$ of $P$ on the component $C'$. By holomorphic continuation of equalities, $\zeta$ is this critical point on all $B$. It follows that one of the critical points is always on $\partial U$ for $c\in B$, leading to the aforementioned contradiction.
\end{proof}

\subsection{Parametrizing the Z-curve in the attracting case}

Let $\U=\partial \D$ and consider the map
\[\Psi:\begin{array}{rcl}
Z_\lambda & \to & \U
\\
c & \mapsto & \phi_{P}(c)/\phi_P(1)
\end{array}\]
where, as usual, $P = P_{\lambda,c}$ and $1$ and $c$ are the critical points of $P$.

Recall that $c=1$ and $c=-1$ both belong to $Z_\lambda$. We have:
\begin{itemize}
\item $\Psi(1) = 1$, since in this case, the two critical points coincide,
\item $\Psi(-1) = -1$, because in this case the map $\phi_P$ commutes with $z\mapsto-z$.
\end{itemize}

Recall that $c\in Z_\lambda$ iff $1/c\in Z_\lambda$.
\begin{lemma}\label{lem:Psiinvc}
  $\Psi(1/c) = 1/\Psi(c)$.
\end{lemma}
\begin{proof}
  Let $P=P_{\lambda,c}$ and $B = P_{\lambda,1/c}$.
  We have $B(z) = P(cz)/c$ and $\phi_{B}(z)=\phi_P(cz)/c$, hence
$\Psi(1/c)$ $=$ $\phi_{B}(1/c)/\phi_{B}(1)$ $=$  $\phi_{P}(c\times 1/c)/\phi_P(c\times 1)$ $=$ $\phi_{P}(1)/\phi_P(c)$ $=$ $1/\Psi(c)$.
\end{proof}
Let us prove that:
\begin{lemma}The map
$\Psi$ is a homeomorphism.
\end{lemma} 
\begin{proof}Note that
\begin{itemize}
\item It is continuous, since $\phi_P$ depends holomorphically on $P$.
\item By \Cref{lem:alltheta}, the map $\Psi$ is surjective.
\item To prove injectivity of $\Psi$ we will use injectivity of $\Phi$, see Eq.~\eqref{eq:inj} as follows:
\end{itemize}
Consider two maps $P_{\lambda,c}$, $P_{\lambda,c'}$ in $Z_\lambda$ that have the same image by $\Psi$. Then their affine class (without marked point) have the same image by $\Phi$ according to \Cref{lem:ppu}.
It follows that the two maps are affine conjugate. Hence either $c'=c$ or $c' = 1/c$.
In the latter case, by \Cref{lem:Psiinvc}, we have $\Psi(c')=1/\Psi(c)$. Since we assumed moreover that $\Psi(c')=\Psi(c)$, it follows that $\Psi(c) = 1/\Psi(c)$, hence $\Psi(c) = \pm 1$. We already know that $\Psi(1)=1$ and $\Psi(-1)=-1$ and hence we have that either $c$ or $1/c$ is equal to $\pm 1$ by the above analysis. But then $c=1/c$.
\end{proof}

\begin{figure}[ht]
	\begin{center}
   \includegraphics[width=\textwidth]{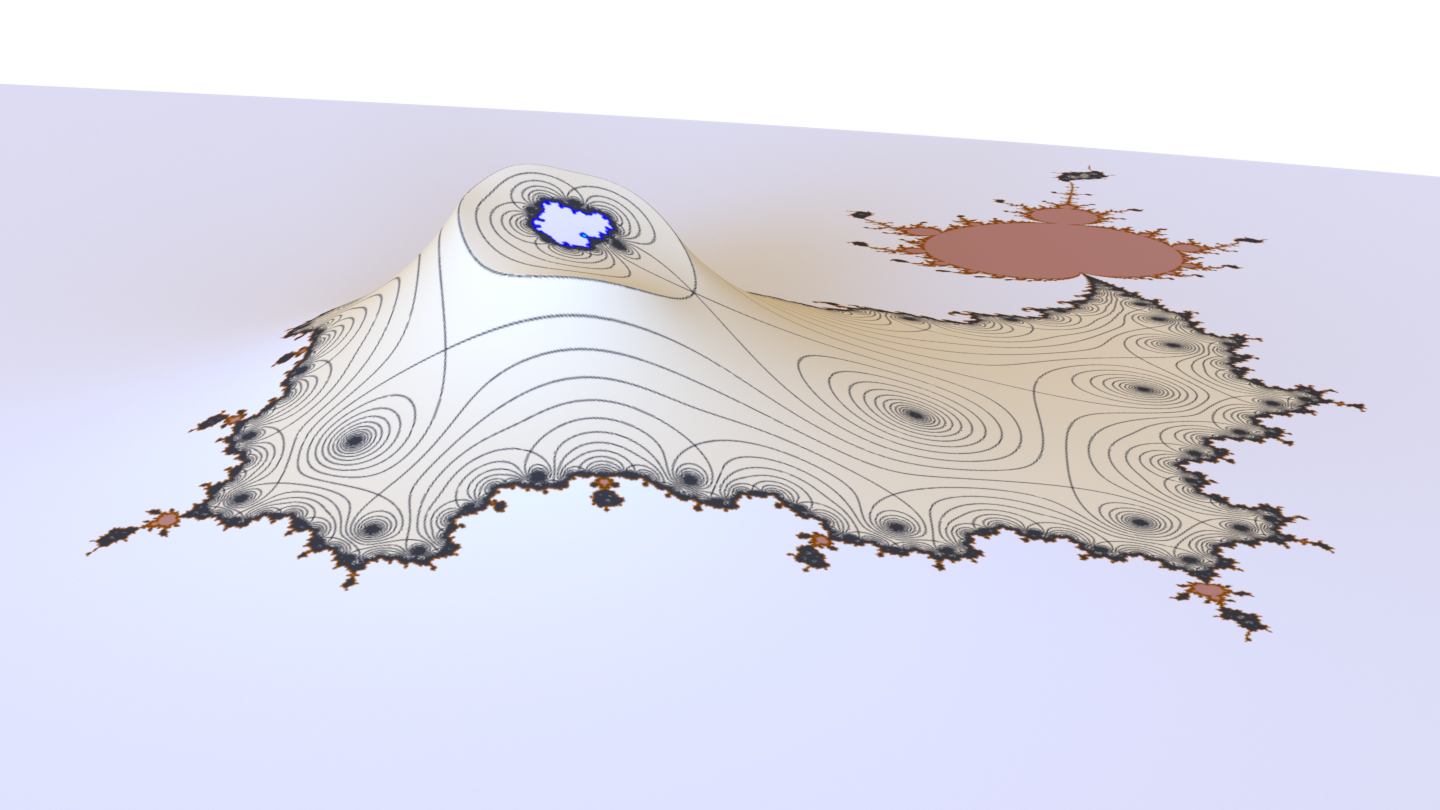}
   \end{center}
   \caption{3D rendering of the graph of the function $c\mapsto \log r(P_{\lambda,c}) - \log |c|$ for $\lambda = 0.8i$ and $c$ varying in a bounded subset of $\C$. This graph, a smooth surface except along a curve where it is creased, is textured with the bifurcation locus of the family $c\mapsto P_{\lambda,c}$. Compared to the function $-\log r$, we added $\log|c|$ and then took the opposite. This allows an elegant representation as a volcano looking scenery. The horizontal scale and the vertical scales have been chosen different to fine-tune this aspect.
   The modified function is defined on $\C$, has a limit as $c\tend 0$, is harmonic outside $Z_\lambda$ and tends to $-\infty$ when $c\tend \infty$.
  Its Laplacian is the opposite of the measure $\mu_\lambda$.}
   \label{pic:heightfield}
\end{figure}

\begin{figure}[ht]
	\begin{center}
   \includegraphics[width=12.5cm]{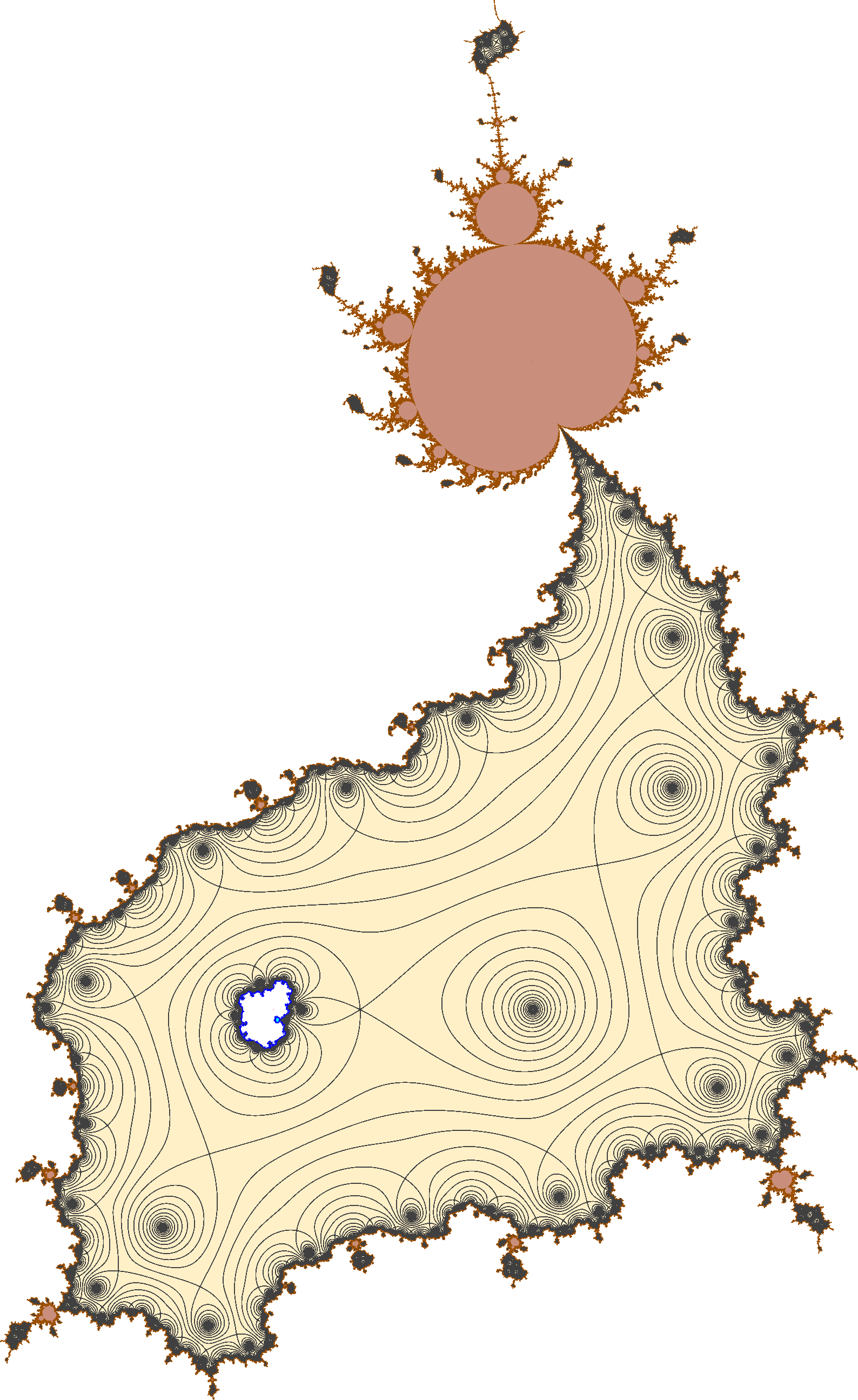}
   \end{center}
   \caption{The bifurcation locus together with equipotentials, c.f.\ \Cref{pic:heightfield,fig:bifattrclines}}
   \label{pic:ba2}
\end{figure}

\section{Siegel slices}\label{sec:5}

\subsection{Introduction}

Shishikura proved that all bounded type Siegel disks of polynomials are quasicircles with a critical point in the boundary.

This applies to our family: when $\theta$ is a bounded type irrational and $\lambda=e^{2\pi i\theta}$, then for all $c\in\C^*$ the Siegel disk of $P_{\lambda,c}$ at $0$ is a quasidisk containing at least one critical point.

Let us recall a theorem of Zakeri, in \cite{art:Zakeri}.
\begin{theorem}[Zakeri]\label{thm:Zak}
 if $\theta$ is a bounded type irrational number and $\lambda = e^{2\pi i\theta}$, let
  \[Z_\lambda = \{c\in\C^*\;|\;\text{both critical points belong to }\partial \Delta(P_{\lambda,c}).\}\]
  Then $Z_\lambda$ is a Jordan curve. Call $I$ and $E$ the bounded and unbounded components of its complement in $\C$. Then $I$ contains $0$ and for all $c\in I\setminus\{0\}$, the critical point on $\partial \Delta(P_{\lambda,c})$ is $z=c$; for all $c\in E$, the critical point on $\partial \Delta(P_{\lambda,c})$ is $z=1$.
\end{theorem}

The set $Z_\lambda$ is referred to here as the \emph{Zakeri curve}. Given a measure $\mu$ let $\Supp\mu$ denote its support.
The object of this section is to prove:

\begin{theorem}\label{thm:bddType}
  Let $\theta$ be a bounded type irrational and $\lambda = e^{2\pi i\theta}$.
  Then the support of $\mu_\lambda$ is equal to the Zakeri curve:
  \[ \Supp \mu_\lambda = Z_\lambda.\]
\end{theorem}

\subsection{Proof}

Let us denote by $\cri_1(c)=1$ and $\cri_2(c)=c$ the two holomorphic parametrizations of the critical points of $P_{\lambda,c}$. Given a simply connected open subset $\Delta$ of $\C$ containing $0$ we denote by $r(\Delta)$ its conformal radius with respect to $0$.

\subsubsection{\texorpdfstring{$\Supp\mu_\lambda \subset Z_\lambda$}{Supp mu lambda is contained in Z lambda}} 

Let $c_0\in\C^*\setminus Z_{\lambda}$. We will prove that $c_0\notin \Supp\mu_\lambda$.

\smallskip

The following result is due to D.\ Sullivan (see \cite{Zakeri2}).

\begin{proposition}[Sullivan]\label{prop:sullivan}
	Let $(f_a)_{a \in B(a_0, r)} : (U, 0) \longrightarrow (\C, 0)$ be a one parameter family of holomorphic maps with $f_a(z) = \lambda z + \cal O(z^2)$ with $|\lambda| =1$ and that depends analytically on $a$. Assume that for all $a$ the map $f_a$ has a Siegel disc $\Delta_a$ around $0$, that $\Delta_a$ has finite conformal radius w.r.t.\ $0$ for at least one parameter and that $\partial \Delta_a$ undergoes a holomorphic motion as $a$ varies. Then, $a \mapsto \log r( \Delta_a )$ is harmonic.
\end{proposition}

\begin{proof}
A simply connected subset of $\C$ has finite conformal radius iff it is not the whole complex plane. As $\partial \Delta_a$ undergoes a holomorphic motion, if one $\Delta_a$ is different from $\C$, then they are all different from $\C$.
By Slodkowsky's theorem, let us extend the motion to a holomorphic motion of all the plane $\C$. Let $z_n$ be a sequence points in the Siegel disk associated to $f_{a_0}$ converging to the boundary of the Siegel disk. For $a$ in a small neighborhood of $a_0$ let $z_n(a)$ be the point that the holomorphic motion transports $z_n$ to.
Let $\psi_a : B(0,r_a) \maps \Delta_a$ be the linearizing parametrization, normalized by $\psi_a(z) = z + \cal O(z^2)$ near $0$.
Note that $a\mapsto r_a$ is continuous by a theorem of Caratheordory (\cite{Po}, Section~1.4, in particular Theorem~1.8 page~29).
Now look at,
\[ u_n(a) = \psi_a^{-1}(z_n(a)) \]
defined for $a\in B(a_0,r)$.
For each $a$, the sequence $(z_n(a))_{n \in \N}$ converges to a point in the boundary of the Siegel disk. Thus, $(|u_n(a)|)_{n \in \N}$ converges to the conformal radius of $\Delta_a$, $r(a)$. The central remark is that the map
\[ (a,z) \maps (\psi_a(z), z) \]
is bi-analytic: indeed, $\psi_a(z)$ is given by a power series in $z$ whose coefficients depend analytically on $a$.
So $a \maps u_n(a)$ is also analytic.
Therefore, the maps $a \mapsto \log |u_n(a)|$ are harmonic.
They are (locally) bounded away from $\infty$: indeed 
we have $|u_n(a)|\leq r_a$ and $a\mapsto r_a$ is continuous.
The map $a \maps \log r(a)$ is the pointwise limit of those maps, an so is harmonic too (by \cite{Ra}, a pointwise limit of positive harmonic maps $h_n$ is harmonic, this immediately adapts to families that are bounded above by $B\in\R$ by considering $B-h_n$.).
\end{proof}

Since $c_0\notin Z_\lambda$, then by \Cref{thm:Zak} one of the critical points $\cri_i(c)$ remains on the boundary of the Siegel disk $\Delta$ for $c$ in a ball $B(c_0,r_0)$ disjoint from $Z_\lambda$.
When $\Delta$ is a Jordan curve, which is the case when $\theta$ has bounded type, the analytic conjugacy $\psi$ from the rotation to $\Delta$ extends to a homeomorphism (see \cite{book:Milnor}, Lemma~18.7), which is still conjugating the rotation to the dynamics.
It follows that the critical point above is not (pre)periodic for any parameter $c\in B(c_0,r_0)$.
Hence its orbit undergoes a holomorphic motion. This motion commutes with the dynamics. This motion extends continously to the closure of the critical orbit by the $\lambda$-lemma (see~\cite{MSS}), and the extension still commutes with the dynamics.
By the conjugacy above, the critical point has a dense orbit in $\partial \Delta$ hence the closure \emph{is} $\partial \Delta$.
Then by \Cref{prop:sullivan} the function $c\mapsto \log r(P_{\lambda,c})$ is harmonic on $B(c_0,r_0)$.

\subsubsection{\texorpdfstring{$Z_\lambda\subset \Supp\mu_\lambda$}{Z lambda is contained in Supp mu lambda}}

Let $c_0\in\C^*$ such that $c_0\notin \Supp\mu$. We will prove that $c_0\notin Z_{\lambda}$. Let start by proving the following assertion, which is valid for all Brjuno rotation number $\theta$ (not only bounded type ones). Still denoting $\lambda = e^{2\pi i\theta}$:

\begin{assertion}
In the complement of $\Supp\mu_\lambda$, the Siegel disk undergoes a holomorphic motion.
\end{assertion}

It follows from the main theorem in \cite{Zakeri2} but we include a proof here for completeness. It follows from the following more general version, whose proof was communicated to us by \'Avila:

\begin{proposition}\label{prop:holomo}
Let $\lambda$ be a complex number of modulus $1$.
Let $(f_a)_{a \in W} : (U, 0) \longrightarrow (\C, 0)$ be a one parameter family of holomorphic maps with $f_a(z) = \lambda z + \cal O(z^2)$. Assume they all have a Siegel disk $\Delta_a$ and assume that $a \mapsto \log r( \Delta_a )$ is harmonic.
Then, $\ov\Delta_a$ undergoes a holomorphic motion w.r.t.\ $a\in W$ that commutes with the dynamics and which is holomorphic in $(a,z)$ in the interior of the Siegel disk.
\end{proposition}
\begin{proof}
Denote $r_a = r(\Delta_a)$.
Let $\psi_a : B(0,r_a) \maps \Delta_a$ be the linearizing parametrization, normalized by $\psi_a(z) = z + \cal O(z^2)$ near $0$.
Since $a\mapsto\log r_a$ is harmonic in $W$, there exists a holomorphic map $h:W \maps \C$ such that: $\log r_a = h(a)$. Let $g=\exp \circ h$, so that
$|g(a)| = r_a $.
Now define on $W \times \Delta_{a_0}$
\[
	\zeta(a,z) = \psi_a \left( g(a) \times \psi_{a_0}^{-1}(z) \right),
\]
$\zeta$ is a holomorphic motion, which hence extends to the boundary of the maximal linearization domain by the $\lambda$-lemma.
This motion satisfies $\zeta(a,f_{a_0}(z)) = f_{a}(\zeta(a,z))$ when $z\in \Delta_{a_0}$, and this relation extends by continuity to all $z\in\partial \Delta_{a_0}$.
\end{proof}
\noindent The proposition above implies the assertion.

\medskip

We still assume here that $\theta\in \cal B$ and consider some $c_0\notin\Supp\mu_\lambda$ where $\lambda = e^{2\pi i\theta}$. Since the set $\Supp\mu_\lambda$ is closed, there exists a ball $B(c_0,r_0)$ that is disjoint from $\Supp\mu_\lambda$.
Let 
\[\zeta_c(z)=\zeta(c,z)\]
be the holomorphic motion given by \Cref{prop:holomo}, based on parameter $c_0$, i.e.\ with $\zeta_{c_0}(z)=z$.
Consider any critical point $\cri(c_0) \in \partial \Delta(P_{\lambda,c_0})$ with either $\cri=\cri_1$ or $\cri=\cri_2$ (there is at least one).
\begin{lemma}\label{lem:crfollows}
  The critical point $\cri(c)$ follows the motion $\zeta$ over $B(c_0,r_0)$, i.e.\ $\zeta_c^{-1}(\cri(c))$ is constant.
\end{lemma}
\begin{proof}
  We will proceed by contradiction and assume that it is not constant. In other words, $\zeta_c^{-1}(\cri(c)) \not\equiv \cri(c_0)$.
  Then the holomorphic functions $c\mapsto \cri(c)$ and $f:c\mapsto \zeta_c(\cri(c_0))$ differ, though they take the same value for $c=c_0$.
  In particular the winding number of $\cri-f$ around $0$ is different from $0$, as $c$ loops through the circle of centre $c_0$ and radius $r_0/2$.
  Consider now a point $z_0\in\Delta(c_0)$ that is very close to $\cri(c_0)$.
  Then the function $g:c\mapsto\zeta_c(z_0)$ has to be close to the function $f$.
  If it is close enough, then the winding number of $\cri-g$ around $0$ will be the same as that of $\cri-f$, as $c$ loops through the same circle as above.
  Hence $\cri-g$ must vanish, say at some parameter $c_2$. Then $\cri(c_2)$ belongs to $\Delta(c_2)$, which is a contradiction.
\end{proof}
\noindent So any critical point that is on $\partial\Delta(P_{\lambda,c_0})$ remains on $\partial \Delta$ when $c$ varies in $B(c_0,r_0)$.

\medskip

Now restrict to the case where $\theta$ has bounded type. If both critical points were on $\partial\Delta(P_{\lambda,c_0})$ then this would be so over $B(c_0,r_0)$, i.e.\ we would have
$B(c_0,r_0)\subset Z_{\lambda}$.
But $Z_\lambda$ has no interior: it is a Jordan curve.
Hence there can be only one critical point on $\partial\Delta(P_{\lambda,c_0})$, i.e.\ $c_0\notin Z_\lambda$.

\section{Parabolic slices}\label{sec:parabo}

Here we assume that $\lambda = e^{2\pi i\theta}$ with $\theta =p/q$ a rational number, written in lowest terms.

\subsection{Introduction}

I defined in my thesis \cite{thesis:Cheritat} the \emph{asymptotic size} of a parabolic point and proposed it as an analogue of the conformal radii of Siegel disks.

\begin{definition}\label{def:asslen}
  Orbits attracted by petals of a fixed non-linearizable parabolic point $p$ of a holomorphic map $f$ satisfy
  \[|f^n(z)-p| \sim \frac{L}{n^{1/r}} \]
  for some constant $L>0$ called the asymptotic size of $p$ and some $r\in\N^*$ which coincides with the number of attracting petals.
\end{definition}

One can compute $L$ from the asymptotic expansion of an iterate of $f$, provided it is tangent to the identity at $p$:
If
\[f^k(p+z) = p + z + C z^{m+1} + \cal O(z^{q+2})\]
then
\[r = m\]
and
\[L = \left|\frac{k}{mC}\right|^{1/m}.\]

Under a conjugacy $g = h \circ f \circ h^{-1}$, the factor $L$ scales as a length:
\[L(g) = L(f) \times |h'(p)|.\]
Correspondingly, the factor $C$ scales as follows:
\[g(p'+z) = p'+z+ \frac{C}{h'(p)^m} z^{m+1} + \cal O(z^{m+2})\]
where $p'=h(p)$.

In the case of the family $P_{\lambda,c}$, the fixed point $z=0$ has a multiplier that is a primitive $q$-th root of unity and it is well-known in this case that the number of petals must be a multiple of $q$, because $P$ permutes the petals in groups of $q$.
Moreover, each cycle of petals attracts a critical point by Fatou's theorems, hence there is at most two cycles of petals, so it follows that
\[m=q\text{ or }m=2q.\]
In fact, one can develop
\[P_{\lambda,c}^q(z) = z + C_{p/q}(c) z^{q+1} + \cal O(z^{q+2})\]
and 
\[m=2q \iff C_{p/q}(c)=0.\]
From now on we abbreviate with
\[C:=C_{p/q}\]
to improve readability.
Now recall that
\[ P_{\lambda,c}(z) = \lambda z \left( 1 - \frac{(1 + \sfrac{1}{c})}{2} z + \frac{\sfrac{1}{c}}{3} z^2 \right).\]
To take advantage of the tools of algebra, it makes sense to use $1/c$ as a variable, so let us call
\[u = \frac{1}{c}.\]
So that
\[ P_{\lambda,c}(z) = \lambda z \left( 1 - \frac{(1 + u)}{2} z + \frac{u}{3} z^2 \right).\]
It follows that $C(c)$ is a polynomial in $u$. We thus define
\[\check C(u) = C(c) = C(1/u).\]
A polynomial $P(z)=\sum a_n z^n$ of degree $d$ is called \emph{symmetric} when its coefficients satisfy $a_{d-k}=a_k$ for all $k$.
\begin{lemma}\label{lem:degC}
  The polynomial $\check C$ has degree $q$ and is symmetric.
\end{lemma}
\begin{proof}
  When $u$ tends to infinity then $c$ tends to $0$, and the map
  $f_c(z) = c^{-1}P_{\lambda,c}(cz)$ converges uniformly on every compact subset of $\C$ to $Q(z)=\lambda z (1-\frac{z}2)$. Let $Q^q(z) = z + C_0 z^{q+1}+\cal O(z^{q+2})$. Then $C_0\neq 0$ because $Q$ can have only one cycle of petals by Fatou's theorem, since it has only one critical point.
  Moreover, the scaling factor $c$ implies that we have
  $f_c^q(z) = z + c^{q}C(c) z^{q+1} + \cdots$.
  Hence $u^{-q}C(1/u)$ has a non-zero limit as $u\tend\infty$, because the limit is $C_0$.
  So the degree is $q$.
  
  The symmetry comes from the symmetry $c^{-1}P_{\lambda,c}(cz) = P_{\lambda,1/c} (z)$ of our family. By the scaling law, $c^{q} C_{p/q}(c) = C_{p/q}(1/c)$, i.e.\
  \[u^{q} C_{p/q}(1/u) =  C_{p/q}(u)\]
  which is another way to express that a degree $q$ polynomial is symmetric.
\end{proof}

In the proof above we saw that the leading coefficient of $\check C$ is the coefficient $C(Q)$ in 
\[Q^q(z) = z + C(Q) z^{q+1}+\cal O(z^{q+2})\]
where
\[Q(z)=\lambda z (1-\frac{z}2).\]
In particular, if we denote $u_i$ the roots of $\check C$, counted with multiplicity, we get
\begin{equation}\label{eq:cCprod}
  \check C(u) = C(Q) \prod_{i=1}^q (u-u_i)
\end{equation}

Consider the function
\[r(c)=\frac{1}{|C(c)|^{1/q}}.\]
This quantity is equal to the asymptotic size $L$ of $0$ for $P_{\lambda,c}$ except when $C(c)=0$, where $r(c)=+\infty$.
The quantity $-\log r$ computes to
\[-\log r(c) = - \frac{1}{q} \log |C(c)|.\]
The function $-\log r$ is then harmonic in $\C^*$ minus the set of zeroes of $C$.
By convenience, we set $\log +\infty = +\infty$, and $-\log r$ is then subharmonic in $\C^*$.
From \cref{eq:cCprod} we get
\begin{equation}\label{eq:logL}
-\log r(c) = -\log L(Q)-\log|c| + \frac{1}{q} \sum \log |c-c_i|
\end{equation}
where $L(Q)$ denotes the asymptotic size of $Q$ at $0$.
The generalized Laplacian applied to $-\log r$ on $\C^*$ is a finite sum of dirac masses that we call $\mu_{\lambda}$ (recall that $\lambda = e^{2\pi i \pq}$):
\[\mu_{\lambda} = \frac{2\pi}q \sum_{i=1}^q \delta_{c_i}\]
where $c_i$ are the roots of $C$ counted with multiplicity. 

\begin{remark*}
In fact 
the roots of $\check C$ are simple: it should follow more or less immediately from Proposition~4.6 in \cite{Bo}.
Note also that an analogue statement was proved in \cite{art:BEE} for the family of degree $2$ rational maps by a transversality arguments using quadratic differentials. We expect that \cite{art:BEE} adapts to the present situation with little modification.
\end{remark*}

By the symmetry of the family, i.e.\ \cref{eq:sym}, 
the following analogue of \cref{eq:symr} holds:
\begin{equation}\label{eq:symr2}
  -\log r(c) + \log |c|= -\log r(1/c)
\end{equation}
We also stress one consequence of the above computations, an analogue of \cref{eq:limr} on page~\pageref{eq:limr}:
\begin{equation}\label{eq:limr2}
  r(c) \underset{c\to\infty}\tend r(Q).
\end{equation}

To finish this section, let us note that

\begin{lemma}\label{lem:ka}
  There exist some $\kappa>0$ such that for all $\pq$, the roots of $\check C$ are contained in $\ov B(0,\kappa)$. The support of all the corresponding measures $\mu_{\lambda}$ are thus contained in $\ov B(0,\kappa)$.
\end{lemma}
\begin{proof} 
  If one critical point escapes to infinity, there can be only one cycle of petals by Fatou's theorem.
  For $c$ to escape, by the proof of \Cref{lem:cesc} and since $|\lambda|=1$, it is enough that $|c|>\kappa$ with
  $\kappa>3$ and $\frac{(\kappa-3)\kappa}{6} > \max\big(\sqrt{6\kappa},6(1+\kappa),\sqrt{12\kappa}\big)$.
  For instance, $\kappa=40$ is enough.
\end{proof}

\section{Limits of measures}\label{sec:7}

A note on terminology : by measures we will always mean \emph{positive} measures. If we ever need other kind of measures, we will call them signed measures or complex measures.

\subsection{Statement} 

How does $\mu_\lambda$ does depend on $\lambda$?

\begin{conjecture}[Buff]
  Let $\U_X = \{\,\exp(2i\pi x)\,;x\in X\,\}$ and let $\cal B$ denote the set of Brjuno numbers.
  The function $\lambda\in\D\cup\U_\Q\cup\U_{\cal B}\mapsto \mu_\lambda$ has a continuous extension to $\ov{\D}$ for the weak-$\ast$ topology on measures.
\end{conjecture}

This states several things: that $\mu_\lambda$ depends continuously on $\lambda$, even at parameters for which $0$ is neutral with rational or Brjuno rotation number, but also that is has a limit at non-Brjuno irrational rotation numbers.

Here we prove a special case of this conjecture:

\begin{theorem}\label{thm:main}
  Let $\theta$ be a bounded type irrational and $p_n/q_n$ its approximants. Then $\mu_{e^{2\pi i\pqn}}\tend \mu_{e^{2\pi i\theta}}$ for the weak-$\ast$ topology.
\end{theorem}

We recall that the locally finite Borel measures on $\R^n$ are called the Radon measures and are in natural bijection with positive linear functionals on the space of continuous real-valued function with compact support $C^0_c(\R^n)$ (no need to endow this space with a topology, thanks to positivity).
Weak-$\ast$ convergence $\mu_n\tend \mu$ for locally finite Borel measures on $\R^n$ means that for all continuous function $\phi:\C\to\R$ (dubbed test functions) with compact support,
\[\int \phi \mu_{n} \tend \int \phi \mu.\]

\subsection{Generalities}

The proof of \Cref{thm:main} will use potential theory and we will recall a few generalities about this and other things.

We first recall a classic fact: the space of Borel measures supported on $\ov B(0,R)$ and of mass $\leq M$ is compact for the weak-$\ast$ topology, which is metrizable. In particular for a sequence of such measures to converge, it is enough to prove the uniqueness of extracted limits.
  
The \emph{potential} of a finite Borel measure $\mu$ with compact support in the plane is
defined by
\[ u(z) = \int \frac{1}{2\pi} \log|z-w| d\mu(w)\]
It is a subharmonic function and satisfies 
\[\Delta u = \mu\]
in the sense of distributions and better:
\[\int (\Delta \phi) u = \int \phi \mu \text{ for all $\phi\in C^0_c(\C)$}\]
i.e.\ the test functions can be taken $C^0$ instead of $C^\infty$.

In the rest of this section on generalities, we will carefully avoid the language of distributions.  \Cref{sub:pfthmmain}. 

\begin{remark*}
Note the difference of convention with \cite{Ra}: there, we have $u(z) = \int \log|z-w| d\mu(w)$ and $\Delta u = 2\pi\mu$.
\end{remark*}

\begin{lemma}
We have the following expansion:
\begin{equation}\label{eq:expanu}
  u(z) = \frac{\on{mass} \mu}{2\pi}\log |z| + 0 + o(1).
\end{equation}
\end{lemma}
\begin{proof} This is well-know, we recall a proof here. When $z$ is not in the support then:\\
$\ds\int \frac{1}{2\pi} \log|z-w| d\mu(w) = \int \frac{1}{2\pi} \left(\log|z| +\log \Big| 1-\frac wz \Big|\right) d\mu(w) = \frac{\on{mass}\mu}{2\pi} \log|z| +$\\
$\ds \frac{1}{2\pi}\int \log \Big| 1-\frac wz \Big| d\mu(w)$.
Now if the support of $\mu$ is contained in $B(0,R)$ and $|z|>2R$ then for $w$ in the support,
$\big|\log |1-w/z|\big| \leq C|w/z|$ with $C>0$ some constant, hence\\
$\ds \left|\int \log \Big| 1-\frac wz \Big| d\mu(w) \right|\leq \frac{C'}{|z|}$ with $C' = C\int |w| d\mu(w)$.
\end{proof}

Let
\[ \ell(z) = \frac{1}{2\pi}\log|z|\]
This map is locally $L^1$, so for $\phi\in\C^\infty_c(\C)$, the convolution $\phi * \ell$ is 
also $C^{\infty}$. However, it does not necessarily have compact support.
The potential of a (finite with compact support) Borel measure defined above is just
\[ u = \mu * \ell\]
where $*$ refers to the convolution operator. 
Note that $\mu * \ell \in \Luloc(\C)$, i.e.\ $\mu * \ell$ is locally integrable w.r.t.\ the Lebesgue measure (this can be deduced from the first part of the statement below, but this is also a classical fact for subharmonic functions, see the paragraph after remark following the lemma).

We will need the following classical lemma:

\begin{lemma}\label{lem:fub}
  For a finite measure with compact support $\mu$ and a continuous test function $\phi$ with compact support, then $\phi\times (\mu * \ell)$ is integrable with respect to the Lebesgue measure and:
  \[\int \phi\times (\mu * \ell) = \int (\phi * \ell) \times \mu.\]
\end{lemma}
(Note: the left hand side is to be understood as an integral over the Lebesgue measure 
of the integrable function: $\phi\times (\mu * \ell)$;
the right hand side as the integral of the continuous function $(\phi * \ell)$ over the measure $\mu$. 
The function $\phi * \ell$ is continous since $\phi$ is continuous with compact support and $\ell\in\Luloc(\C)$.)

\begin{proof}
  This is a simple application of the Fubini theorems, but we will check it carefully.
  Let $\lambda$ be the Lebesgue measure.
  Consider the product measure $\lambda\times\mu$ and the measurable function $(z,w)\mapsto \phi(z)\ell(z-w)$. 
  Fixing $w$, the integral of its absolute value with respect to $d\lambda(z)$ is $\int_\R |\phi(z)\ell(z-w)| d\lambda(z) = \int_{\Supp \phi} |\phi(z)\ell(z-w)| d\lambda(z)$, i.e.\ we can restrict $z$ to the support of $\phi$. Using  $v=z-w$ we get 
  $\int_{\Supp \phi} |\phi(z)\ell(z-w)| d\lambda(z) \leq \big(\!\max|\phi| \big)\int_{v\in -w+\Supp\phi} |\ell(v)|d\lambda(v)$.
  Decompose 
  \[\ell(v) = \log |v| = \max(0,\log |v|) + \min(0,\log|v|) = \ell^+(v) - \ell^-(v).\]
  We have
  \[  |\ell(v)|  = \ell^-(v) + \ell^+(v) \]
  so, denoting $S_w = -w+\Supp\phi$:
  \begin{eqnarray*}
  \int_{v\in S_w} |\ell(v)| & = & \int_{S_w} \ell^-(v) + \int_{S_w} \ell^+(v)
  \\ 
  & \leq &  \int_{\C} \ell^-  +  \lambda(\Supp\phi) \max_{v\in S_w} \ell^+(v) 
  \end{eqnarray*}
  One computes $\int_{\C} \ell^- = \pi/2 < +\infty$. Let $R>0$ so that $\Supp\phi\subset B(0,R)$ and $R'>0$ so that $\Supp \mu \subset B(0,R')$.
  Then $\forall w \in\Supp\mu$, $\forall v\in S_w$:
  \[\ell^+(v) \leq \max\big (0,\log(|w|+R)\big) \leq \max\big (0,\log(R'+R)\big).\]  
  Summing up: there is some $K>0$ such that for all $w$ in the support of $\mu$:
  \[\int_\C |\phi(z)\ell(z-w)|d\lambda(z) \leq K.\]
  Since $\mu$ has finite mass it follows that
  \[\int_{\C\times\C} |\phi(z)\ell(z-w)|d\lambda(z) d\mu(w)<+\infty.\]
  So we can apply the Fubini-Tonelli theorem, from which follows that the function $z\mapsto \phi(z) \left(\int_\C \ell(z-w) d\mu(w)\right)$ is $L^1$ with respect to the Lebesgue measure and that we have
  \[\int_\C \phi(z) \left(\int_\C \ell(z-w) d\mu(w)\right)d\lambda(z) = 
    \int_\C \left(\int_\C \phi(z) \ell(z-w) d\lambda(z)\right)  d\mu(w)
  .\]
  The left hand side is equal to $\int \phi\times (\mu * \ell)$ 
  and using $\ell(z-w) = \ell(w-z)$ the right hand side is equal to $\int (\phi * \ell) \times \mu$.
\end{proof}
 
\begin{remark*}
  This would be tempting to write the conclusion of the previous lemma as
  \[ \langle \mu * \ell, \phi\rangle = \langle \mu, \phi * \ell\rangle\]
  except that as a test function, $\phi * \ell$ is indeed smooth but does not necessarily have compact support, so we prefer to avoid this notation.
\end{remark*}

We recall that a subharmonic function on a connected open subset $X$ of $\R^n$
that is not $\equiv -\infty$ is in $\Luloc(X)$, see theorem~2.5.1 in \cite{Ra} or corollary~3.2.8 in \cite{Ho}. Also, a subharmonic function is upper semi-continuous and takes values in $[-\infty,+\infty)$, hence it is always locally bounded from above.

\subsection{Proof of Theorem~\ref{thm:main}}\label{sub:pfthmmain}

Let us denote by
\[u_{\theta} = \text{the potential of $\mu_\lambda$} = \mu_\lambda * \ell\]
where $\lambda = e^{2\pi i \theta}$ for $\theta$ a Brjuno number or a rational number.
Recall that when $\theta$ is a Brjuno number, we call $r(c)$ the conformal radius of the Siegel disk and when $\theta$ is rational, we set $r(c) =L$, the asymptotic size of the parabolic point when it has $q$ petals, and $r(c)=+\infty$ when it has $2q$ petals. We use $\log +\infty = +\infty$ as a convenience when speaking of $\log r$.
In the rational case $\mu_\lambda$ is a sum of Dirac masses: $\mu_\lambda = \sum_{i=1}^{q} \frac{2\pi}{q} \delta_{c_i}$.
It follows that $u_\pq(c) = \frac{1}{q}\sum \log |c-c_i|$.

In the rational or Brjuno case, from \cref{eq:logr,eq:logL} it follows that, for $c\neq 0$:
\begin{equation}\label{eq:pot}
u_\theta(c) = -\log r(P_{\lambda,c}) + \log|c| + \log r(Q_\lambda).
\end{equation}
From \cref{eq:symr,eq:symr2} we have
\[  -\log r(P_{\lambda,1/c}) = -\log r(P_{\lambda,c}) + \log |c|.
\]
and since by \cref{eq:pot} we have
\[ -\log r(P_{\lambda,c})  = u_\theta(c) - \log|c| - \log r(Q_\lambda)\]
we get $\forall c\in\C^*$:
\begin{equation}
  u_\theta(c) = -\log r(P_{\lambda,1/c}) + \log r(Q_\lambda)
\end{equation}
which will interest us when $c\tend 0$ : indeed by \cref{eq:limr,eq:limr2}
\begin{corollary}\label{cor:zero}
  For $\theta$ a Brjuno number or a rational number:
  \[u_\theta(0) = 0\]
\end{corollary}
 
By the symmetry formulas \cref{eq:symr,eq:symr2} we have, for all $\theta$ Brjuno or rational:
\begin{equation}\label{eq:symu}
  u_\theta(c) = u_\theta(1/c) + \log |c|
\end{equation}

To alleviate the notations and in particular nested subscript and superscripts, which tend to be hard to read, we first set
\[
\setlength{\arraycolsep}{3pt}
\begin{array}{rclrclrclrclrcl}
  P[\lambda,c] &=& P_{\lambda,c}
  &
  Q[\lambda] &=& Q_{\lambda}  
  &
  u[\lambda] &=& u_{\lambda}  
  &
  \mu[\lambda] &=& \mu_{\lambda}
  &
  r[\lambda](c) &=& r(P_{\lambda,c})
\end{array}
\]
and then we will use the following abuse of notations:
\[
\setlength{\arraycolsep}{3pt}
\begin{array}{rclrclrclrcl}
P_{n,c} &=& P[e^{2\pi i\pqn},c]
&
Q_{n} &=& Q[e^{2\pi i\pqn}]
&
u_n &=& u[e^{2\pi i\pqn}]
&
\mu_n &=& \mu[e^{2\pi i\pqn}]
\\
P_{\theta,c} &=& P[e^{2\pi  i\theta},c]
&
Q_{\theta} &=& Q[e^{2\pi  i\theta}]
&
u_\theta &=& u[e^{2\pi i \theta}]
&
\mu_\theta &=& \mu[e^{2\pi i \theta}]
\end{array}
\]
and denote
\[
\begin{array}{rcl}
r_n(c) &=& r[e^{2\pi i\pqn}](c)
\\
r_\theta(c) &=& r[e^{2\pi i\theta}](c).
\end{array}
\]

One key point is the following bound.
\begin{proposition}\label{prop:maj}
  \[\limsup_{n\to\infty} \sup_\C (u_n-u_\theta) \leq 0.\]
\end{proposition}
\noindent Its proof is based on a study in \cite{thesis:Cheritat} and is postponed to \Cref{sub:pfmaj}.

\medskip

We then prove a partial result in the more general case when $\theta$ is a Brjuno number:
\begin{proposition}\label{prop:super}
Let $\theta$ be a Brjuno number, $\lambda=e^{2\pi i \theta}$ and let $W_\theta$ be the component containing a neighborhood of $0$ of the complement of the support of $\mu_{\lambda}$. Then
$u_{n} \tend u_\theta$ in $\Luloc (W_\theta)$.
\end{proposition}
\noindent Note that in this proposition, the convergence is only claimed on $W_\theta$. However, by the symmetry formula \cref{eq:symu}, the convergence also occurs on $1/W_\theta := \{\,1/z\,;\,z\in W_\theta\text{ and }z\neq 0\,\}$.
In \Cref{sub:pfsuper} we deduce \Cref{prop:super} from \Cref{prop:maj}.
In fact, we will only need a weaker result than \Cref{prop:maj}: that
\[\limsup_{n\to\infty} \sup_K (u_n-u_\theta) \leq 0\]
for all compact subset $K$ of the set $W_\theta$; but our method gives the stronger \Cref{prop:maj}.

\begin{remark*}
We do not need the following fact but find it interesting: $1/W_\theta$ is disjoint from $W_\theta$: otherwise $W_\theta$ would be a neighborhood of $0$ and of $\infty$ and one of the critical point would remain on $\partial \Delta(P_{\theta,c})$ by \Cref{lem:crfollows}, which contradicts \Cref{lem:cesc}).
\end{remark*}

Let us denote
\[D_n \stend D\]
the weak-$\ast$ convergence.

We will then complete the job by the following:
\begin{corollary}\label{cor:c2}
If $\theta$ is a Brjuno number and $\C= \ov{ W_\theta}\cup \ov{1/W_\theta}$  then
$\mu_n  \stend \mu_\theta$.
\end{corollary}
\noindent In \Cref{sub:pfc2} we deduce \Cref{cor:c2} from \Cref{prop:super}. 

\medskip

\Cref{cor:c2} applies in particular to the case when $\theta$ has bounded type since by \Cref{thm:bddType} we know that the support is a Jordan curve. This proves \Cref{thm:main}.

\medskip

In the subsequent sections, we prove the three statements above.

\subsection{Proof of Corollary~\ref{cor:c2} from Proposition~\ref{prop:super}}\label{sub:pfc2}

Recall that the measures $\mu_n$ and $\mu_\theta$ all have total mass $2\pi$ and support in a common ball $\ov B(0,R)$ for some $R>0$ by \Cref{lem:qlr,prop:harmo,lem:ka}.
By weak-$\ast$ compactness of the set $E$ of Borel measures of mass $2\pi$ on $\ov B(0,R)$, it is enough to prove that for all subsequence of $\mu_n$ that has a weak-$\ast$ limit $\mu\in E$, then $\mu=\mu_\theta$.

Recall that $u_n = \mu_n * \ell$ and $u_\theta = \mu_\theta * \ell$.
Let
\[u = \mu * \ell.\]
Then $u$ is a subharmonic function.

\begin{lemma}\label{lem:eqW} Let $\lambda$ be the Lebesgue measure:
  \[u_n\lambda \stend u\lambda\]
\end{lemma}
\begin{proof}
  For all continuous test function $\phi$, the function $\phi * \ell$ is continuous (but not necessarily with compact support) and by \Cref{lem:fub}:
  \[\int (\phi * \ell) \times \mu_n = \int \phi \times (\ell * \mu_n) =\int \phi u_n\lambda\]
  and 
  \[\int (\phi * \ell) \times \mu = \int \phi \times (\ell * \mu) = \int \phi u\lambda.\]
  Now $\int (\phi * \ell) \times \mu_n \tend \int (\phi * \ell) \times \mu$ by definition of weak-$\ast$ convergence of $\mu_n$ to $\mu$.
\end{proof}

By \Cref{prop:super}, $u_n\tend u_\theta$ in $\Luloc(W_\theta)$.
By the symmetry relations \cref{eq:symu}, this also holds on $1/W_\theta$.
This implies the following weaker statement: $u_{n}\lambda \stend u_\theta\lambda$ on $W_\theta\cup 1/ W_\theta$.

It follows that $u\lambda=u_\theta\lambda$ on $W_\theta\cup 1/ W_\theta$, so $u=u_\theta$
almost everywhere on $W_\theta\cup 1/ W_\theta$. Since both functions are subharmonic, this implies (see theorem~2.7.5 in \cite{Ra})
\[u = u_\theta \text{ on $W_\theta\cup 1/ W_\theta$.}\]

To extend this equality to $\C$, we will use the notion of non-thin sets, see \cite{Ra}.

\begin{definition}[Def.\ 3.8.1 page~79 in \cite{Ra}]
A subset $S$ of $\C$ is \emph{non-thin} at $\zeta\in \C$, if $\zeta \in \ov{S\setminus\{\zeta\}}$ and if for all subharmonic function defined in a neighborhood of $\zeta$, 
  \[\limsup_{\substack{z\to\zeta \\ z\in S\setminus\{\zeta\}}} u(z) = u(\zeta).\]
\end{definition}

\begin{theorem}[Thm.\ 3.8.3 page~79 of \cite{Ra}]
  A connected set containing more than one point is non-thin at every point of its closure.
\end{theorem}

By hypothesis, $\C=\ov{W_\theta}\cup \ov{1/W_\theta}$.
We already know that $u=u_\theta$ on $W_\theta$ and on $1/W_\theta$.
Now for all $\zeta \in \ov{W_\theta}$, 
\[ u(\zeta) = \limsup_{\substack{z \to \zeta \\ z \in W_\theta}} u(z) = \limsup_{\substack{z \to \zeta \\ z \in W_\theta}} u_\theta(z) = u_\theta(\zeta)\]
and similarly with $1/W_\theta$ in place of $W_\theta$.

Hence $u=u_\theta$ on $\C$, so their generalized Laplacians are equal: $\mu = \mu_\theta$.
This ends the proof of \Cref{cor:c2}.

\subsection{Proof of Proposition~\ref{prop:super} from Proposition~\ref{prop:maj}}\label{sub:pfsuper}

We will use a nice trick suggested by Xavier Buff.

The subharmonic function $u_\theta$ is harmonic on $W_\theta$, hence $u_n-u_\theta$ is subharmonic on $W_\theta$.
Moreover by \Cref{cor:zero}, $u_n(0) = 0 = u_\theta(0)$ so $u_n-u_\theta$ vanishes at the origin.

By \Cref{prop:maj}, the \Cref{prop:super} will be a consequence of the following proposition applied to the functions $f_n = u_n - u_\theta$ on $X=W_\theta$ and with $x_0=0$.

\begin{proposition}\label{prop:upper2}
Assume that $f_n$ is a sequence of subharmonic functions on a connected open subset $X$ of $\C$ and assume that
\[\exists x_0\in X\text{ such that } f_n(x_0) \tend 0 \]
and that for every compact subset  $K$ of $X$,
\[ \limsup_{n\to\infty}\ (\sup_K f_n) \leq 0\]
Then $f_n \tend 0$ in $\Luloc(X)$.
\end{proposition}
\begin{proof}
  For $n$ big enough we have $f_n(x_0)\neq -\infty$, so $f_n \not\equiv - \infty$, so it is in $\Luloc(X)$.
  Consider the set $\cal A$ of points of $X$ which have an open ball neighborhood $B$ on which $\int_B |f_n|\tend 0$. Then we claim that relative to $X$, the set $\cal A$ is open, closed and non-empty.
  
  Open is immediate.
  
  Closed follows from the following argument: Let $x\in X$ such that $x$ is in the closure of $\cal A$. Let $B=B(x,\eps)$ be compactly contained in $X$.
  By hypothesis, there is s sequence $M_n\in\R$ such that $M_n\to 0$ and such that for all $n$,
  \[f_n \leq M_n\]
  on $B$. We have $M_n-f_n\geq0$ and $|f_n| = |-M_n+M_n-f_n| \leq |M_n| + M_n-f_n$, so it is enough to prove that $\int_B(M_n-f_n) \tend 0$.
  For $y\in B$, let
  \[\phi_y (z) = x + \frac{(z-x)+(y-x)}{1 + \ov{(y-x)}(z-x)/\eps^2}\]
  which is a conformal automorphism of $B$ mapping $x$ to $y$.
  Since $f_n\circ \phi_y$ is also subharmonic (by corollary~2.4.3 in \cite{Ra}, subharmonicity is invariant by an analytic change of variable in the domain), we have
  \begin{equation}\label{eq:e6}
    f_n(y) \lambda(B) \leq \int_B f_n\circ\phi_y(z) d\lambda(z)
  \end{equation}
  Since $x\in \ov{\cal A}$, there is some $x'\in B$ such that $x'\in\cal A$, hence there is some open ball $B'\subset X$ containing $x'$ and a sequence $M'_n>0$ such that $M'_n\tend 0$ and such that $\forall n$,
  \[\int_{B'} |f_n|\leq M'_n.\]
  This is a fortiori true if we replace $B'$ by any open ball contained in $B'$, hence we can assume that $B'\Subset B$.
  For any $f\in L^1(B)$, we have by the change of variable $w=\phi_y(z)$,
  $z=\phi_y^{-1}(w)$:
  \[\int_{B} f(\phi_y(z)) d\lambda(z)
  = \int_{B} f(w) |(\phi_{y}^{-1})'(w)|^2 d\lambda(w)\]
  If we let $y$ vary in $B'\Subset B$ then the complicated term $|(\phi_{y}^{-1})'(w)|$ remains bounded away from $0$ and $\infty$ by constants that depend only on $B$ and $B'$.
  Let $c\in(0,1)$ be a lower bound. 
  Now recall that $f_n\leq M_n$, hence taking $f = M_n-f_n\geq 0$ above,
  \[ \int_{B} (M_n-f_n(\phi_y(z))) d\lambda(z)   
  \geq c^2 \int_{B} (M_n-f_n(w)) d\lambda(w)\]  
  Whence
  \[\int_{B} (M_n-f_n(w)) d\lambda(w) \leq c^{-2}\left(M_n\lambda(B) - \int_B f_n\circ\phi_y\right)\]
  and using \cref{eq:e6} we get
  \[\int_{B} M_n-f_n  \leq c^{-2}(M_n - f_n(y))\lambda(B)\]
  Averaging the above over $y\in B'$ implies
  \begin{equation}\label{eq:e7}
  \frac{1}{\lambda(B')}\int_{B} M_n-f_n \leq c^{-2}\frac{1}{\lambda(B')} \int_{B'} M_n-f_n
  \end{equation}
  Now $\int_{B'} M_n-f_n$, which is non-negative, tends to $0$: indeed, $M_n-f_n\leq |f_n| + |M_n|$ and $\int_{B'} f_n\leq M'_n$.
  By \cref{eq:e7}, it follows that $\int_{B} M-f_n\tend 0$. Q.E.D.

  Non-empty: we prove that $x_0\in \cal A$.
  Let $B(x_0,\eps)$ compactly contained in $X$.
  By hypothesis, there is a sequence $M_n\tend 0$ such that for all $n$, $f_n \leq M_n$ on $B$.
  By subharmonicity, $f_n(x_0) \lambda(B) \leq \int_B f_n$.
 it follows that $\int_B |M_n-f_n| = \int_B M_n-f_n \leq (M_n-f_n(x_0))\lambda(B)$. 
  Hence $\int_B |f_n| \leq \int_B |M_n| + \int_B |M_n-f_n| \leq (|M_n| + M_n - f_n(x_0))\pi \eps^2$.
  
  By connectedness $\cal A=X$.
\end{proof}
This ends the proof of \Cref{prop:super}.

\subsection{Proof of Proposition~\ref{prop:maj}}\label{sub:pfmaj}

To prove it, we will adapt an inequality in \cite{thesis:Cheritat}, for which the following lemma from~\cite{thesis:Jellouli}, was crucial.
\begin{lemma}[Jellouli]\label{lem:jellouli}
Let:
\begin{itemize}
\item $f_\alpha : (U,0) \longrightarrow (\C,0)$ be a family of holomorphic map defined for $\alpha$ in an open interval $I$ on a common domain $U$ and fixing the origin with multiplier $\o{\alpha}$,
\item $\alpha_\ast\in I$ be irrational,
\item $\alpha_n = \pqn$ be the convergents of $\alpha_\ast$.
\end{itemize}
We assume that $\forall \alpha,\beta\in I$, $\|f_\alpha-f_\beta\|_\infty \leq \Lambda |\alpha-\beta|$ and that $f_{\alpha_\ast}$ is linearizable at $0$.
Then $f_{\alpha_n}^{q_n} \longrightarrow \id$ uniformly on every compact subsets of $\Delta_{\alpha_\ast}$.
\end{lemma}

Its proof goes by conjugating $f_\alpha$ by the normalized linearizing map $\phi_{\alpha_\ast}$ of $f_{\alpha_\ast}$, this gives $F_\alpha = \phi_{\alpha_\ast} \circ f_\alpha\circ\phi_{\alpha_\ast}^{-1}$, and proving the following two things by induction: let \[ r_\ast = r(\Delta(f_{\alpha_\ast}))\]
for every compact $K \subset r_\ast \D$, there exists $N=N(K) \in \N$ and $C=C(K) > 0$ such that $\forall n \geq N$, the $q_n$ first iterates of $F_{\alpha_n}$ are defined on $K$ and do not leave $r(\Delta_{\alpha}) \D$ and
$\forall z \in K, \, \forall n \geq N, \, \forall k \in \N$:
\[ k \leq q_n \implies |F_{\alpha_n}^k (z) - e^{2 i \pi k \pqn} z | \leq \frac{C |z| k}{q_n^2} \]
We will take advantage of one consequence of this: let us write the expansions
\[f_{\alpha_n}^{q_n}(z) = z + C_n z^{q_n+1} + \cal O(z^{q_n+2})\]
Then
\begin{lemma}\label{lem:l1}
Under the same asumptions, 
\[\limsup |C_n|^{1/q_n} \leq 1/r_\ast\]
\end{lemma}
\begin{proof}
  Since $\phi_{\alpha_*}'(0)=1$, it follows that
  \[F_{\alpha_n}^{q_n}(z) = z + C_n z^{q_n+1} + \cal O(z^{q_n+2})\]
  for the same $C_n$ as for $f_{\alpha_n}$. Consider any $\rho<r_\ast$.
  For $n$ big enough, $F_{\alpha_n}^{q_n}$ is defined on $B(0,\rho)$.
  It has to take values in $B(0,r_\ast)$.
  By the Cauchy formula,
  \[
    C_n = \frac{1}{i 2\pi} \int_{\partial B(0,\rho)} \frac{F_{\alpha_n}^{q_n}(z)}{z^{q_n+2}} dz
  \]
  consequently
  \[ |C_n| \leq \frac{r_\ast}{\rho^{1+q_n}} \]
  hence
  \[ |C_n|^{1/q_n} \leq \frac{1}{\rho}\left(\frac{r_\ast}{\rho}\right)^{1/q_n}\]
  so
  \[ \limsup |C_n|^{1/q_n}\leq \frac{1}{\rho} .\]
  Since this is valid for all $\rho<r_\ast$, the conclusion follows.
\end{proof}

Now to prove the \Cref{prop:maj}, we will need a form of uniformity of the above computations when the family depends on a supplementary parameter in Jellouli's lemma, so we will dig into its proof and pay attention to uniformity.

Let $\psi_{c}$ be the linearizing parametrization of $P_{\theta,c}$ and $r(c)$
be its radius of convergence. Recall that $\psi_c$ is a holomorphic function defined on $B(0,r(c))$ whose image is the Siegel disk of $P_{e^{2\pi i \theta},c}$.
Recall that I proved im my thesis \cite{thesis:Cheritat} that $r$ is a continuous function of $c$.
Let 
\[f_{n,c}(z) = \psi_c^{-1} \circ P_{n,c} \circ \psi_c.\]
This function is defined in some subset of $B(0,r(c))$.

We will use
\begin{theorem}\label{lem:d0}
For all Schlicht functions $f$,
\[\forall z \in\D,\ d(f(z),\partial f(\D))\geq d_0(|z|) = \frac{1}{4}\left(\frac{1-|z|}{1+|z|}\right)^2\]
where $d$ denotes the Euclidean distance.
\end{theorem}
\begin{proof}
By corollary~1.4 in \cite{Po} we have $d(f(z),\partial f(\D))\geq \frac{1}{4}(1-|z|^2) |f'(z)|$ and by equation~(11) in theorem~1.6 in \cite{Po}, $|f'(z)|\geq \frac{1-|z|}{(1+|z|)^3}$, so $d(f(z),\partial f(\D))\geq \frac{1}{4}\left(\frac{1-|z|}{1+|z|}\right)^2$.
\end{proof}

In the case of our linearizing parametrization $\psi_\theta$, we can apply the above to $z\in\D \mapsto r^{-1}\psi_\theta(r z)$ where $r$ is the conformal radius of $\Delta_\theta=\Delta(P_{\theta,c})$, which yields:
\begin{equation}\label{eq:distB}
  d(\psi_\theta(z),\partial \Delta_\theta) \geq rd_0(r^{-1}|z|).
\end{equation}

\begin{theorem}\label{lem:d1}
For all injective holomorphic $f:\D\to\C$, if $a\in f(\D)$ and $b\in\C$ is such that $|b-a| \leq \frac{1}{2} d(a,\partial f(\D))$ then $b\in f(\D)$ and
\[|f^{-1}(b)-f^{-1}(a)| \leq 2 |b-a|/d(a,\partial f(\D)).\]
\end{theorem} 
\begin{proof} Let $a'=f^{-1}(a)$, $b' = f^{-1}(b)$ and $U=f(\D)$.
We use the Schwarz-Pick hyperbolic metric on $\rho_U(w)|dw|$ on $U$. By classical estimates, $\rho_U(w) \leq \frac{1}{d(w,\partial U)} \leq \frac{2}{d(a,\partial U)}$ if $w\in B(a,\frac12 d(a,\partial U))$, so the straight segment from $a$ to $b$ has hyperbolic length $\leq 2 |b-a| / d(a,\partial U)$.
It follows that the hyperbolic distance in $\D$ from $a'$ to $b'$ is $\leq 2 |b'-a'| /d(a,\partial U)$.
The result then follows from the fact that on $\D$ the hyperbolic distance is greater than the Euclidean distance.
\end{proof}

In the case of our linearizing parametrization $\psi_\theta : r\D \to \Delta_\theta$, this gives
\begin{equation}\label{eq:distB2}
  |\psi_\theta^{-1}(b) - \psi_\theta^{-1}(a)| \leq 2 r |b-a|/d(a,\partial \Delta_\theta)
\end{equation}
under the condition $|b-a| \leq \frac{1}{2} d(a,\partial \Delta_\theta)$.

\begin{remark*}
In the previous two theorems, we do not need the explicit bounds but only the existence of a bound, so a compactness argument could also have given a quick proof of them.
\end{remark*}

We have
\begin{equation}\label{eq:e5}
P_{n,c}-P_{\theta,c} = (e^{2\pi i \pqn}-e^{2\pi i \theta}) \left( 1 - \frac{(1 + \sfrac{1}{c})}{2} z + \frac{\sfrac{1}{c}}{3} z^2 \right)
\end{equation}

We will now prove inequalities for
\[\fbox{$\ds|c|\geq 1$}\]
and deduce later inequalities for $|c|\leq 1$ using the symmetry of the family.

\begin{lemma}
There exists $M>0$ such that for all Brjuno number $\theta$, for all $c\in\C$ with $|c|\geq 1$ then the Siegel disk of $P_{e^{2\pi i \theta},c}$ is contained in $B(0,M)$.
\end{lemma}
\begin{proof}
By \Cref{lem:qlr}, if $|c|>833/45$ then the Siegel disk is contained in $B(0,10)$.
Otherwise by \Cref{lem:cesc}, a trap in the basin of infinity is given by $|z|>R$ with $R = \max(6|c+1|,\sqrt{12|c|})$.
If $1\leq|c|\leq R_1=833/45$, we have $R\leq R_2:=6(R_1+1)$.
\end{proof}

It follows, using \cref{eq:e5} and $|c|\geq 1$ that on $\Delta(P_{\theta,c})$, 
\begin{equation}\label{eq:e4}
|P_{n,c}-P_{\theta,c}| \leq M'\left|\pqn-\theta\right| 
\end{equation}
with $M' = 2\pi \left( 1 + M + \frac{1}{3} M^2 \right)$ which is \emph{independent of $c$}.

Let us denote $r=r_\theta(c)$.
A point $z\in B(0,r)$ will be in the domain of $f_{n,c}$ iff $P_{n,c}(\psi_\theta(z))\in \Delta_\theta$.
With $\theta$ in place of $n$, we have $P_{\theta,c}(\psi_\theta(z)) = \psi_\theta(R_\theta(z)) \in \Delta_\theta$. By \cref{eq:e4},
\[|P_{n,c}(\psi_\theta(z))-P_{\theta,c}(\psi_\theta(z))| \leq M'\left|\pqn-\theta\right|\]
and by \cref{eq:distB},
\[d(\psi_\theta(R_\theta(z)) , \partial\Delta_\theta) \geq r d_0(r^{-1}|z|).\]
For $f_{n,c}(z)$ to be defined, having $|P_{n,c}(\psi_\theta(z)) - \psi_\theta(R_\theta(z)) |\leq d(\psi_\theta(R_\theta(z)) , \partial\Delta_\theta) $ will be enough, hence it is enough that
\[M'\left|\pqn-\theta\right| \leq r d_0(r^{-1}|z|).\]
In particular, the domain of $f_{n,c}$ contains $B(0,(1-\eps) r)$ as soon as
\[M'\left|\pqn-\theta\right| \leq r d_0(1-\eps).\]
Consider any $z\in B(0,r)$. By applying \cref{eq:distB2} to $a = P_{\theta,c}(\psi_\theta(z)) = \psi_\theta(R_\theta(z))$ and $b=P_{n,c}(\psi_\theta(z))$ we get that if 
\[M'|\pqn-\theta| \leq \frac{1}{2} r d_0(r^{-1}|z|)\]
then $M'|\pqn-\theta| \leq \frac{1}{2}d(a,\partial \Delta_\theta)$ and hence we can apply \cref{eq:distB2}:
\begin{align*}
 |f_{n,c}(z)-R_\theta(z)| & =  |\psi_\theta^{-1}(b)-\psi_\theta^{-1}(a)|
 \\
 & \leq 2 r |b-a| / d(a,\partial\Delta_\theta) \leq 2M'|\pqn-\theta|/d_0(r^{-1}|z|).
\end{align*}

To sum up:
\begin{corollary}
For $|c|\geq 1$ and $0<\eps<1$, for all $n$ such that
\begin{equation}\label{eq:cond}
\frac{M'|\pqn-\theta|}{d_0(1-\eps)} \leq \frac{1}{2} r_\theta(c)
\end{equation}
then $D := B(0,(1-\eps) r_\theta(c)) \subset \on{dom}f_{n,c}(z)$ and $\forall z\in D$
\begin{equation}
  |f_{n,c}(z)-R_\theta(z)| \leq 2\frac{M'|\pqn-\theta|}{d_0(1-\eps)}.
\end{equation}
\end{corollary}

\medskip

Recall that by the theory of continued fractions, 
\[\forall n\in\N,\ |\pqn-\theta| \leq \frac{1}{q_n^2}.\]

Consider now the condition
\begin{equation}\label{eq:cond2}
\frac{M'/q_n}{d_0(1-\eps)} \leq  \frac{\eps}{2}r_\theta(c).
\end{equation}
Note that $r_\theta(c)$ has a lower bound on $|c|\geq 1$, since it is continuous w.r.t.\ $c$ and has a limit when $c\tend\infty$ by \cref{eq:limr}.
It follows that for a fixed $\eps$, 
as soon as $n$ is big enough
the condition above will be satisfied for all $c\in\C$ with $|c|\geq 1$.

Let us still denote $r=r_\theta(c)$. Now if \cref{eq:cond2} is satisfied then a fortiori \cref{eq:cond} is satisfied.
Assume now that $0<\eps<1/2$. Then for $n$ big enough as above, we can then prove by induction on $k$ with $0\leq k \leq q_n$ that
for all $z\in B(0,(1-2\eps)r)$ (note the factor $2$ in front of $\eps$), $f_{n,c}^k(z)$ is defined,
\[
  |f_{n,c}^k(z)-R_\theta^k(z)| \leq 2 \frac{M'k/q_n^2}{d_0(1-\eps)}.
\]
and $f_{n,c}^k(z)\in B(0,(1-\eps)r)$.
By a similar computation as before, it follows that 
\[|C_n(c)| \leq \frac{r}{((1-2\eps)r)^{1+q_n}} = \frac{1}{r^{q_n}} \cdot \frac{1}{(1-2\eps)^{1+q_n}}\]
Recall that $r_n(c) = \frac{1}{|C_n(c)|^{1/q_n}}$, hence
\[-\log r_n(c) = \frac{\log|C_n(c)|}{q_n}\]
so
\[ -\log r_n(c) \leq - \log r_\theta(c) + \frac{1+q_n}{q_n} \log\frac{1}{1-2\eps}.\]
Now
\[u_n(c)-u_\theta(c) = -\log r_n(c) + \log r_\theta(c) +\log L(Q_n) - \log r(Q_\theta),\]
hence
\[u_n(c)-u_\theta(c) \leq  \frac{1+q_n}{q_n} \log\frac{1}{1-2\eps} + \log L(Q_n) - \log r(Q_\theta)\]

Now we use the following theorem from \cite{thesis:Cheritat}.
\begin{theorem}[Chéritat]
\[L(Q_n) \tend r(Q_\theta)\]
\end{theorem}

It follows that for $n$ big enough,
\[\limsup_{n\to\infty} \sup_{|c|\geq 1} u_n(c)-u_\theta(c) \leq \log\frac{1}{1-2\eps}\]
Since this is true for all $\eps$:
\[\limsup_{n\to\infty} \sup_{|c|\geq 1} (u_n(c)-u_\theta(c)) \leq 0.\]

The case
\[\fbox{$\ds 0<|c|\leq 1$}\]
immediately follows from 
\[u_n(1/c)-u_\theta(1/c) = u_n(c)-u_\theta(c)\]
which is a consequence of \cref{eq:symu}.

This ends the proof of \Cref{prop:maj}.

\bibliographystyle{alpha} 
\bibliography{biblio}

\end{document}